\documentclass[10pt,letterpaper]{article}

\usepackage{amsfonts}
\usepackage{amsmath}
\usepackage{amsrefs}
\usepackage{amssymb}
\usepackage{amsthm}
\usepackage{array}
\usepackage{bm}

\usepackage{pictex}
\usepackage{fancyhdr}
\usepackage{indentfirst}
\usepackage{ifthen}
\usepackage{accents}

\catcode`!=11 

\newcount\aux%
\newcount\auxa%
\newcount\auxb%
\newcount\m%
\newcount\n%
\newcount\x%
\newcount\y%
\newcount\xl%
\newcount\yl%
\newcount\d%
\newcount\dnm%
\newcount\xa%
\newcount\xb%
\newcount\xmed%
\newcount\xc%
\newcount\xd%
\newcount\ya%
\newcount\yb%
\newcount\ymed%
\newcount\yc%
\newcount\yd
\newcount\expansao%
\newcount\tipografo
\newcount\distanciaobjmor
\newcount\tipoarco
\newif\ifpara%
\newbox\caixa%
\newbox\caixaaux%
\newif\ifnvazia%
\newif\ifvazia%
\newif\ifcompara%
\newif\ifdiferentes%
\newcount\xaux%
\newcount\yaux%
\newcount\guardaauxa%
\newcount\alt%
\newcount\larg%
\newcount\prof%
\newcount\auxqx
\newcount\auxqy
\newif\ifajusta%
\newif\ifajustadist
\def\objPartida{}%
\def\objChegada{}%
\def\objNulo{}%

\def\!vazia{:}

\def\!pilhanvazia#1{\let\arg=#1%
\if:\arg\ \nvaziafalse\vaziatrue \else \nvaziatrue\vaziafalse\fi}

\def\!coloca#1#2{\edef\pilha{#1.#2}}

\def\!guarda(#1)(#2,#3)(#4,#5,#6){\def\id{#1}%
\xaux=#2%
\yaux=#3%
\alt=#4%
\larg=#5%
\prof=#6%
}

\def\!topaux#1.#2:{\!guarda#1}
\def\!topo#1{\expandafter\!topaux#1}

\def\!popaux#1.#2:{\def\pilha{#2:}}
\def\!retira#1{\expandafter\!popaux#1}

\def\!comparaaux#1#2{\let\argA=#1\let\argB=#2%
\ifx\argA\argB\comparatrue\diferentesfalse\else\comparafalse\diferentestrue\fi}

\def\!compara#1#2{\!comparaaux{#1}{#2}}

\def\!absoluto#1#2{\n=#1%
  \ifnum \n > 0
    #2=\n
  \else
    \multiply \n by -1
    #2=\n
  \fi}

\def\solidarrow{0}

\def\atright{-1}

\def\!ajusta#1#2#3#4#5#6{\aux=#5%
  \let\auxobj=#6%
  \ifcase \tipografo    
    \ifnum\number\aux=10
      \ajustadisttrue 
    \else
      \ajustadistfalse  
    \fi
  \else  
   \ajustadistfalse
  \fi
  \ifajustadist
   \let\pilhaaux=\pilha%
   \loop%
     \!topo{\pilha}%
     \!retira{\pilha}%
     \!compara{\id}{\auxobj}%
     \ifcompara\nvaziafalse \else\!pilhanvazia\pilha \fi%
     \ifnvazia%
   \repeat%
   \let\pilha=\pilhaaux%
   \ifvazia%
    \ifdiferentes%
     \larg=1310720
     \prof=655360%
     \alt=655360%
    \fi%
   \fi%
   \divide\larg by 131072
   \divide\prof by 65536
   \divide\alt by 65536
   \ifnum\number\y=\number\yl
    \advance\larg by 3
    \ifnum\number\larg>\aux
     #5=\larg
    \fi
   \else
    \ifnum\number\x=\number\xl
     \ifnum\number\yl>\number\y
      \ifnum\number\alt>\aux
       #5=\alt
      \fi
     \else
      \advance\prof by 5
      \ifnum\number\prof>\aux
       #5=\prof
      \fi
     \fi
    \else
     \auxqx=\x
     \advance\auxqx by -\xl
     \!absoluto{\auxqx}{\auxqx}%
     \auxqy=\y
     \advance\auxqy by -\yl
     \!absoluto{\auxqy}{\auxqy}%
     \ifnum\auxqx>\auxqy
      \ifnum\larg<10
       \larg=10
      \fi
      \advance\larg by 3
      #5=\larg
     \else
      \ifnum\yl>\y
       \ifnum\larg<10
        \larg=10
       \fi
      \advance\alt by 6
       #5=\alt
      \else
      \advance\prof by 11
       #5=\prof
      \fi
     \fi
    \fi
   \fi
\fi} 

\def\!raiz#1#2{\n=#1%
  \m=1%
  \loop
    \aux=\m%
    \advance \aux by 1%
    \multiply \aux by \aux%
    \ifnum \aux < \n%
      \advance \m by 1%
      \paratrue%
    \else\ifnum \aux=\n%
      \advance \m by 1%
      \paratrue%
       \else\parafalse%
       \fi
    \fi
  \ifpara%
  \repeat
#2=\m}

\def\!ucoord#1#2#3#4#5#6#7{\aux=#2%
  \advance \aux by -#1%
  \multiply \aux by #4%
  \divide \aux by #5%
  \ifnum #7 = -1 \multiply \aux by -1 \fi%
  \advance \aux by #3%
#6=\aux}

\def\!quadrado#1#2#3{\aux=#1%
  \advance \aux by -#2%
  \multiply \aux by \aux%
#3=\aux}

\def\!distnomemor#1#2#3#4#5#6{\setbox0=\hbox{#5}%
  \aux=#1
  \advance \aux by -#3
  \ifnum \aux=0
     \aux=\wd0 \divide \aux by 131072
     \advance \aux by 3
     #6=\aux
  \else
     \aux=#2
     \advance \aux by -#4
     \ifnum \aux=0
        \aux=\ht0 \advance \aux by \dp0 \divide \aux by 131072
        \advance \aux by 3
        #6=\aux%
     \else
     #6=3
     \fi
   \fi
}

\def\begindc#1{\!ifnextchar[{\!begindc{#1}}{\!begindc{#1}[30]}}
\def\!begindc#1[#2]{\beginpicture
  \let\pilha=\!vazia
  \setcoordinatesystem units <1pt,1pt>
  \expansao=#2
  \ifcase #1
    \distanciaobjmor=10
    \tipoarco=0         
    \tipografo=0        
  \or
    \distanciaobjmor=2
    \tipoarco=0         
    \tipografo=1        
  \or
    \distanciaobjmor=1
    \tipoarco=2         
    \tipografo=2        
  \or
    \distanciaobjmor=8
    \tipoarco=0         
    \tipografo=3        
  \or
    \distanciaobjmor=8
    \tipoarco=2         
    \tipografo=4        
  \fi}

\def\enddc{\endpicture}

\def\mor{%
  \!ifnextchar({\!morxy}{\!morObjA}}
\def\!morxy(#1,#2){%
  \!ifnextchar({\!morxyl{#1}{#2}}{\!morObjB{#1}{#2}}}
\def\!morxyl#1#2(#3,#4){%
  \!ifnextchar[{\!mora{#1}{#2}{#3}{#4}}{\!mora{#1}{#2}{#3}{#4}[\number\distanciaobjmor,\number\distanciaobjmor]}}%
\def\!morObjA#1{%
 \let\pilhaaux=\pilha%
 \def\objPartida{#1}%
 \loop%
    \!topo\pilha%
    \!retira\pilha%
    \!compara{\id}{\objPartida}%
    \ifcompara \nvaziafalse \else \!pilhanvazia\pilha \fi%
   \ifnvazia%
 \repeat%
 \ifvazia%
  \ifdiferentes%
   Error: Incorrect label specification%
   \xaux=1%
   \yaux=1%
  \fi%
 \fi%
 \let\pilha=\pilhaaux%
 \!ifnextchar({\!morxyl{\number\xaux}{\number\yaux}}{\!morObjB{\number\xaux}{\number\yaux}}}
\def\!morObjB#1#2#3{%
  \x=#1
  \y=#2
 \def\objChegada{#3}%
 \let\pilhaaux=\pilha%
 \loop
    \!topo\pilha %
    \!retira\pilha%
    \!compara{\id}{\objChegada}%
    \ifcompara \nvaziafalse \else \!pilhanvazia\pilha \fi
   \ifnvazia
 \repeat
 \ifvazia
  \ifdiferentes%
   Error: Incorrect label specification
   \xaux=\x%
   \advance\xaux by \x%
   \yaux=\y%
   \advance\yaux by \y%
  \fi
 \fi
 \let\pilha=\pilhaaux
 \!ifnextchar[{\!mora{\number\x}{\number\y}{\number\xaux}{\number\yaux}}{\!mora{\number\x}{\number\y}{\number\xaux}{\number\yaux}[\number\distanciaobjmor,\number\distanciaobjmor]}}
\def\!mora#1#2#3#4[#5,#6]#7{%
  \!ifnextchar[{\!morb{#1}{#2}{#3}{#4}{#5}{#6}{#7}}{\!morb{#1}{#2}{#3}{#4}{#5}{#6}{#7}[1,\number\tipoarco] }}
\def\!morb#1#2#3#4#5#6#7[#8,#9]{\x=#1%
  \y=#2%
  \xl=#3%
  \yl=#4%
  \multiply \x by \expansao%
  \multiply \y by \expansao%
  \multiply \xl by \expansao%
  \multiply \yl by \expansao%
  \!quadrado{\number\x}{\number\xl}{\auxa}%
  \!quadrado{\number\y}{\number\yl}{\auxb}%
  \d=\auxa%
  \advance \d by \auxb%
  \!raiz{\d}{\d}%
  \auxa=#5
  \!compara{\objNulo}{\objPartida}%
  \ifdiferentes
   \!ajusta{\x}{\xl}{\y}{\yl}{\auxa}{\objPartida}%
   \ajustatrue
   \def\objPartida{}
  \fi
  \guardaauxa=\auxa
  \!ucoord{\number\x}{\number\xl}{\number\x}{\auxa}{\number\d}{\xa}{1}%
  \!ucoord{\number\y}{\number\yl}{\number\y}{\auxa}{\number\d}{\ya}{1}%
  \auxa=\d%
  \auxb=#6
  \!compara{\objNulo}{\objChegada}%
  \ifdiferentes
   \!ajusta{\x}{\xl}{\y}{\yl}{\auxb}{\objChegada}%
   \def\objChegada{}
  \fi
  \advance \auxa by -\auxb%
  \!ucoord{\number\x}{\number\xl}{\number\x}{\number\auxa}{\number\d}{\xb}{1}%
  \!ucoord{\number\y}{\number\yl}{\number\y}{\number\auxa}{\number\d}{\yb}{1}%
  \xmed=\xa%
  \advance \xmed by \xb%
  \divide \xmed by 2
  \ymed=\ya%
  \advance \ymed by \yb%
  \divide \ymed by 2
  \!distnomemor{\number\x}{\number\y}{\number\xl}{\number\yl}{#7}{\dnm}%
  \!ucoord{\number\y}{\number\yl}{\number\xmed}{\number\dnm}{\number\d}{\xc}{-#8}%
  \!ucoord{\number\x}{\number\xl}{\number\ymed}{\number\dnm}{\number\d}{\yc}{#8}%
\ifcase #9            
  \arrow <4pt> [.2,1.1] from {\xa} {\ya} to {\xb} {\yb}
\or                   
  \setdashes
  \arrow <4pt> [.2,1.1] from {\xa} {\ya} to {\xb} {\yb}
  \setsolid
\or                   
  \setlinear
  \plot {\xa} {\ya}  {\xb} {\yb} /
\or                   
  \setdashes
  \setlinear
  \plot {\xa} {\ya}  {\xb} {\yb} /
  \setsolid
\or                   
  \setdots
  \setlinear
  \plot {\xa} {\ya}  {\xb} {\yb} /
  \setsolid
\or                   
  \auxa=\guardaauxa
  \advance \auxa by 3%
 \!ucoord{\number\x}{\number\xl}{\number\x}{\number\auxa}{\number\d}{\xa}{1}%
 \!ucoord{\number\y}{\number\yl}{\number\y}{\number\auxa}{\number\d}{\ya}{1}%
 \!ucoord{\number\y}{\number\yl}{\number\xa}{3}{\number\d}{\xd}{-1}%
 \!ucoord{\number\x}{\number\xl}{\number\ya}{3}{\number\d}{\yd}{1}%
  \arrow <4pt> [.2,1.1] from {\xa} {\ya} to {\xb} {\yb}
  \circulararc -180 degrees from {\xa} {\ya} center at {\xd} {\yd}
\or                   
  \auxa=3
 \!ucoord{\number\y}{\number\yl}{\number\xa}{\number\auxa}{\number\d}{\xmed}{-1}%
 \!ucoord{\number\x}{\number\xl}{\number\ya}{\number\auxa}{\number\d}{\ymed}{1}%
 \!ucoord{\number\y}{\number\yl}{\number\xa}{\number\auxa}{\number\d}{\xd}{1}%
 \!ucoord{\number\x}{\number\xl}{\number\ya}{\number\auxa}{\number\d}{\yd}{-1}%
  \arrow <4pt> [.2,1.1] from {\xa} {\ya} to {\xb} {\yb}
  \setlinear
  \plot {\xmed} {\ymed}  {\xd} {\yd} /
\or                   
  \arrow <4pt> [.2,1.1] from {\xa} {\ya} to {\xb} {\yb}
  \setlinear
  \arrow <6pt> [0,.72] from {\xa} {\ya} to {\xb} {\yb}
\fi
\auxa=\xl
\advance \auxa by -\x%
\ifnum \auxa=0
  \put {#7} at {\xc} {\yc}
\else
  \auxb=\yl
  \advance \auxb by -\y%
  \ifnum \auxb=0 \put {#7} at {\xc} {\yc}
  \else
    \ifnum \auxa > 0
      \ifnum \auxb > 0
        \ifnum #8=1
          \put {#7} [rb] at {\xc} {\yc}
        \else
          \put {#7} [lt] at {\xc} {\yc}
        \fi
      \else
        \ifnum #8=1
          \put {#7} [lb] at {\xc} {\yc}
        \else
          \put {#7} [rt] at {\xc} {\yc}
        \fi
      \fi
    \else
      \ifnum \auxb > 0
        \ifnum #8=1
          \put {#7} [rt] at {\xc} {\yc}
        \else
          \put {#7} [lb] at {\xc} {\yc}
        \fi
      \else
        \ifnum #8=1
          \put {#7} [lt] at {\xc} {\yc}
        \else
          \put {#7} [rb] at {\xc} {\yc}
        \fi
      \fi
    \fi
  \fi
\fi
}

\def\modifplot(#1{\!modifqcurve #1}
\def\!modifqcurve(#1,#2){\x=#1%
  \y=#2%
  \multiply \x by \expansao%
  \multiply \y by \expansao%
  \!start (\x,\y)
  \!modifQjoin}
\def\!modifQjoin(#1,#2)(#3,#4){\x=#1%
  \y=#2%
  \xl=#3%
  \yl=#4%
  \multiply \x by \expansao%
  \multiply \y by \expansao%
  \multiply \xl by \expansao%
  \multiply \yl by \expansao%
  \!qjoin (\x,\y) (\xl,\yl)             
  \!ifnextchar){\!fim}{\!modifQjoin}}
\def\!fim){\ignorespaces}

\def\setaxy(#1{\!pontosxy #1}
\def\!pontosxy(#1,#2){%
  \!maispontosxy}
\def\!maispontosxy(#1,#2)(#3,#4){%
  \!ifnextchar){\!fimxy#3,#4}{\!maispontosxy}}
\def\!fimxy#1,#2){\x=#1%
  \y=#2
  \multiply \x by \expansao
  \multiply \y by \expansao
  \xl=\x%
  \yl=\y%
  \aux=1%
  \multiply \aux by \auxa%
  \advance\xl by \aux%
  \aux=1%
  \multiply \aux by \auxb%
  \advance\yl by \aux%
  \arrow <4pt> [.2,1.1] from {\x} {\y} to {\xl} {\yl}}

\def\cmor#1 #2(#3,#4)#5{%
  \!ifnextchar[{\!cmora{#1}{#2}{#3}{#4}{#5}}{\!cmora{#1}{#2}{#3}{#4}{#5}[0] }}
\def\!cmora#1#2#3#4#5[#6]{%
  \ifcase #2
      \auxa=0
      \auxb=1
    \or
      \auxa=0
      \auxb=-1
    \or
      \auxa=1
      \auxb=0
    \or
      \auxa=-1
      \auxb=0
    \fi  
  \ifcase #6  
    \modifplot#1
    \setaxy#1
  \or  
    \setdashes
    \modifplot#1
    \setaxy#1
    \setsolid
  \or  
    \modifplot#1
  \fi  
  \x=#3%
  \y=#4%
  \multiply \x by \expansao%
  \multiply \y by \expansao%
  \put {#5} at {\x} {\y}}

\def\obj(#1,#2){%
  \!ifnextchar[{\!obja{#1}{#2}}{\!obja{#1}{#2}[Nulo]}}
\def\!obja#1#2[#3]#4{%
  \!ifnextchar[{\!objb{#1}{#2}{#3}{#4}}{\!objb{#1}{#2}{#3}{#4}[1]}}
\def\!objb#1#2#3#4[#5]{%
  \x=#1%
  \y=#2%
  \def\!pinta{\normalsize$\bullet$}
  \def\!nulo{Nulo}%
  \def\!arg{#3}%
  \!compara{\!arg}{\!nulo}%
  \ifcompara\def\!arg{#4}\fi%
  \multiply \x by \expansao%
  \multiply \y by \expansao%
  \setbox\caixa=\hbox{#4}%
  \!coloca{(\!arg)(#1,#2)(\number\ht\caixa,\number\wd\caixa,\number\dp\caixa)}{\pilha}%
  \auxa=\wd\caixa \divide \auxa by 131072
  \advance \auxa by 5
  \auxb=\ht\caixa
  \advance \auxb by \number\dp\caixa
  \divide \auxb by 131072
  \advance \auxb by 5
  \ifcase \tipografo    
    \put{#4} at {\x} {\y}
  \or                   
    \ifcase #5 
      \put{#4} at {\x} {\y}
    \or        
      \put{\!pinta} at {\x} {\y}
      \advance \y by \number\auxb  
      \put{#4} at {\x} {\y}
    \or        
      \put{\!pinta} at {\x} {\y}
      \advance \auxa by -2  
      \advance \auxb by -2  
      \advance \x by \number\auxa  
      \advance \y by \number\auxb  
      \put{#4} at {\x} {\y}
    \or        
      \put{\!pinta} at {\x} {\y}
      \advance \x by \number\auxa  
      \put{#4} at {\x} {\y}
    \or        
      \put{\!pinta} at {\x} {\y}
      \advance \auxa by -2  
      \advance \auxb by -2  
      \advance \x by \number\auxa  
      \advance \y by -\number\auxb  
      \put{#4} at {\x} {\y}
    \or        
      \put{\!pinta} at {\x} {\y}
      \advance \y by -\number\auxb  
      \put{#4} at {\x} {\y}
    \or        
      \put{\!pinta} at {\x} {\y}
      \advance \auxa by -2  
      \advance \auxb by -2  
      \advance \x by -\number\auxa  
      \advance \y by -\number\auxb  
      \put{#4} at {\x} {\y}
    \or        
      \put{\!pinta} at {\x} {\y}
      \advance \x by -\number\auxa  
      \put{#4} at {\x} {\y}
    \or        
      \put{\!pinta} at {\x} {\y}
      \advance \auxa by -2  
      \advance \auxb by -2  
      \advance \x by -\number\auxa  
      \advance \y by \number\auxb  
      \put{#4} at {\x} {\y}
    \fi
  \or                   
    \ifcase #5 
      \put{#4} at {\x} {\y}
    \or        
      \put{\!pinta} at {\x} {\y}
      \advance \y by \number\auxb  
      \put{#4} at {\x} {\y}
    \or        
      \put{\!pinta} at {\x} {\y}
      \advance \auxa by -2  
      \advance \auxb by -2  
      \advance \x by \number\auxa  
      \advance \y by \number\auxb  
      \put{#4} at {\x} {\y}
    \or        
      \put{\!pinta} at {\x} {\y}
      \advance \x by \number\auxa  
      \put{#4} at {\x} {\y}
    \or        
      \put{\!pinta} at {\x} {\y}
      \advance \auxa by -2  
      \advance \auxb by -2
      \advance \x by \number\auxa  
      \advance \y by -\number\auxb 
      \put{#4} at {\x} {\y}
    \or        
      \put{\!pinta} at {\x} {\y}
      \advance \y by -\number\auxb 
      \put{#4} at {\x} {\y}
    \or        
      \put{\!pinta} at {\x} {\y}
      \advance \auxa by -2  
      \advance \auxb by -2
      \advance \x by -\number\auxa 
      \advance \y by -\number\auxb 
      \put{#4} at {\x} {\y}
    \or        
      \put{\!pinta} at {\x} {\y}
      \advance \x by -\number\auxa 
      \put{#4} at {\x} {\y}
    \or        
      \put{\!pinta} at {\x} {\y}
      \advance \auxa by -2  
      \advance \auxb by -2
      \advance \x by -\number\auxa 
      \advance \y by \number\auxb  
      \put{#4} at {\x} {\y}
    \fi
   \else 
     \ifnum\auxa<\auxb 
       \aux=\auxb
     \else
       \aux=\auxa
     \fi
     \ifdim\wd\caixa<1em
       \dimen99 = 1em
       \aux=\dimen99 \divide \aux by 131072
       \advance \aux by 5
     \fi
     \advance\aux by -2 
     \multiply\aux by 2 %
     \ifnum\aux<30
       \put{\circle{\aux}} [Bl] at {\x} {\y}
     \else
       \multiply\auxa by 2
       \multiply\auxb by 2
       \put{\oval(\auxa,\auxb)} [Bl] at {\x} {\y}
     \fi
     \put{#4} at {\x} {\y}
   \fi
}

\catcode`!=12 

\numberwithin{equation}{section}
\newtheorem{Theorem}{Theorem}[section]

\newtheorem{Lemma}[Theorem]{Lemma}
\newtheorem{Corollary}[Theorem]{Corollary}
\newtheorem{Definition}[Theorem]{Definition}
\theoremstyle{definition}
\newtheorem{Example}[Theorem]{Example}
\newtheorem{Remark}[Theorem]{Remark}
\newfont{\deffont}{cmbxti10}
\newfont{\german}{eufm10}
\newfont{\mymath}{cmr12}

\newcommand\lieg{\mathfrak{g}}

\newcommand\lier{\mathfrak{r}}
\newcommand\lies{\mathfrak{s}}
\renewcommand\qed{\hfill\hbox{\vrule width 4pt height 6pt depth 1.5 pt}}

\newcommand\hook{\mathbin{\raise2.5pt\hbox{\hbox{{\vbox{\hrule height.4pt
width6pt depth0pt}}}\vrule height3pt width.4pt depth0pt}\,}}

\newcommand\extd{d\mspace{2mu}}
\newcommand\cTM{T^*\kern-2ptM}

\newcommand\Ad{\text{\rm Ad}}

\newcommand\Inv{\text{\rm Int}}
\newcommand\semibasic{\text{\bf  sb}}

\newcommand\vess{\text{\german v\german e\german s\german s}}

\newcommand\thetaX{\theta_{\kern -1 pt X}}
\newcommand\thetaY{\theta_Y}

\newcommand\ann{\text{\rm ann}}
\newcommand\CalPf{\mathcal{P}\text{\it \kern -.3pt f}}
\newcommand\Real{\text{\bf  R}}
\newcommand\Ann{\text{\rm Ann}}

\renewcommand\:{\colon}

\newcommand\TM{T\kern -2pt M}

\newcommand\real{\text{\bf R}}
\newcommand\mycap{\hbox{\ $\rlap{\kern -.3pt $\cap$}\raise.8pt\hbox{$\scriptstyle+$}$\ } }

\DeclareMathOperator{\rank}{rank}

\newboolean{proofmode}

\newcommand{\StTag}[1]{ \label{st:#1}
\ifthenelse{\boolean{proofmode}}{\ \marginpar{\quad\scriptsize st:#1} }{}      }

\newcommand{\EqTag}[1]{
\ifthenelse{\boolean{proofmode}}
{ {\label{eq:#1}}
  \stepcounter{equation}
  \tag{\theequation \rlap{\kern 23 pt{\scriptsize eq:#1}}}
}
{\label{eq:#1}}
 }

\newcommand{\EqRef}[1]{\eqref{eq:#1}}
\newcommand{\StRef}[1]{\ref{st:#1}}
\newcolumntype{C}{>\scriptstyle>{$}c <{$} }
\newcolumntype{L}{>\scriptstyle >{$} l <{$} }
\newcommand\CalA{\mathcal{A}}

\newcommand\CalI{\mathcal{I}}
\newcommand\CalL{\mathcal{L}}
\newcommand\CalJ{\mathcal{J}}
\newcommand\CalK{\mathcal{K}}

\newcommand\CalF{\mathcal{F}}

\newcommand\CalS{\mathcal{S}}
\newcommand\CalU{\mathcal{U}}
\newcommand\CalW{\mathcal{W}}
\newcommand\CalV{\mathcal{V}}

\newcommand\barM{
	\hbox{\kern 2.3 true pt
	\vbox{\hrule width 8.5  true pt height .3 true pt \kern .9 true pt
	\hbox{\kern -2.3 true pt $M$}}}}
\newcommand\barU{
	\hbox{\kern .8 true pt
	\vbox{\hrule width 6.5  true pt height .3 true pt \kern .9 true pt
	\hbox{\kern -.8 true pt $U$}}}}
\newcommand\barCalI{
	\hbox{\kern 4.3 true pt
	\vbox{\hrule width 6.5  true pt height .3 true pt \kern .9 true pt
	\hbox{ \kern -4.3 true pt  $\CalI$}}}}
\newcommand\barXi{
	\hbox{\kern 1 true pt
	\vbox{\hrule width 6.5  true pt height .3 true pt \kern .9 true pt
	\hbox{\kern 1 true pt $\Xi$}}}}
\newcommand\Largehat{\smash{\raise -7.5 pt \hbox{\rm\Large\^{}}}}
\newcommand\LARGEhat{\smash{\raise -9.5 pt \hbox{\rm\LARGE\^{}}}}
\newcommand\hugehat{\smash{\raise -7.5 pt \hbox{\rm\huge\^{}}}}
\newcommand\Hugehat{\smash{\raise -7.5 pt \hbox{\rm\Huge\^{}}}}

\newcommand\Largecheck{\smash{\raise -7.5 pt \hbox{\rm\Large\v{}}}}
\newcommand\LARGEcheck{\smash{\raise -9 pt \hbox{\rm\LARGE\v{}}}}
\newcommand\hugecheck{\smash{\raise -7.5 pt \hbox{\rm\huge\v{}}}}
\newcommand\Hugecheck{\smash{\raise -7.5 pt \hbox{\rm\Huge\v{}}}}

\newcommand\hE{{\accentset{\LARGEhat}{E}}}
\newcommand\hH{{\accentset{\LARGEhat}{H}}}
\newcommand\hI{{\accentset{\LARGEhat}{I}}}

\newcommand\hR{{\accentset{\LARGEhat}{R}}}
\newcommand\hS{{\accentset{\LARGEhat}{S}}}
\newcommand\hT{{\accentset{\LARGEhat}{T}}}
\newcommand\hU{{\accentset{\LARGEhat}{U}}}
\newcommand\hV{{\accentset{\LARGEhat}{V}}}
\newcommand\hW{{\accentset{\LARGEhat}{W}}}
\newcommand\hX{{\accentset{\LARGEhat}{X}}}
\newcommand\hY{{\accentset{\LARGEhat}{Y}}}
\newcommand\hZ{{\accentset{\LARGEhat}{Z}}}
\newcommand\homega{\accentset{\Largehat}{\omega}}
\newcommand\htheta{\accentset{\Largehat}{\theta}}
\newcommand\hsigma{\accentset{\Largehat}{\sigma}}
\newcommand{\heta}{\accentset{\Largehat}{\eta}}
\newcommand{\hpi}{\accentset{\Largehat}{\pi}}
\newcommand{\hphi}{\accentset{\Largehat}{\phi}}
\newcommand{\hmu}{\accentset{\Largehat}{\mu}}
\newcommand{\hrho}{\accentset{\Largehat}{\rho}}
\newcommand{\hvarphi}{\accentset{\Largehat}{\varphi}}
\newcommand{\hOmega}{\accentset{\LARGEhat}{\Omega}}
\newcommand{\hGamma}{\accentset{\LARGEhat}{\Gamma}}
\newcommand\hstar{\accentset{\Largehat}{\,*\,}}

\newcommand\hCalU{{\accentset{\LARGEhat}{{\mathcal U}}}}
\newcommand\hCalV{{\accentset{\LARGEhat}{{\mathcal V}}}}
\newcommand\hCalW{{\accentset{\LARGEhat}{{\mathcal W}}}}

\newcommand\cE{{\accentset{\LARGEcheck}{E}}}
\newcommand\cH{{\accentset{\LARGEcheck}{H}}}
\newcommand\cI{{\accentset{\LARGEcheck}{I}}}

\newcommand\cR{{\accentset{\LARGEcheck}{R}}}
\newcommand\cS{{\accentset{\LARGEcheck}{S}}}
\newcommand\cT{{\accentset{\LARGEcheck}{T}}}
\newcommand\cU{{\accentset{\LARGEcheck}{U}}}
\newcommand\cV{{\accentset{\LARGEcheck}{V}}}
\newcommand\cW{{\accentset{\LARGEcheck}{W}}}
\newcommand\cX{{\accentset{\LARGEcheck}{X}}}
\newcommand\cY{{\accentset{\LARGEcheck}{Y}}}
\newcommand\cZ{{\accentset{\LARGEcheck}{Z}}}
\newcommand\comega{\accentset{\Largecheck}{\omega}}
\newcommand\ctheta{\accentset{\Largecheck}{\theta}}
\newcommand\cpi{\accentset{\Largecheck}{\pi}}
\newcommand\ceta{\accentset{\Largecheck}{\eta}}
\newcommand\csigma{\accentset{\Largecheck}{\sigma}}
\newcommand\cphi{\accentset{\Largecheck}{\phi}}
\newcommand\cmu{\accentset{\Largecheck}{\mu}}
\newcommand\crho{\accentset{\Largecheck}{\rho}}
\newcommand\cvarphi{\accentset{\Largecheck}{\varphi}}
\newcommand\cOmega{\accentset{\LARGEcheck}{\Omega}}
\newcommand\cGamma{\accentset{\LARGEcheck}{\Gamma}}
\newcommand\cstar{\accentset{\Largecheck}{\,*\,}}

\newcommand\cCalU{{\accentset{\LARGEcheck}{{\mathcal U}}}}
\newcommand\cCalV{{\accentset{\LARGEcheck}{{\mathcal V}}}}
\newcommand\cCalW{{\accentset{\LARGEcheck}{{\mathcal W}}}}

\newcommand\bfA{\boldsymbol{A}}
\newcommand\bfB{\boldsymbol{B}}
\newcommand\bfC{\boldsymbol{C}}
\newcommand\bfE{\boldsymbol{E}}
\newcommand\bfF{\boldsymbol{F}}
\newcommand\bfG{\boldsymbol{G}}
\newcommand\bfH{\boldsymbol{H}}
\newcommand\bfK{\boldsymbol{K}}
\newcommand\bfM{\boldsymbol{M}}
\newcommand\bfN{\boldsymbol{N}}
\newcommand\bfq{\mathbf{q}}
\newcommand\bfR{\boldsymbol{R}}

\newcommand\bfX{\boldsymbol{X}}
\newcommand\bfY{\boldsymbol{Y}}
\newcommand\bfP{\boldsymbol{P}}
\newcommand\bfQ{\boldsymbol{Q}}			

\newcommand\bfphi{{\boldsymbol{\phi}}}
\newcommand\bftheta{{\boldsymbol{\theta}}}
\newcommand\bfthetaX{{\boldsymbol{\theta_{\kern -1 pt X}}}}
\newcommand\bfthetaY{{\boldsymbol{\theta_Y}}}
\newcommand\bfhE{\boldsymbol{{\hat E}}}
\newcommand\bfhF{\boldsymbol{{\hat F}}} 		
\newcommand\bfhI{\boldsymbol{{\hat I}}}
\newcommand\bfhU{\boldsymbol{{\hat U}}}
\newcommand\bfheta{\boldsymbol{{\hat \eta}}}
\newcommand\bfhiota{\boldsymbol{{\hat \iota}}}
\newcommand\bfhpi{\boldsymbol{{\hat \pi}}}
\newcommand\bfhsigma{\boldsymbol{{\hat \sigma}}}
\newcommand\bfhtheta{\boldsymbol{{\hat \theta}}}
\newcommand\bfhzeta{\boldsymbol{{\hat \zeta}}}
\newcommand\bfcE{\boldsymbol{{\check E}}}
\newcommand\bfcF{\boldsymbol{{\check F}}}
\newcommand\bfcI{\boldsymbol{{\check I}}}
\newcommand\bfcU{\boldsymbol{{\check U}}}
\newcommand\bfceta{\boldsymbol{{\check \eta}}}
\newcommand\bfciota{\boldsymbol{{\check\iota}}}
\newcommand\bfcpi{\boldsymbol{{\check\pi}}}
\newcommand\bfcsigma{\boldsymbol{{\check\sigma}}}
\newcommand\bfctheta{\boldsymbol{{\check \theta}}}
\newcommand\bfczeta{\boldsymbol{{\check \zeta}}}
\newcommand\bfthetapr{{\boldsymbol{\theta'}}}
\newcommand\bfhsigmapr{{\boldsymbol{{{\hat\sigma}'}}}}
\newcommand\bfcsigmapr{{\boldsymbol{{{\check\sigma}'}}}}

\newcommand\bfhetapr{{\boldsymbol{{{\hat\eta}'}}}}
\newcommand\bfcetapr{{\boldsymbol{{{\check\eta}'}}}}
\newcommand\bfvecthetapr{ {\boldsymbol{{{\hat s}'}}}}
\newcommand\bfvectcetapr{ {\boldsymbol{{{\check s}'}}}}
\newcommand\Rtheta{{ \raise 1pt \hbox{$\scriptstyle {\boldsymbol{\theta}}$}}}
\newcommand\Ltheta{{ \lower 1pt \hbox{$\scriptstyle {\boldsymbol{\theta}}$}}}
\newcommand\bfvecttheta{{\boldsymbol{\partial}_{\Ltheta}}}

\newcommand\Rsigma{{ \raise 1pt \hbox{$\scriptstyle \sigma$}}}
\newcommand\Hsigma{{\boldsymbol{\displaystyle{ \hat {\Rsigma }}}}}
\newcommand\bfvecthsigma{{\boldsymbol{\partial}_{\Hsigma}}}

\newcommand\Csigma{{\boldsymbol{\displaystyle{ \check {\Rsigma }}}}}
\newcommand\bfvectcsigma{{\boldsymbol{\partial}_{\Csigma}}}

\newcommand\Reta{{ \raise 1pt \hbox{$\scriptstyle \eta$}}}
\newcommand\Heta{{\boldsymbol{\displaystyle{ \hat {\Reta }}}}}
\newcommand\bfvectheta{{\boldsymbol{\partial}_{\Heta}}}

\newcommand\Ceta{{\boldsymbol{\displaystyle{ \check {\Reta }}}}}
\newcommand\bfvectceta{{\boldsymbol{\partial}_{\Ceta}}}

\newcommand\vecthsigma{\partial_{{\displaystyle \hat {\raise 1.3pt \hbox{$\scriptstyle \sigma$}}}^a}}
\newcommand\vectcsigma{\partial_{{\displaystyle \check {\raise 1.3pt \hbox{$\scriptstyle \sigma$}}}^\alpha}}

\newcommand\bfone{ {\boldsymbol{1}}}

\newcommand\bfhR{{\mathbf{\hat R}}}
\newcommand\bfcR{{\mathbf{\check R}}}

\newcommand\bfhS{{\boldsymbol{ {\hat S}}}}

\setboolean{proofmode}{false}

\begin{document}

\title{
	Superposition Formulas for
\\
	Exterior Differential Systems
}

\author{ Ian Anderson \\ Dept of Math. and Stat. \\ Utah State University  \and Mark Fels  \\Dept of Math. and Stat. \\ Utah State University   \and Peter Vassiliou \\  Dept. of Math and Stat. \\ University of Canberra}
\maketitle
\newpage
\setcounter{page}{1}
\section{Introduction}
	
	In this  paper  we use the method of symmetry reduction for exterior differential
	systems to obtain a  far-reaching generalization of  Vessiot's integration method \cite{vessiot:1939a}, \cite{vessiot:1942a} for
	Darboux integrable, partial differential equations.
	This group-theoretic approach provides deep insights into this classical method; uncovers the fundamental geometric
	invariants of Darboux integrable systems;  provides for their algorithmic integration;
	and has  applications well beyond  those currently found in the literature. In particular,
        our integration method is  applicable to systems of hyperbolic  PDE  such as the Toda lattice equations,
	 \cite{leznov-saveliev:1980a}, \cite{leznov-saveliev:1983a}, \cite{sokoliv-ziber:1995a}, 2-dimensional  wave maps
	\cite{barbashov-nesterenko-chervyakov:1982a}  and systems of overdetermined PDE such as those studied
	by Cartan \cite{cartan:1911a}.

	Central to our generalization of  Vessiot's work is the novel concept  of  a {\deffont superposition formula} for
	an exterior differential system
	$\CalI$ on a manifold $M$. A superposition formula for $\CalI$ is a pair of differential systems
	$\hCalW$, $\cCalW$, defined on  manifolds $M_1$ and $M_2$, and a mapping
\begin{equation}
	\Sigma \: M_1 \times M_2 \to M
\EqTag{Intro1}
\end{equation}
	such that
\begin{equation}
	\Sigma^*(\CalI) \subset \pi^*_1(\hCalW)  + \pi^*_2(\cCalW).          %
\EqTag{Intro2}
\end{equation}
	Here   $\pi^*_1(\hCalW +\pi^*_2(\cCalW)$
	is the differential system
	generated by the pullbacks of $\hCalW$ and $\cCalW$ to  the  product manifold $M_1 \times M_2 $ by the canonical projection maps
	$\pi_1$ and $\pi_2$.
	It is then
	clear that if  $\hphi\: N_1 \to M_1$ and $\cphi\: N_2 \to M_2$
	are integral manifolds for $\hCalW$ and $\cCalW$,  then
\begin{equation}
	\phi= \Sigma\circ (\hphi,\, \cphi) \:
	N_1 \times N_2 \to M
\EqTag{Intro3}
\end{equation}
	is an (possibly non-immersed) integral manifold of $\CalI$.

	In this article we  shall
\newcounter{fig}
\begin{list}{\bfseries [\roman{fig}]}{\usecounter{fig}
\setlength{\topsep}{1 pt}
\setlength{\itemsep}{0pt}
\setlength{\leftmargin}{20 pt}
\setlength{\listparindent}{15 pt} }
\item 	establish general sufficiency conditions (in terms of geometric invariants of the differential system $\CalI$)
	for the existence of a superposition formula;
\item  	establish general sufficiency conditions under which
	the superposition formula gives {\it all \/} local integral manifolds of $\CalI$  in terms of the integral manifolds of
	$\pi^*_1(\hCalW)  + \pi^*_2(\cCalW)$ on $M_1 \times M_2$ (in which case we say that the superposition formula is surjective);
\item	provide an algorithmic procedure for finding the superposition formula; and
\item demonstrate the effectiveness of our approach  with an extensive number  of examples and applications.
\end{list}

	Differential systems  admitting superposition formula are easily constructed by symmetry reduction.
	To briefly describe this construction,
	let $G$ be a symmetry group of a differential system $\CalW$ on a manifold $N$.  We assume that
	the quotient space $M = N/G$  of $N$ by the orbits of $G$ has a smooth manifold structure for which
	the projection map  $\bfq \colon N \to M$ is smooth.    We  then define the {\deffont $G$ reduction of $\CalW$ or
	quotient of $\CalW$ by $G$}  as the differential system on $M$  given by
\begin{equation}
         \CalW/G =  \{\, \omega \in \Omega^*(M) \, | \, \bfq^*(\omega) \in \CalW\, \}.
\end{equation}
	The traditional application of symmetry reduction has been to integrate $\CalW$ by integrating $\CalW/G$. See, for example, \cite{anderson-fels:2005a}.

	But now suppose that differential systems $\hCalW$ and $\cCalW$ on manifolds $M_1$ and $M_2$ have a common symmetry group $G$.
	Define the differential system $\CalW = \pi^*_1(\hCalW)  + \pi^*_2(\cCalW)$ on $M_1\times M_2$  and let $G$ act on $M_1\times M_2$ by the diagonal action.
	Then, by definition, the quotient map  $\bfq \colon M_1\times M_2 \to M = (M_1\times M_2)/G$
	defines a superposition formula for the quotient differential system $\CalW/G$ on $M$.
	In this paper we discover the means by which the inverse process to  symmetry reduction is possible,
	that is, {\it we show  how certain general classes of differential systems $\CalI$ on $M$  can be identified with a quotient system
	$\CalW/G$  in which case the integral manifolds of $\CalI$ can  then be found from those of $\hCalW$  and $\cCalW$. }

\par
	To intrinsically describe the class of differential systems for which we shall construct superposition formulas, we first introduce the
	definition of a decomposable differential system.

\begin{Definition}
\StTag{Intro1}	
	An exterior differential system  $\CalI$ on $M$ is   {\deffont decomposable of type $[p,q]$}, where
	$p,q \geq 2$, if about each point $x \in M$ there is a  coframe
\begin{equation}
	\tilde \theta^1,\ \ldots,\ \tilde \theta^r,\
	\hsigma^1,\ \dots,\ \hsigma^{p}, \
	\csigma^1,\ \dots,\ \csigma^{q},
\EqTag{Intro4}
\end{equation}
	such that $\CalI$ is algebraically generated by  1-forms and 2-forms
\begin{equation}
	\CalI = \{\, \tilde \theta^1, \ \dots, \ \tilde \theta^r, \ \hOmega^1,\ \dots,\ \hOmega^{s},\ \cOmega^1,\  \dots, \ \cOmega^{t} \, \},
\EqTag{Intro5}	
\end{equation}
	where $s, t \geq  1$,
	$\hOmega^a \in \Omega^2(\hsigma^1,\dots,\hsigma^{p})$, and
	$\cOmega^\alpha \in \Omega^2(\csigma^1,\dots,\csigma^{q})$.
	The differential systems algebraically generated by
\begin{equation}
	\hCalV =\{\,\tilde \theta^i,\, \hsigma^a, \,\cOmega^\alpha\, \}
	\quad\text{and}\quad
	\cCalV = \{\, \tilde \theta^i,\, \csigma^\alpha, \,\hOmega^a\, \}
\EqTag{Intro6}
\end{equation}
	are called the associated {\deffont singular differential systems} for $\CalI$ with respect to the decomposition \EqRef{Intro5}.
\end{Definition}
\par
	 With the goal of constructing superposition formulas, we have found it most natural
	to focus on the case where $\CalI$ is decomposable (but not necessarily Pfaffian)  and $\hCalV$ and $\cCalV$ are
	(constant rank) Pfaffian\footnote{We use the term Pfaffian system  to designate  either a constant rank subbundle $V$  of $T^*M$ or  the differential system,
	denoted by the corresponding calligraphic letter   $\CalV$, generated by the sections of $V$.}.
	All of the examples we consider are of this type.  Note that any class $r$ hyperbolic differential system, as
	defined in   \cite{	bryant-griffiths-hsu:1995a},  is a
	decomposable differential system   and that the associated characteristic  Pfaffian systems coincide, for $r>0$,
	with the singular systems\EqRef{Intro6}.

	The definition of a Darboux integrable, decomposable  differential system is given in terms of its singular systems.
	For any Pfaffian system $V$, let $V^{(\infty)}$
	denoted the largest integrable subbundle of $V$. The rank of $V^{(\infty)}$
	gives the number of functionally independent first integrals for $V$. By definition,
	a scalar second order partial differential equation in the plane is Darboux integrable if the associated singular
	Pfaffian systems $\hV$ and $\cV$ each admit at least 2 (functionally independent) first integrals. Thus,
	in order to generalize the definition of Darboux integrablity,
	we  must determine the required number of functionally independent first integrals necessary to
	integrate a general decomposable Pfaffian system. We do this with the following definitions.
\begin{Definition}
\StTag{Intro2}
	 A pair of Pfaffian systems $\hV$ and $\cV$ define a
	{\deffont Darboux pair} if  the following conditions hold.
\par
\smallskip
\noindent
{\bf[i]}
\vskip -30pt
\begin{equation}\hV  + \cV^{(\infty)} = \cTM
	\quad\text{and}\quad
        \cV  + \hV^{(\infty)} = \cTM.
\EqTag{Intro7}
\end{equation}
\par
\smallskip
\noindent
{\bf[ii]}
\vskip -30pt
\begin{equation}
 \hV^{(\infty)} \cap \cV^{\infty} = \{\,0\, \}.
\EqTag{Intro8}
\end{equation}
\par
\smallskip
\noindent
{\bf[iii]}
\vskip-30pt
\begin{equation}
	\extd\omega  \in  \Omega^2(\hV) + \Omega^2(\cV)\quad \text{for  all\quad $\omega\in \Omega^1(\hV \cap \cV)$}\ .
	\EqTag{Intro9}
\end{equation}
\end{Definition}

\begin{Definition}
\StTag{Intro4}
	Let $\CalI$ be a decomposable differential system and assume that the  associated singular systems
	$\hV$ and $\cV$ are Pfaffian. Then $\CalI$ is said to be {\deffont Darboux integrable} if $\{\, \hV,\,  \cV\,\}$
	define a Darboux pair.	
\end{Definition}
	Property {\bf [i]} of Definition \StRef{Intro2}
	is the critical one -- it will insure that there are a sufficient number of first integrals  to construct
	a superposition formula. Property {\bf[ii]} is  a technical condition which states
	simply 	that $\hV$ and $\cV$ share no common  integrals -- this condition can always be
	satisfied by  restricting  $\hV$ and $\cV$ to a level set of any common integrals.
	The  form of the structure equations for $\hV \cap \cV$
	required by property {\bf[iii]} is always satisfied when $\hV$ and $\cV$ are the singular Pfaffian systems for a
	decomposable differential system $\CalI$.

	Our main result can now be stated.
\begin{Theorem}
\StTag{Intro5}
	Let $\CalI$ be a decomposable differential system on $M$ whose associated singular Pfaffian systems
	$\{\, \hV, \, \cV \,\}$ define a Darboux pair. Then there are Pfaffian systems $W_1$  and $W_2$ on  manifolds $M_1$ and $M_2$
	which admit a common Lie group $G$ of symmetries and such that
\par
\smallskip
\noindent
{\bf [i]} the manifold $M$ can be identified (at least locally) as the quotient of $M_1\times M_2$ by the diagonal action of the group $G$;
\par
\smallskip
\noindent
{\bf[ii]}
\vskip-30pt
\begin{equation}
	\CalI = (\pi^*_1(\hCalW +\pi^*_2(\cCalW))/G; \quad\text{and}
\end{equation}
\par
\smallskip
\noindent
{\bf[iii]} the  quotient map $\bfq\colon M_1\times M_2 \to M$ defines a surjective superposition formula for $\CalI$.
\end{Theorem}
\par
	The  manifolds  $M_1$ and $M_2$  in Theorem \StRef{Intro5} are simply any maximal integral manifolds for
	$\hV^{(\infty)}$ and $\cV^{(\infty)}$   and the Pfaffian
	systems $W_1$ and $W_2$ are just the restrictions of $\cV$ and $\hV$  to these manifolds.
	But the proper identification of the  Lie group $G$ and its action on $M_1$ and $M_2$ is not so easy to uncover. This is done
	through a sequence of non-trivial coframe adaptations
	and represents the principle technical achievement of the paper (See Theorem \StRef{CofrAdapt1} and Definition \StRef{SuperForm33}) We call  $G$ 	the {\deffont Vessiot group} for the Darboux integrable, differential system $\CalI$. 	
\par	
	The paper is organized as follows. In Section 2 we obtain some simple sufficiency conditions for a differential system to be decomposable
        and we give necessary and sufficient conditions for a Pfaffian system to be decomposable (Theorem \StRef{Dec3}). We  also introduce the initial
	adapted coframes for a Darboux pair (Theorem \StRef{First9}). In Section 3  we answer the question of when the symmetry reduction of a Darboux
pair is also a Darboux pair (Theorem \StRef{SymRed1}) and we use this result to give a general method for constructing Darboux integrable differential systems (Theorem \StRef{SymRed8}).
	Section 4 establishes the sequence of coframe adaptations leading to the definition of the Vessiot group $G$ and its action on $M$
	and hence on $M_1$ and $M_2$ (Theorem \StRef{CofrAdapt1}).
	In Section 5 we construct the superposition formula (Theorem \StRef{SuperForm6}) and prove that it may be identified with the quotient for the diagonal action of the
	Vessiot group. This proves Theorem \StRef{Intro5}.

	A substantial portion of the paper is devoted to examples.
	These examples, given in Section 6, provide closed-form general solutions to a wide range of differential equations,
	explicitly illustrate the various  coframe adaptations of Section 4, and
	underscore the importance of the Vessiot group as the fundamental invariant for Darboux integrable differential systems.
	The first two examples are taken from the classical literature. The equation considered in Example \StRef{Ex1} is
        specifically chosen  to illustrate in complete detail  all the various coframe adaptations used to determine the superposition formula.
	In Example \StRef{Ex2} we integrate all non-linear Darboux equations of the type $u_{xx} = f(u_{yy})$ (with second order invariants)
	and relate the integration of these equations to Cartan's classification of rank 3 Pfaffian systems  in 5 variables \cite{cartan:1910a}.
	This connection between the method of Darboux and Cartan's classification is  apparent
	from our interpretation of the method of Darboux in terms of symmetry reduction of differential systems.

	In Examples \StRef{Ex3} and \StRef{Ex4} we present a number of examples of Darboux integrable equations  where the unknown function takes values
	in a group or in a non-commutative algebra. Example \StRef{Ex3}
	provides us with a system whose Vessiot group is an arbitrary Lie group $G$.

	Some of the simplest examples of  systems of Darboux integrable partial differential equations
	can be constructed by the coupling of a nonlinear Darboux integrable scalar equation to a linear or Moutard-type equation.
	These are presented in  Example \StRef{Ex5}. It is noteworthy that for these equations the Vessiot group is a semi-direct product
	of the Vessiot group for the non-linear equation with an Abelian group.
	The representation theoretical implications of this observation will be further explored elsewhere.

	In Example \StRef{Ex6} we illustrate the computational power of our methods by
	explicitly integrating the $B_2$ Toda lattice system.
	While the general solution to the $A_n$ Toda systems  were known to Darboux, this is, to the best of our knowledge, the  first time
	explicit general solutions to other  Toda systems have been given.  Based on this work we conjecture that the Vessiot group for
	the $\mathfrak{g}$ Toda lattice system, where $\mathfrak{g}$ is any semi-simple Lie group, is the associated Lie group $G$ itself.
	In Example \StRef{Ex7} we integrate a
	wave map system -- this system admits a number of interesting geometric properties which will be explored in detail in a subsequent paper.
 	Finally, in Example \StRef{Ex8}, we  present some original examples of Darboux integrable, non-linear,
	over-determined systems in 3 independent variables.

	The  extensive computations required by all these examples were done using the DifferentialGeometry
	package in Maple 11. This research was supported by NSF grant DMS-0410373 and DMS-0713830.
\newpage

\section{Preliminaries}

\subsection{Decomposable Differential Systems}

	In this section we give some simple necessary conditions, in terms of the notion of
	singular vectors, for a differential system\footnote{We assume that all exterior differential systems are of constant rank.} $\CalI$ to be decomposable (see Definition \StRef{Intro1});
	we give sufficient conditions for a Pfaffian system to be decomposable;  and we
	address the problem  (see Theorem \StRef{Dec6}) of
	re-constructing a decomposable differential system from its associated singular systems.
\par
	Let $\CalI$ be a differential system on $M$. Fix a point $x$ in $M$ and let
\begin{equation*}
	E^1_x(\CalI) = \{\, X\in T_x M \ |\   \theta(X) = 0 \quad \text{for all 1-forms $\theta \in \CalI$} \,\}.
\end{equation*}
	The {\deffont polar equations} determined by a non-zero vector $X \in E^1_x(\CalI)$ are, by definition,
	the linear system of  equations for $Y \in   E^1_x(\CalI)$ given by
\begin{equation*}
	\theta(X,Y) = 0 \quad \text{for all 2-forms $\theta \in \CalI$}.
\end{equation*}
	Then  $X\in  E^1_x(\CalI)$ is said to be
	{\deffont regular} if the rank of its polar equations is maximal and {\deffont singular} otherwise.
\begin{Lemma}
\StTag{Dec1}
	Let $\CalI$ be a decomposable differential system of  type $[p,q]$.
	Then  $ E^1_x(\CalI)$
	decomposes into a direct sum of $p$ and  $q$ dimensional subspaces
\begin{equation}
	 E^1_x(\CalI) = S_1 \oplus S_2
\EqTag{Dec1}
\end{equation}
	such that
\par
\smallskip
\noindent
{\bf[i]}
	every vector in $S_1$ and every vector in $S_2$ is singular; and
\par
\noindent
\smallskip
{\bf[ii]}
	every 2 plane spanned by any pair of vectors $X \in S_1$ and $Y\in S_2$ is an integral 2-plane for $\CalI$.
\end{Lemma}
\begin{proof}
Let $\{\, \partial_{\tilde \theta^i}, \vecthsigma,\vectcsigma \,\}$
	denote the dual frame to \EqRef{Intro4}. Then $E^1_x(\CalI) = \text{span}\{\, \vecthsigma,\vectcsigma \, \}$
	and the polar equations for  $X = X_1 + X_2$, where $X_1 =   t^a \vecthsigma$ and
	$X_2 =   s^\alpha\vectcsigma$, are
\begin{equation*}
	\hOmega^a(X_1, Y) = 0  \quad\text{and}\quad
	\cOmega^\alpha(X_2, Y)  =0.
\end{equation*}
        Then, clearly,  the vectors $X_1$ and $X_2$ are singular vectors
	and, moreover, the plane spanned by $X_1, X_2$ is an integral 2-plane.
\end{proof}
\begin{Remark}	
	 The differential system $\hCalV$  defined by \EqRef{Intro6} is the smallest differential system containing $\CalI$  and $\Ann(S_2)$
	and it is for this reason that
	we have opted to called  $\hCalV$ (and $\cCalV$) the  singular differential systems associated to the decomposition of $\CalI$.
\end{Remark}	
\par
	The necessary conditions for decomposability given in Lemma \StRef{Dec1} are seldom sufficient.
	However, if  $\CalI$ is a Pfaffian system and  properties {\bf [i]}  and {\bf [ii]} of Lemma \StRef{Dec1} are satisfied,
	then  there exists  a local coframe
\begin{equation*}
	\theta^1,\ldots, \theta^r,\  \hsigma^1,\dots,\hsigma^{p}, \csigma^1,\dots,\csigma^{q}
\end{equation*}
	on $M$ with $I = \{\, \theta^i  \,\}$  and with structure equations
\begin{equation}
	\extd \theta^i \equiv A^i_{ab}\,\hsigma^a\wedge \hsigma^b + B^i_{\alpha\beta} \,\csigma^\alpha \wedge \csigma^\beta\quad \mod I.
\EqTag{Dec2}
\end{equation}	
	At each point $x_0$ of $M$ define linear maps
\begin{equation}
	A = [A^i_{ab}(x_0)] \: \real^r \to \Lambda^2(\real^p)
	\quad
	\text{and}
	\quad
	B =[B^i_{\alpha\beta}(x_0)] \:\real^r \to \Lambda^2(\real^q).
\end{equation}

\begin{Theorem}
\StTag{Dec3}
	Let $I$ be a Pfaffian system. Then $I$ is decomposable if
	properties {\bf [i]}  and {\bf [ii]}  of  Lemma  \StRef{Dec1} are satisfied;
	the matrices $A^i_{ab}$ and $B^i_{\alpha\beta}$  given by \EqRef{Dec2},
	are non-zero, constant rank;  and
\begin{equation}
	\dim\big(\ker(A) + \ker(B)\big) = \rank I.
\EqTag{Dec4}
\end{equation}
\end{Theorem}
\begin{proof}
	Equation \EqRef{Dec4} implies that
	we can find $r$ linearly independent $r$ dimensional column vectors
	$T^1$,\ldots, $T^{r_1}$,\ldots, $T^{r_2}$, \ldots, $T^r$ such that
\begin{gather*}
	\ker(A) \cap \ker(B)
	= \text{span}\{\,T^1, \dots , T^{r_1} \, \},
\\
	\ker(B)
 	= \text{span}\{\,T^1, \dots, T^{r_1},T^{r_1+1},\dots ,T^{r_2} \,\}, \quad\text{and}
\\
	\ker(A)
 	= \text{span}\{\,T^1, \dots,  T^{r_1}, T^{r_2+1},\dots ,T^{r}\,\}.
\end{gather*}
        The 1-forms $\tilde \theta^i = t^i_j \,\theta^j$, where $T^i = [t^i_j]$,
	then satisfy
\begin{equation}
\begin{alignedat}{2}
	\extd \tilde\theta^i
&	\equiv
	0 \quad
&&	\text{for $i=1\dots r_1$,}
\\
	\extd \tilde \theta^i
&	\equiv  t^i_jA^j_{ab}\,\hsigma^a\wedge \hsigma^b\quad
&&	\text{for $i=r_1 +1\dots r_2$,}	
\\
	\extd \tilde \theta^i
&	\equiv t^i_j B^j_{\alpha\beta} \,\csigma^\alpha \wedge \csigma^\beta
	\quad
&&	\text{for $i=r_2 +1\dots r$,}
\end{alignedat}
\EqTag{Dec5}
\end{equation}
	and the decomposability of $I$ follows by taking $\hOmega^i =t^i_jA^j_{ab}\,\hsigma^a\wedge \hsigma^b$ for $i=r_1 +1\dots r_2$
	and $\cOmega^i =t^i_jB^j_{\alpha\beta}\,\csigma^a\wedge \csigma^b$ for $i=r_2 +1 \dots r$.
	That the integers  $s$ and $t$ in Definition \StRef{Intro1}  satisfy   $s,t  \geq 1$ follows from the fact that
	$A^i_{ab}$ and $B^i_{\alpha\beta}$  are non-zero and that $[t^i_j]$ is invertible.
\end{proof}
\par
\begin{Corollary}
\StTag{Dec4}
	If $\CalI$ is a decomposable Pfaffian system, then the singular systems $\hCalV$ and $\cCalV$ are
	also Pfaffian.
\end{Corollary}

\begin{proof} This is immediate from \EqRef{Intro6} and \EqRef{Dec5}.
\end{proof}

	To each decomposable exterior differential system  $\CalI$ we have associated a
	pair of differential systems $\hCalV$ and $\cCalV$ and, when these are Pfaffian,  we
	have defined what it means for $\{\, \hV, \cV\, \}$ to define a Darboux pair.
	In our subsequent analysis, the Darboux pair  $\{\, \hV, \cV\, \}$ will be taken as the fundamental object of study.
	From this viewpoint, it becomes  important to address the problem of
	reconstructing the EDS $\CalI$ from its singular systems.
	To this end, we introduce the following novel construction.
\begin{Definition}
\StTag{Dec5}
	Let $\hCalV$ and $\cCalV$ be two differential systems on a manifold $M$.  Define a new  differential system
	$\CalK = \hCalV \mycap \cCalV$ to be the exterior differential system whose integral elements are precisely
	the integral elements of $\hCalV$,
	the integral elements of $\cCalV$, and  the sum of
	integral elements of $\hCalV$ with  integral elements of $\cCalV$.
\end{Definition}
\par
	This definition is motivated by two observations. First we note that  if $\CalI$ is
	a decomposable exterior differential system  with singular Pfaffian systems $\hV$ and $\cV$
	then in general $\CalI$
	properly contains the Pfaffian system with generators $\hV \cap \cV$ and is
	properly contained in the EDS   $\hCalV  \cap \cCalV$ so that neither of these set-theoretic constructions reproduce $\CalI$.
\par	
	Secondly, one can interprete the EDS $\CalK$ in Definition \StRef{Dec5} as
	satisfying an {\deffont infinitesimal superposition principle} with respect to
	$\{\, \hCalV, \cCalV \,\}$ at the level of integral elements. Of course, this by itself
	does not imply that $\CalK$ admits a superposition principle for its integral manifolds. The
	main results of this article  can then be re-formulated as  follows: if the EDS $\CalK$ admits
	a superposition principle for its integral elements with respect to a Darboux pair
	$\{\, \hV,  \cV\,\}$, then $\CalK$ admits a superposition principle for its integral manifolds.
\par
\begin{Theorem}
\StTag{Dec6}
	If $\CalI$ is a decomposable exterior differential system with
	singular differential systems  $\hCalV$ and $\cCalV$,  then $\CalI = \hCalV \mycap \cCalV$.
\end{Theorem}
\begin{proof}  We first remark that the differential system $\hCalV \mycap \cCalV$ is always contained in $\hCalV \cap \cCalV$.
	Let $\{\,\tilde \theta^i,\, \hsigma^a,\, \csigma^\alpha\, \}$ be the coframe given by Definition \StRef{Intro1}.
	By definition,  the algebraic generators for $\hCalV$ and $\cCalV$
	are
\begin{equation*}
	\hCalV  = \{\, \tilde \theta^i,\, \hsigma^a,\, \cOmega^\beta \, \}
	\quad\text{and}\quad
	\cCalV 	= \{\, \tilde \theta,\, \csigma^\alpha,\, \hOmega^b\,  \}
\end{equation*}
	from which it is not difficult to show that
\begin{equation*}
	\hCalV \cap \cCalV = \{\, \tilde \theta^i,\, \hOmega^b,\, \cOmega^\beta,\, \hsigma^a \wedge \csigma^\alpha\} .
\end{equation*}
	Since $\partial_{\csigma^\alpha}$ defines a 1-integral element for $\hCalV$ and
	$\partial_{\hsigma^\alpha}$ defines a 1-integral element for $\cCalV$,  the 2-plane spanned by
	$\{\partial_{\csigma^\alpha},\, \partial_{\hsigma^a} \}$ must be a integral element  for  $\hCalV \mycap \cCalV$
	and therefore the forms  $\hsigma^a \wedge \csigma^\alpha \notin   \hCalV \mycap \cCalV$. In view of \EqRef{Intro5},
	we conclude that   $\hCalV \mycap \cCalV \subset \CalI$.
	
	To complete the proof, it suffices to check that every $(r+s)$-form $\omega\in \CalI$
	vanishes on every  plane $\hE_r + \cE_s$, where  $\hE_r$ is an $r$-integral plane of $\hCalV$ and
	$\cE_s$ is an $s$-integral plane of $\cCalV$. The form $\omega$ is a linear combination of the wedge product of the 1-forms in
	\EqRef{Intro5}
	with arbitrary  $(r + s -1)$-forms
	$\rho$ and the  wedge product of the 2-forms in
	\EqRef{Intro5} with arbitrary $(r+s -2)$-forms $\tau$.
	It is clear that  $(\tilde \theta^i \wedge \rho)_{|(\hat E_r + \check E_s)} = 0$. The values of
	$(\hOmega^b \wedge \tau)_{|(\hat E_r + \check E_s)}$ can be calculated as linear combinations of terms involving
\begin{equation*}
	\hOmega^b(\hX, \hY), \quad
	\hOmega^b(\hX, \cY), \quad
	\hOmega^b(\cX, \hY),\quad
	\hOmega^b(\cX, \cY),
\end{equation*}
	where  $\hX, \hY \in \hE_r$ and $\cX, \cY \in \cE_s$.  These terms	
	all vanish and in this way we may complete the proof of the theorem.
\end{proof}

\par
	
\subsection{The first adapted coframes for a Darboux pair}		
	Let \{\,$\hV, \cV \,\}$ be a Darboux pair on a manifold $M$. In addition to the properties listed in
	Definition \StRef{Intro2}, we shall assume that the
	derived systems  $\hV^{(\infty)}$ and  $\cV^{(\infty)}$, as well as the intersections
	$\hV^{(\infty)} \cap\cV$, $\hV \cap \cV^{(\infty)}$ and
	$\hV \cap \cV$ are all (constant rank) subbundles of $T^*\kern -2ptM$.
\par
	To define our initial adapted coframe for the Darboux pair $\{\, \hV, \cV\, \}$,
	we first chose tuples of independent 1-forms $\bfheta$ and $\bfceta$ which satisfy
\begin{subequations}
\label{eq:First4}
\begin{equation}
	\hV^{(\infty)} \cap\cV =  \text{span} \{\bfheta\}
	\quad\text{and}\quad
	\hV \cap\cV^{(\infty)} = \text{span}\{\bfceta\}.
\EqTag{First4a}
\end{equation}
	Property \EqRef{Intro8} implies that the forms $\{\, \bfheta,\, \bfceta\,\}$ are linearly independent.
	Then chose tuples of 1-forms $\bftheta$, $\bfhsigma$ and  $\bfcsigma$  such that
\begin{equation}
\begin{gathered}
	\hV^{(\infty)} = \text{span} \{\, \bfhsigma,\, \bfheta \,\},
	\
	\cV^{(\infty)} = \text{span} \{\, \bfcsigma,\, \bfceta \,\},
	\ 
	\hV \cap \cV  = \text{span} \{\,\bftheta,\, \bfheta,\,\bfceta \,\}
\end{gathered}
\EqTag{First4b}
\end{equation}
	and such that the sets of 1-forms $\{\, \bfhsigma,\, \bfheta \,\}$, $\{\, \bfcsigma,\, \bfceta \,\}$  and
	$\{\,\bftheta,\, \bfheta,\,\bfceta \,\}$ are linearly independent.

\begin{Lemma}
\StTag{First15}
	Let \{\,$\hV, \cV \,\}$ be a Darboux pair on a manifold $M$. Then the
	1-forms $\{\,\bftheta,\, \bfhsigma,\, \bfheta,\, \bfcsigma, \,\bfceta \, \}$
	satisfying \EqRef{First4}
	define a local coframe for $T^*\kern -2pt M$ and
\begin{equation}
	\hV = \text{\rm span} \{\,\bftheta,\, \bfhsigma,\, \bfheta,\,\bfceta \,\}
	\quad\text{and}\quad
	\cV  = \text{\rm span} \{\,\bftheta,\,  \bfheta,\, \bfcsigma,\,\bfceta \,\}.
\EqTag{First4c}
\end{equation}
\end{Lemma}
\end{subequations}
\begin{proof}
	From the definitions \EqRef{First4a} and \EqRef{First4b} it is clear that
	$\text{span} \{\,\bftheta,\, \bfhsigma,\, \bfheta,\,\bfceta \,\} \subset \hV$  and
	$\text{span} \{\,\bftheta,\,  \bfheta,\, \bfcsigma,\,\bfceta \,\} \subset \cV$.
	Property \EqRef{Intro7} then implies that the forms $\{\,\bftheta,\, \bfhsigma,\, \bfheta,\, \bfcsigma, \,\bfceta \, \}$
	span  $T^*\kern -2pt M$.

	To prove the first of \EqRef{First4c}, let
	$\omega = \mathbf a \,  \bftheta + \mathbf b \, \bfhsigma + \mathbf c \,\bfheta + \mathbf d \, \bfcsigma + \mathbf e \,\bfceta $.
	If  $\omega \in \hV$, then the 1-forms $ \mathbf  d \,\bfcsigma \in \hV$ and therefore,  by $\EqRef{First4b}$,
	$\mathbf d \,\bfcsigma \in \hV  \cap \cV^{(\infty)}$.
	Since the forms $\bfceta$ and $\bfcsigma$ are independent, we must have $\mathbf d = 0$ and \EqRef{First4c} is established.
	
	To prove the independence of the forms $\{\,\bftheta,\, \bfhsigma,\, \bfheta,\, \bfcsigma, \,\bfceta \, \}$, set $\omega = 0$.
	Then the argument just given  implies that $\mathbf d = 0$ so that $\mathbf b \bfhsigma \in \hV \cap \cV$. Therefore
	$\mathbf b \bfhsigma \in \hV^{(\infty)} \cap \cV$ and this forces $\mathbf b =0$. The independence of the forms
	$\{\,\bftheta,\, \bfheta,\,\bfceta \,\}$ then gives $ \mathbf a =   \mathbf c = \mathbf e = 0$ and the lemma is proved.
\end{proof}
  	Any local coframe $\{\,\bftheta,\, \bfhsigma,\, \bfheta,\, \bfcsigma, \,\bfceta \, \}$  satisfying \EqRef{First4} is called a
	{\deffont 0-adapted coframe for the Darboux pair $\{\hV,\, \cV\}$.}

	We  shall also need the  definition of a 0-adapted frame as the dual of the 0-adapted coframe. Let
\begin{equation}
\begin{gathered}
	\hH = \ann\  \hV,  \quad
	\cH = \ann\  \cV,  \quad
	\hH ^{(\infty)} =  \ann\  \hV^{(\infty)}, \quad
	\cH ^{(\infty)} =  \ann\ \cV^{(\infty)},
\\[2\jot]
	\text{and}\qquad	
	K =  \hH ^{(\infty)}  \cap \cH ^{(\infty)}.
\end{gathered}
\EqTag{First5}
\end{equation}
	The definition of Darboux integrability can be re-formulated in terms of the distributions $\hH, \cH$
	(\cite{vassiliou:2001a}). If we introduce the dual basis
	$\{\,  \bfvecttheta, \,\bfvecthsigma,\, \bfvectheta,\, \bfvectcsigma,\, \bfvectceta \, \}$
	to the 0-adapted coframe
	$\{\bftheta, \bfhsigma, \bfheta, \bfcsigma, \bfceta \}$, that is,
\begin{equation}
	\bftheta(\bfvecttheta) = \bfone, \quad \bfhsigma(\bfvecthsigma)= \bfone ,\quad \bfheta(\bfvectheta)= \bfone,
	\quad \bfcsigma(\bfvectcsigma)= \bfone, \quad\bfceta(\bfvectceta) = \bfone ,
\end{equation}
	with all others pairings yielding 0, then it readily follows that
\begin{equation}
\begin{gathered}
\begin{aligned}
	\hH &= \text{span} \{\, \bfvectcsigma \, \},
	&\quad
	\hH^{(\infty)} &= \text{span} \{\, \bfvecttheta, \, \bfvectcsigma, \, \bfvectceta \, \},
\\
	\cH &= \text{span} \{\bfvecthsigma\},
	&\quad
	\cH^{(\infty)} & = \text{span} \{\, \bfvecttheta, \, \bfvecthsigma, \, \bfvectheta \,\},
\end{aligned}
\\
	\text{and}\quad K = \text{span} \{\, \bfvecttheta \, \}.
\end{gathered}
\EqTag{First6}
\end{equation}
	Any local frame $\{\,  \bfvecttheta, \,\bfvecthsigma,\, \bfvectheta,\, \bfvectcsigma,\, \bfvectceta \, \}$  satisfying \EqRef{First6}
        is called a 0-adapted frame.
\par	
	We shall use the following lemma repeatedly.
\begin{Lemma}
\StTag{FirstLem1}
	Let $f$ be a real-valued function on $M$.
	If $X(f) = 0$ for all vector fields in $\hH$, then
	$d f  \in  \hV^{(\infty)}$. Likewise, if $X(f) = 0$ for all vector fields in $\cH$, then
	$d f  \in  \cV^{(\infty)}$.
\end{Lemma}
\begin{proof}
	It suffices to note that if  $X(f) = 0$ for all vector fields $X \in\hH$,
	then $Z(f) = Z \hook d f = 0$
	for all vector fields $Z \in \hH^{(\infty)}$ and therefore
	$d f \in\ann\ \hH^{(\infty)} = \hV^{(\infty)}$.
\end{proof}
\par
	
	Real-valued functions on $M$ with $df \in V$ (or $V^{(\infty)}$) are called
	{\deffont first integrals} for $V$. We denote the algebra of all first integrals for $V$ by $\Inv(V)$.

	Our 1-adapted coframe for a Darboux pair is easily constructed from complete sets
	of functionally independent first integrals for $\hV^{(\infty)}$
        and $\cV^{(\infty)}$, that is, locally defined functions  $\{\,\bfhI \,  \}$  and $\{\,\bfcI \, \}$ such that
\begin{equation*}
	\hV^{(\infty)} = \text{span}\{\,\extd \bfhI \,\}
	\quad\text{and}\quad
	\cV^{(\infty)} = \text{span}\{\,\extd \bfcI \,\}.
\end{equation*}
	Complete the 1-forms $\{\, \bfheta \, \}$ in \EqRef{First4a} to the local basis
	\EqRef{First4b} for $\hV^{(\infty)}$ using forms
\begin{equation}
	\bfhsigma = \extd \bfhI_1
\EqTag{First14}
\end{equation}
	chosen from the  $\{\,\extd \bfhI \,\}$.
	Next let  $\{\,\bfhI_2 \,\}$ be a complementary set of invariants to the set $\{\, \bfhI_1\, \}$, that is,
	$\{\,\bfhI \} = \{\,\bfhI_1,\, \bfhI_2\, \}$.  Because the forms $\bfheta$ belong to
	$\hV^{\infty}$ we can write
\begin{equation*}
	\bfheta =  \bfhR_0\, \extd \bfhI_1 + \bfhS_0\, \extd \bfhI_2.
\end{equation*}
	Since the 1-forms  $\bfheta$ are independent of the 1-forms $\bfhsigma$, the coefficient matrix $\bfhS_0$
	must be invertible and we can therefore adjust our local
	basis of 1-forms  $\{\, \bfheta\,\}$ for $\hV^{(\infty)} \cap \cV$ by setting
\begin{equation}
	\bfheta = \extd \bfhI_2 + \bfhR\, \extd \bfhI_1 = \extd \bfhI_2 + \bfhR\, \bfhsigma.
\EqTag{First15}
\end{equation}
        The  exterior derivatives of these forms are
\begin{equation*}
	\extd \bfheta = \extd \bfhR  \wedge  \bfhsigma = \bfvectcsigma(\bfhR)\,  \bfcsigma \wedge \bfhsigma + \cdots
\end{equation*}
        and therefore, on account of \EqRef{Intro9}, the fact that $\bfheta \in  \hV^{(\infty)} \cap \cV$, and \EqRef{First4c},
	we must have $\bfvectcsigma(\bfhR)  = 0$. Lemma \StRef{FirstLem1}
        implies that $\extd (\bfhR) \in \hV^{(\infty)}$ and consequently
\begin{equation}
	\extd \bfheta =  \bfhE \, \bfheta \wedge  \bfhsigma     + \bfhF  \bfhsigma \wedge  \bfhsigma.
\end{equation}
	Similar arguments are used to modify the  forms  $\bfcsigma$ and $\bfceta$.
	These adaptations are summarized in the  following theorem.
\begin{Theorem}
\StTag{First9}
	Let $\{\, \hV, \cV\}$ be a Darboux pair on  a manifold $M$. Then about each point of $M$ there exists a
	0-adapted coframe   $\{\,\bftheta,\, \bfhsigma,\, \bfheta,\, \bfcsigma, \,\bfceta \, \}$
	with structure equations  	
\begin{gather}
\begin{aligned}
	\extd \bfhsigma
	&= 0,
	&\quad
	\extd \bfheta
	&=  \bfhE \, \bfheta \wedge  \bfhsigma    + \bfhF  \bfhsigma \wedge  \bfhsigma,
\\[1\jot]
	\extd \bfcsigma
	&=0,
	&\quad
	\extd \bfceta
	&=  \bfcE\,  \bfceta \wedge  \bfcsigma     + \bfcF \, \bfcsigma \wedge  \bfcsigma,
\end{aligned}
\EqTag{First10}
\\[1\jot]
	\extd \bftheta
	\equiv
	\bfA\, \bfhsigma \wedge \bfhsigma   +
	\bfB\, \bfcsigma    \wedge \bfcsigma
\quad
	\mod \{\bftheta,\, \bfheta,\,   \bfceta \}.
\EqTag{First11}
\end{gather}
\end{Theorem}
	We remark that the structure equations for the 1-forms  $\bftheta \in \hV \cap \cV$ are a consequence of
	\EqRef{Intro9} and the properties \EqRef{First4} of a 0-adapted coframe.
	A {\deffont 1-adapted coframe} for the Darboux pair $\{\, \hV, \cV\, \}$  is a 0-adapted coframe
        $\{\,\bftheta,\, \bfhsigma,\, \bfheta,\, \bfcsigma, \,\bfceta \, \}$  satisfying the structure equations
	 \EqRef{First10} and \EqRef{First11}.
\begin{Corollary}
\StTag{First20}
	If $\{\,\hV,\, \cV\}$ define a Darboux pair, then $\hCalV \mycap \cCalV$ is always
	a decomposable exterior differential system   for which  $\hV$ and  $\cV$ are  singular Pfaffian systems.
\end{Corollary}
\begin{proof}
	
	It follows from Theorem \StRef{First9} that the algebraic generators for $\hCalV$ and $\cCalV$ are
\begin{align*}
	\hCalV
	 & = \{\, \bftheta,\ \bfheta,\ \bfceta,\  \bfhsigma,\ \bfcF \, \bfcsigma \wedge  \bfcsigma,\ \bfB\,\bfcsigma    \wedge \bfcsigma \,\}
	\quad\text{and}
\\[2\jot]
	\cCalV
	 & = \{\, \bftheta,\ \bfheta,\ \bfceta,\   \bfcsigma,\ \bfhF\, \bfhsigma \wedge \bfhsigma,\
	\bfA\, \bfhsigma \wedge \bfhsigma \,\}.
\end{align*}
	From these equations it is not difficult to argue that the differential system  $\hCalV \cap \cCalV$ has generators
\begin{equation*}
	\hCalV \cap \cCalV
	=  \{\, \bftheta,\ \bfheta,\ \bfceta,\  \bfhsigma \wedge  \bfcsigma,\
	\bfhF\, \bfhsigma \wedge \bfhsigma,\   \bfA\, \bfhsigma \wedge \bfhsigma,\
	\bfcF \, \bfcsigma \wedge  \bfcsigma,\ \bfB\, \bfcsigma    \wedge \bfcsigma \,\}.
\EqTag{First32}
\end{equation*}
	A repetition of the arguments used in the proof of  Theorem \StRef{Dec6} shows that
\begin{align}
	\hCalV \mycap \cCalV
	= \{\, \bftheta,\ \bfheta,\ \bfceta,\  \bfhF\, \bfhsigma \wedge \bfhsigma,\
	\bfA\, \bfhsigma \wedge \bfhsigma,\
	\bfcF \, \bfcsigma \wedge  \bfcsigma,\ \bfB\, \bfcsigma    \wedge \bfcsigma \,\}.
\EqTag{First20}
\end{align}
	It is then clear that $\hCalV \mycap \cCalV$ is a decomposable differential system (Definition \StRef{Intro1})
	and that $\hV$ and $\cV$ are singular Pfaffian systems.
\end{proof}

\begin{Example}
        There are many examples of differential equations which can be described either by Pfaffian differential systems or by
	differential systems generated by 1-forms and 2-forms.  Our definition of  Darboux integrability in terms
of decomposable differential system is such that theseequations
	will be Darboux integrable irrespective of their formulation as an exterior differential system.
	The  simplest example, but by no means the only
	example, is  the  wave equation $u_{xy}= 0$. The wave equation can be  encoded as adifferential system on manifold of dimensions 5,  6 and 7 by
\begin{alignat*}{3}
	\CalI_1
&	= \{\, \theta, dp \wedge dx,  dq \wedge dy \, \}, \
&
	\hV_1
& 	= \{\, \theta,  dp,  dx \, \},
&
\cV
& 	= \{\, \theta,  dq,  dy \, \},
\\
	\CalI_2
&	= \{\, \theta,  \theta_x,  dq \wedge dy \, \}, \
&
	\hV_2
& 	= \{\, \theta,   \theta_x, dp,  dx \, \},
&	
	\cV_2
& 	= \{\, \theta,   \theta_x, dq,  dy \, \},
\\
	\CalI_3
&	= \{\, \theta, \theta_x, \theta_y \, \},
&
	\hV_3
& 	= \{\, \theta,   \theta_x, \theta_y, dp,  dx \, \}, \
&	
	\cV_3
& 	= \{\, \theta,   \theta_x, \theta_y, dq,  dy \, \},
\end{alignat*}
	where $\theta = du - p\, dx - q\, dy$, $\theta_x = dp - r\, dx$ and $\theta_y = dq - t\, dy$.
	Each of these satisfy the definition of a  decomposable, Darboux integrable differential system.
	Consistent with Corollary
	\StRef{Dec6}, one easily checks that $\CalI_i = \hCalV_i \mycap \cCalV_i$.

	We remark that for the standard contact system on $J^1(\Real,\Real)$, that is,
	$\CalI_4 =  \{\, \theta,\, dp \wedge dx +  dq \wedge dy \, \}$, the Pfaffian systems
	$\hV_4 = \{\, \theta,\,  dp,\,  dx \, \}$ and
	$ \cV = \{\, \theta, \, dq,\,  dy \, \}$,
	are singular systems  which form a Darboux pair but $\hV_4 \mycap \cV_4 \neq  \CalI_4$.
	Equality fails because $\CalI_4$ is not decomposable. \qed
\end{Example}

\begin{Remark}
\StTag{FirstRem1}
	In very special cases, such as Liouville's equation $u_{xy}= e^u$, one finds that
	$\hV \cap \cV^{(\infty)}=  \hV^{(\infty)}  \cap \cV = \{\,0\,\}$
	so that the vector space sums appearing in \EqRef{Intro7} are direct sum decompositions.
	Under these circumstances, our 0-adapted coframe is simply $\{\,\bftheta, \bfhsigma, \bfcsigma \,\}$.	
	This initial coframe  is  then automatically 2-adapted  (see Theorem \StRef{SecondThm1}) and the
	sequence of coframe adaptations needed to prove  Theorem \StRef{CofrAdapt1}
	can begin with the third coframe adaptation (Section 4.2).
\end{Remark}	
\begin{Remark}
\StTag{FirstRem2}
	We define an {\deffont involution of  a  Darboux pair} $\{\, \hV, \cV \, \}$
	on $M$ to be a diffeomorphism $\Phi \: M \to M$ such that $\Phi^2= \text{\rm id}_M$
	and $\Phi^*(\hV) = \cV$.  For such maps $\Phi^*(\hV^{(\infty)}) = \cV^{(\infty)}$ and therefore, given 1 adapted co-frame elements
	$\bfheta$ and $\bfhsigma$, one may take
\begin{equation*}
	\bfceta = \Phi^*(\bfheta)
	\quad\text{and}\quad
	\bfcsigma = \Phi^*(\bfhsigma).
\end{equation*}
	Such involutions are quite common, especially for differential equations arising in geometric and physical applications, where
	the diffeomorphism $\Phi$ arises from a symmetry of the differential equations in the independent variables.
\end{Remark}
%

\newpage

\section{The group theoretic construction of Darboux pairs}

\subsection{ Direct Sums of Pfaffian Systems and Trivial Darboux Pairs}
	Let $\CalW_1$ and $\CalW_2$ be differential systems on manifolds $M_1$ and $M_2$.  Then the {\deffont direct sum } $\CalW_1 + \CalW_2$
	is the EDS on $M_1 \times M_2$ generated by the pullbacks of $\CalW_1$ and $\CalW_2$ by the canonical projection maps
	$\pi_i \colon M_1\times M_2 \to M_i$.
        The following theorem  shows that if
	$\CalW_1$ and $\CalW_2$ are Pfaffian, then $\CalW_1 + \CalW_2$ is Darboux integrable in the sense of Definition \StRef{Intro4}.

\begin{Theorem}
\StTag{Trivial1}
	Let $W_1$ and $W_2$ be Pfaffian systems on manifolds $M_1$ and $M_2$
	such that $W_1^{(\infty)} = W_2^{(\infty)} = 0$.  Then the direct sum
	$J= W_1 + W_2$ on $M_1\times M_2$ is a decomposable 	
	Pfaffian system. The  associated singular Pfaffian systems
\begin{equation}
	\hW = W_1 \oplus\Lambda^1(M_2)
	\quad\text{and}\quad
	\cW = \Lambda^1(M_1) \oplus W_2
\EqTag{Trivial2}
\end{equation}
	form a Darboux pair and
\begin{equation}
	\CalW_1 +\CalW_2 = \hCalW \mycap \cCalW,
\EqTag{Trivial15}
\end{equation}	
	where $\CalW_1, \CalW_2, \hCalW, \cCalW$ are the differential systems generated by $W_1, W_2, \hW, \cW$.
\end{Theorem}
\begin{proof}  It is easy to check that $J$ is decomposable with singular systems \EqRef{Trivial2}.
	Properties {\bf [i]} and {\bf [ii]}  of  Definition  \StRef{Intro2}
        follow directly from \EqRef{Trivial2} and the simple observations that
\begin{equation}
	\hW^{(\infty)} = \Lambda^1(M_2) \quad\text{and}\quad \cW^{(\infty)} = \Lambda^1(M_1).
\EqTag{Trivial4}
\end{equation}
\par
	Let $\omega = \sum_{i}(f_i \, \theta^i_1 + g_i\, \theta^i_2)$,
	where the coefficients $f_i, g_i \in C^{\infty}(M_1\times M_2)$ and the $\theta^i_\ell \in W_\ell$ are
	1-forms in $J$.
	Then to check property  {\bf [iii]}, it suffices to note
	that every summand  in the  exterior derivative
\begin{align*}
	\extd \omega
	=    \sum_{i}d_{1}f_i \wedge \theta^i_1+ d_{2}f_i \wedge \theta^i_1+  f_i \,d_{1}\theta^i_1
   +           d_{1}g_j\wedge \theta^j_2 + d_{2} g_j\wedge  \theta^j_2 + g_i\, d_{2}\theta^i_2,
\end{align*}
	where $d_i = d_{M_i}$, lies either in $\Omega^2(\hW)$ or $\Omega^2(\cW)$. Equation \EqRef{Trivial15} is a
	direct consequence of Theorem \StRef{Dec6}.
\end{proof}
	Darboux pairs of the form \EqRef{Trivial2} are said to be trivial.

\subsection{Symmetry Reduction of Darboux Pairs}

        A Lie group $G$ is  a  {\deffont regular symmetry group of a  differential system $\CalI$} on a manifold $M$ if
	there is a regular  action of  $G$ on $M$ which preserves $\CalI$.
	By the definition of regularity, the quotient space $M/G$  of $M$ by the orbits of $G$
	has a smooth manifold structure such that {\bf [i]} the  projection map $\bfq\colon M \to M/G$ is a
	smooth submersion,  and {\bf [ii]}   local, smooth  cross-sections to $\bfq$  exist
	on some open neighborhood of each point of $M/G$.
	 The {\deffont $G$ reduction of $\CalI$ or
	the quotient of $\CalI$ by $G$} is the differential system on $\barM = M/G$  given by
\begin{equation}
         \CalI/G =  \{\, \omega \in \Omega^*(\barM) \, | \, \bfq^*(\omega) \in \CalI \, \}.
\EqTag{SymRed1}
\end{equation}	
	Details of this construction  and basis facts regarding the $G$ reduction of differential systems are  given in \cite{anderson-fels:2005a}.

	To obtain the main results of this section,  we shall require that $G$ be {\deffont transverse} to $\CalI$ in the sense which we now define.
	Let  $I$ be any $G$ invariant sub-bundle of $T^*M$ and let  $\Gamma$ be the infinitesimal generators for $G$.
	We say that $G$ is transverse to $I$ if  for  each $x\in M$
\begin{equation}
	\Gamma_x  \cap \ann\, I_x  =\{\, 0 \, \}.
\EqTag{SymRed2}
\end{equation}
	We say that $G$ is transverse to the differential system $\CalI$ if $G$ is transverse to the sub-bundle of $T^*M$ whose sections
	are 1-forms in  $\CalI$.  With  this transversality condition, $\CalI/G$ is assured to be a constant rank differential system whose local sections
	can be determined as follows. If $A$ is any sub-bundle of $\Lambda^k(M)$ or sub-algebra of $\Lambda^*(M)$, then
	the sub-bundle or  sub-algebra of {\deffont semi-basic forms in $A$} is defined as
\begin{equation}
	A_{\semibasic}  = \{ \, \omega \in A \, |\, X\hook \omega = 0 \text{ for all $X \in \Gamma$} \,\}.
\end{equation}
	Denote by $\CalI_\semibasic$ the sections of $\CalI$ which are semi-basic, that is, which take values in $\Lambda^*_{\semibasic}(M)$.
	Then one can show (see Lemma 4.1 of \cite{anderson-fels:2005a}) that
\newcommand\smallbarU{
	\hbox{\kern .8 true pt
	\vbox{\hrule width 4.5  true pt height .15 true pt \kern .9 true pt
	\hbox{\kern -.8 true pt \scriptsize $U$}}}}
\begin{equation}
	(\CalI/G)_{|\smallbarU} =  \xi^*(\CalI_\semibasic{}_{|U}),
\EqTag{SymRed3}
\end{equation}
	where $\barU$ is any (sufficiently small) open set of $\barM$, $\xi\:\barU \to M$ is a smooth cross-section of $\bfq$, and
	$U =\bfq^{-1}(\barU)$.  We remark that the $G$ reduction $\CalI/G$ of a Pfaffian system $\CalI$ need not be Pfaffian (but see {\bf [c]} below).

	For the proof of the next theorem  on the $G$ reduction of Darboux pairs,
	we shall require the following elementary facts regarding the reduction of Pfaffian systems.
	Here $I$ and $J$ are Pfaffian systems with regular symmetry  group $G$,  we  assume that $G$ is transverse to $I$, and
	the  $G$ bundle reduction of $I$ is defined as
\begin{equation}
	I/G =  \{\, \omega \in \Lambda^1(\barM) \, | \, \bfq^*(\omega) \in I \, \}.
\EqTag{SymRed4}
\end{equation}
	By transversality, this is a sub-bundle of $T^*\barM$ of dimension $\dim I - q$, where $q$ is the dimension of the orbits of $G$.

\par
\medskip
\noindent
{\bf [a]} If $I\subset J$, then $G$ is transverse to $J$.
\par
\par
\medskip
\noindent
{\bf [b]} If $G$ is transverse to $I \cap J$,  then $(I/G) \cap (J/G)  = (I \cap J)/G$.
\par
\medskip
\noindent
{\bf [c]} If $G$ is transverse to the derived system $I'$, then the quotient differential system $\CalI/G$ is the constant rank
	Pfaffian system determined by $I/G$ and
	$(I/G)' = I'/G$ .
\medskip
\par
\noindent
{\bf [d]} If $G$ is transverse to  $I^{(\infty)}$, then $(I/G)^{(\infty)} =  I^{(\infty)}/G $.
\medskip
\par
	Facts {\bf [a]} and {\bf [b]} are trivial,  while {\bf [c]}
	is  Theorem 5.1 in \cite{anderson-fels:2005a}. Fact {\bf [d]} follows from {\bf [a]} and {\bf [c]} by induction.
	
	Finally, we shall also  use the following  technical observation.   If  $Z$ is a sub-bundle of $I$,  then  the canonical pairing
	$(X, \omega) =  \omega(X)$  on $\Gamma_x \times Z_x$, where $x\in M$, is non-degenerate if and only if
\begin{equation}
	I = \bfq^*(I/G) \oplus Z.
\EqTag{SymRed5}
\end{equation}
	Such bundles $Z$ can always be constructed locally on $G$ invariant open subsets of $M$.

\newcommand\barhV{\accentset{ \raise 2.5pt \vbox{\hrule width 6.5  true pt height .3 true pt}    }{$\hV$}}
\newcommand\barcV{\accentset{ \raise 2.5pt \vbox{\hrule width 6.5  true pt height .3 true pt}    }{$\cV$}}
\newcommand\barhCalV{\accentset{ \raise 2.5pt \vbox{\hrule width 6.5  true pt height .3 true pt}    }{$\hCalV$}}
\newcommand\barcCalV{\accentset{ \raise 2.5pt \vbox{\hrule width 6.5  true pt height .3 true pt}    }{$\cCalV$}}

\newcommand\barhW{\accentset{ \raise 2.5pt \vbox{\hrule width 6.5  true pt height .3 true pt}    }{$\hW$}}
\newcommand\barcW{\accentset{ \raise 2.5pt \vbox{\hrule width 6.5  true pt height .3 true pt}    }{$\cW$}}
\newcommand\barhCalW{\accentset{ \raise 2.5pt \vbox{\hrule width 6.5  true pt height .3 true pt}    }{$\hCalW$}}
\newcommand\barcCalW{\accentset{ \raise 2.5pt \vbox{\hrule width 6.5  true pt height .3 true pt}    }{$\cCalW$}}

\begin{Theorem}
\StTag{SymRed1}
	Let  $\{\,\hW, \cW \, \}$ be a Darboux pair on $M$. Suppose that $G$ is a regular symmetry group  of $\hW$ and $\cW$ and that,
	in addition, $G$ is transverse to  $\hW \cap \cW^{(\infty)}$ and $\hW^{(\infty)} \cap   \cW$.  	Then the quotient
	differential systems $\{\,\hCalW/G, \cCalW/G \, \}$ on $\barM$ are (constant rank) Pfaffian systems and form a Darboux pair.
\end{Theorem}

\begin{proof}[Proof of Theorem \StRef{SymRed1}]
	
	Since $\hW^{(\infty)} \cap \cW \subset \hW^{(\infty)} \subset\hW' \subset \hW$ and $G$ is transverse to $\hW^{(\infty)} \cap \cW$,
	it follows that $G$ is transverse to $\hW^{(\infty)}$, $\hW'$ and $\hW$.  Thus,  by {\bf [c]},
	$\hCalV= \hCalW/G$ and, similarly,  $\cCalV = \cCalW/G$ are the Pfaffian systems for $\hV = \hW/G$ and $\cV = \cW/G$.
	By {\bf [d]}
\begin{equation}
	(\hW/G)^{(\infty)} = \hW^{(\infty)}/G
	\quad\text{and}\quad
	(\cW/G)^{(\infty)} = \cW^{(\infty)}/G.
\EqTag{SymReg6}
\end{equation}

	We now verify the three conditions in Definition \StRef{Intro2} for $\{\,\hV, \cV\,\}$ to be a Darboux pair.
	From \EqRef{SymRed4}, \EqRef{SymReg6} and  fact {\bf [b]}
	we deduce that
\begin{align*}
	\dim(\hV)
	&= \dim \hW - q,
\\
	\dim(\hV^{(\infty)}) &= \dim  (\hW^{(\infty)}/G)  = \dim \hW^{(\infty)} - q, \quad \text{and}
\\
	\dim (\hV \cap \hV^{(\infty)})
	&= \dim\big((\hW/G) \cap  (\cW^{(\infty)}/G)\big) = \dim \big((\hW \cap \cW^{(\infty)})/G \big)
\\
	&= \dim(\hW \cap \cW^{(\infty)}) - q.
\end{align*}
	From  these equations and  property {\bf [i]} in the definition of Darboux pair (applied to $\{\,\hW, \cW\,\}$) we calculate
\begin{align*}
	\dim\big(\hV   + \cV^{(\infty)}\big)
&	= \dim(\hV)  +\dim (\cV^{(\infty)}) - \dim(\hV \cap \cV^{(\infty)})
\\
&	= \dim\hW  +\dim \cW^{(\infty)} + \dim(\hW \cap \cW^{(\infty)})  - q
\\
&	= \dim M - q = \dim \barM.
\end{align*}
	This proves that  $\hV   + \cV^{(\infty)} = T^*\barM$. Similar arguments yield  $\hV^{(\infty)} + \cV   = T^*\barM$
	and so $\{\hV, \cV\}$ satisfy property {\bf[i]} of a Darboux pair.
\par	
	To check that  $\{\hV, \cV\}$ satisfy property {\bf [ii]} of a Darboux pair, we
	use {\bf [b]} and \EqRef{SymReg6}  to calculate
\begin{equation}
	\hV^{(\infty)} \cap  \cV^{(\infty)} = (\hW^{(\infty)}/G) \cap (\cW^{(\infty)}/G)
	=  (\hW^{(\infty)} \cap  \cW^{(\infty)})/G  = 0.
\end{equation}
	
	To prove property {\bf [iii]} of a Darboux pair, we first prove that
\begin{equation}
	[\Lambda^2(\hW) + \Lambda^2(\cW)]_{\semibasic} =  \Lambda^2(\hW_{\semibasic}) + \Lambda^2(\cW_{\semibasic}).
\EqTag{SymRed7}
\end{equation}
	It suffices to check this locally.
	Since $G$ acts transversally to $\hW^{(\infty)} \cap \cW$, there is a locally defined sub-bundle $Z$ of  $\hW^{(\infty)} \cap \cW$
	such that the canonical pairing on $\Gamma\times Z$ is pointwise non-degenerate. For any such sub-bundle $Z$ we have that
	$\hW =  \hW_{\semibasic} \oplus Z$  and $\cW =  \cW_{\semibasic} \oplus Z$ and therefore
\begin{equation*}
	\Lambda^2(\hW) + \Lambda^2(\cW) =   \Lambda^2(\hW_{\semibasic}) + \Lambda^2(\cW_{\semibasic}) + T^*M \wedge Z.
\end{equation*}
	This (local) decomposition immediately leads to \EqRef{SymRed7}. To complete the proof of property {\bf [iii]}, let
	$\omega \in \hV \cap \cV$. Then $\bfq^*(\omega) \in \hW \cap \cW$ and hence, by the properties of the
	Darboux pair $\{\,\hW, \cW\, \}$ and \EqRef{SymRed7},  $\bfq^*(\extd \omega) \in   \Omega^2(\hW_{\semibasic}) + \Omega^2(\cW_{\semibasic})$.
	We pullback this last equation by a local cross-section of $\bfq$ and invoke \EqRef{SymRed3}  to conclude that
	$\extd \omega \in \Omega^2(\hV) + \Omega^2(\cV)$.
\end{proof}

\begin{Theorem}
\StTag{SymRed8}
	Let  $\{\,\hW, \cW \, \}$ be a Darboux pair on $M$. Suppose that $G$ is
	a regular symmetry group  of $\hW$ and $\cW$ and that,
	in addition, $G$ is transverse to  $\hW^{(\infty)} \cap   \cW$ and $\hW \cap \cW^{(\infty)}$. Then
\begin{equation}
	(\hCalW \mycap \cCalW)/  G =  (\hCalW/G) \mycap (\cCalW /G).
\EqTag{SymRed9}
\end{equation}
\end{Theorem}
\begin{proof} We first remark that Theorem \StRef{SymRed1} implies that
         $\hV = \hW/G$ and $\cV = \cW/G$ define a Darboux pair and,  second,  that it suffices to check \EqRef{SymRed9} locally.

	Transversality implies that there are locally defined sub-bundles $\hZ\subset \hW^{(\infty)} \cap \cW$ and $\cZ \subset \hW \cap \cW^{(\infty)}$
	on which the natural pairing with $\Gamma$ (the infinitesimal generators for $G$) is non-degenerate at each point $x\in M$. It follows from \EqRef{SymRed5} that
\begin{subequations}
\begin{alignat}{2}
	\hW^{(\infty)} \cap \cW &= \bfq^*(\hV^{(\infty)} \cap   \cV)  \oplus \hZ,
&\quad
	\hW^{(\infty)} &= \bfq^*(\hV^{(\infty)}) \oplus \hZ,
\EqTag{SymRed10a}
\\[2\jot]
\hW \cap \cW^{(\infty)} &= \bfq^*(\hV \cap   \cV^{(\infty)})  \oplus \cZ,
&\quad
	\cW^{(\infty)} &= \bfq^*(\cV^{(\infty)}) \oplus \cZ,
\EqTag{SymRed10b}
\\[2\jot]
	\hW& = \bfq^*(\hV) \oplus \hZ,
&\quad
	\cW &= \bfq^*(\hV) \oplus \cZ,\quad\text{and}
\EqTag{SymRed10c}
\end{alignat}
\begin{equation}
	\cW \cap \hW = \bfq^*(\cV \cap \hV)  + \hZ =  \bfq^*(\cV \cap \hV)  + \cZ.
\EqTag{SymRed10d}
\end{equation}
\end{subequations}
	If $\bfhzeta$ and $\bfczeta$ are local basis of sections for $\hZ$ and $\cZ$ then,  because  $\bfhzeta \in \hW^{(\infty)}$ and
 $\bfczeta \in \cW^{(\infty)}$,  $\extd \bfhzeta \equiv 0 \mod  \hW^{(\infty)}$ and $\extd \bfczeta \equiv 0 \mod  \cW^{(\infty)}$, or
\begin{equation*}
	\extd \bfhzeta \equiv 0 \mod \{\, \bfq^*(\hV^{(\infty)}),\,\bfhzeta \, \}
\quad\text{and}\quad
	\extd \bfczeta \equiv 0 \mod \{\, \bfq^*(\cV^{(\infty)}),\, \bfczeta \, \}.
\end{equation*}
	These  equations, together with \EqRef{SymRed10c}, show that the Pfaffian systems $\hCalW$ and $\cCalW$ are algebraically generated as
\begin{equation}
	\hCalW = \{\,\bfq^*(\hCalV), \,\bfhzeta \,\}
	\quad\text{and}\quad
	\cCalW = \{\,\bfq^*(\cCalV),\, \bfczeta \,\}.
\end{equation}
	To complete the proof of the theorem, we now make the key observation that since  $\bfczeta \in \hW \cap \cW$ , \EqRef{SymRed10c} implies that
	we may express these forms as linear combinations of the $\bfhzeta$ and the forms in $\bfq^*(\hV \cap \cV)$.
	Hence $\cCalW$ is also algebraically generated as
	$\cCalW = \{\,\bfq^*(\cCalV), \bfczeta \,\}$.  We may list an explicit set of local generators for $\bfq^*(\hCalV)$ and  $\bfq^*(\cCalV)$
	using the 1-adapted coframe constructed for the Darboux pair $\{\,\hV, \cV\,\}$ in Theorem \StRef{First9}. The arguments given in the
	proof of Theorem \StRef{First20} can then  be repeated to show that
	$\hCalW \mycap \cCalW = \{\,\bfq^*(\hCalV \mycap \cCalV), \bfhzeta \,\}$
	and hence $(\hCalW \mycap \cCalW)_{\semibasic} = \bfq^*(\hCalV \mycap \cCalV)$, as required.
\end{proof}
\begin{Corollary}
\StTag{SymRed2}
	Let $W_1$ and $W_2$ be Pfaffian systems on manifolds $M_1$ and $M_2$
	with $W_1^{(\infty)} = W_2^{(\infty)} = 0$. Suppose that
\par
\medskip
\noindent
{\bf [i]} a Lie group $G$ acts freely on $M_1$ and $M_2$ and as  symmetry groups of
	both $W_1$ and $W_2$;
\par
\medskip
\noindent
{\bf [ii]} $G$  is transverse to both  $W_1$ and $W_2$; and
\par
\medskip
\noindent
{\bf [iii]} the diagonal action of $G$ on $M_1\times M_2$ is regular.
\par
\medskip
\noindent
	Then the quotient differential system
\begin{equation}
	\CalJ = (\CalW_1 + \CalW_2)/G
\end{equation}
	is a decomposable, Darboux integrable, differential system with respect to the Darboux pair
\begin{equation}
	\hU = (W_1 \oplus \Lambda^1(M_2))/G
	\quad\text{and}\quad
	\cU = (\Lambda^1(M_1) \oplus W_2)/G.
\EqTag{SymRed11}
\end{equation}
\end{Corollary}
\begin{proof} As in  Theorem \StRef{Trivial1}, define the Darboux pair
	$\hW = W_1 \oplus\Lambda^1(M_2)$ and $\cW = \Lambda^1(M_1) \oplus W_2$.
	To conclude from  Theorem \StRef{SymRed1} that
	$\hU = \hW/G$, $\cU = \cW/G$ are also a Darboux pair, we observe that {\bf [i]} and  {\bf[ii]} imply that
the diagonal action of $G$ on $M_1 \times M_2$ is transverse to the  Pfaffian systems $W_1$ and $W_2$ pulled back to
	$M_1\times M_2$ (a fact which is not true without {\bf [i]}. We also note that, by \EqRef{Trivial4},
\begin{equation}
	\hW \cap \cW^{(\infty)} = W_1
	\quad\text{and}\quad
	\hW^{(\infty)} \cap \cW = W_2.
\end{equation}
	We now use \EqRef{Trivial15} and Theorem \StRef{SymRed8} to deduce that
\begin{equation*}
	\CalJ = (\CalW_1 + \CalW_2)/G =  ( \hCalW \mycap \cCalW)/G =  (\hCalW/G) \mycap (\cCalW/G) = \hCalU \mycap\cCalU.
\end{equation*}
	An application of Corollary  \StRef{First20} completes the proof.
\end{proof}

\newpage
\section{The coframe adaptations for a Darboux pair}

	In this section we present a series of coframe adaptations for any Darboux pair $\{\,\hV,\, \cV \, \}$.
	These coframe adaptations lead to the proof of the following theorem.

\begin{Theorem}
\StTag{CofrAdapt1}
	Let  $\{\, \hV, \cV \,\}$ be a Darboux pair of Pfaffian differential systems on a manifold $M$.
	Then around each point of $M$ there are 2 coframes
\begin{equation}
	\{\, \htheta^i, \hpi^a, \cpi^\alpha\,\}
	\quad\text{and}\quad
	\{\, \ctheta^i, \hpi^a, \cpi^\alpha\,\}
\EqTag{CofrAdapt2}
\end{equation}
	such that
\begin{equation}
	\hV = \{\, \htheta^i, \hpi^a \,\}, \quad
	\hV^{\infty} =   \{\,\hpi^a \,\},\quad
	\cV= \{\,  \ctheta^i, \cpi^a \, \}, \quad
	\cV^{\infty} =   \{\,\cpi^a \,\}		
\EqTag{CofrAdapt3}
\end{equation}
	and with structure equations
\begin{equation}
\begin{aligned}
	\extd \htheta^i & = \frac12  G^i_{\alpha\beta} \cpi^\alpha \wedge \cpi^\beta + \frac12 C^i_{jk} \htheta^j \wedge \htheta^k \quad\text{and}
\\
	\extd \ctheta^i & = \frac12 H^i_{ab} \hpi^a \wedge \hpi^b - \frac12 C^i_{jk} \ctheta^j \wedge \ctheta^k.
\end{aligned}
\EqTag{CofrAdapt4}
\end{equation}
	The coefficients  $C^i_{jk}$ are the structure constants of a Lie algebra whose isomorphism class is an invariant of the
	Darboux pair  $\{\, \hV, \cV\,\}$.
\end{Theorem}
	
\par
	We have already constructed the first coframe adaptation in Section 2.
	The second coframe (Section 4.1)  adaptation provides us with a stronger form for the decomposition of the
	structure equations than that initially  given by \EqRef{Dec5}.
	The third and fourth coframe adaptations (Section 4.2 and 4.3) lead to the definition
	of the Vessiot Lie algebra  $\vess(\hV, \cV)$, with structure constants  $C^i_{jk}$,
	as a fundamental invariant for any Darboux integrable differential system.
	The importance of this algebra was
	first observed by Vassiliou \cite{vassiliou:2001a} who introduced the notion of
	the tangential symmetry algebra for certain  special classes of
	Darboux integrable systems. The tangential symmetry algebra is a Lie algebra
	of vector fields $\Gamma$ on $M$ which is isomorphic (as an abstract Lie algebra) to the Vessiot Lie algebra.
	The tangential symmetry algebra  also plays a significant  role in
	Eendebak's projection method \cite{eendebak:2007a} for Darboux integrable equations.

	Of course, the tangential symmetry algebra $\Gamma$ determines a local transformation group
	on $M$ but this transformation group will generally not be the correct one  required for
	constructing the superposition formula (See Example 6.1).  One  final coframe adaptation  (Section 4.4) is needed to arrive
	at the proper transformation groups (See Definition \StRef{SuperForm33}) for this construction  --
	this is the coframe  that satisfies the structure equations given in  Theorem \StRef{CofrAdapt1}.

	The coframe adaptations 1 -- 4 involve only differentiations and elementary linear algebra operations.
	For the fifth coframe adaptation, in the
	case where the Vessiot algebra is not semi-simple, one must integrate a linear system of
	total differential equations (see \EqRef{Fifth56})
	and also solve the exterior derivative equation $d(\alpha) = \beta$,  where $\beta$
	is a closed 1 or 2-form (see \EqRef{Fifth66} and \EqRef{Fifth71}).
\subsection{The second adapted coframe for a Darboux pair.}

	Let $\{ \hV\, , \cV \}$ be a Darboux pair and let
	$\{\,\bftheta,\, \bfhsigma,\, \bfheta,\, \bfcsigma, \,\bfceta \, \}$ be a 1-adapted coframe with dual frame
	$\{\,  \bfvecttheta,\, \bfvecthsigma,\, \bfvectheta,\, \bfvectcsigma,\, \bfvectceta \, \}$.  From this  point forward, we are solely interested in
	adjustments to the forms $\bftheta$ that will simplify the 1-adapted structure equations
\begin{equation}
	\extd \bftheta	= \bfA\,  \bfhsigma \wedge \bfhsigma + \bfB\, \bfcsigma \wedge \bfcsigma
	\mod \{ \,\bftheta,\, \bfheta,\, \bfceta \,\}.
\EqTag{Second4}
\end{equation}
        Of the ``mixed"  wedge products
\begin{equation}
	\bfhsigma \wedge\bfcsigma,\quad  \bfhsigma\wedge\bfceta,
	\quad \bfheta \wedge\bfcsigma, \quad\text{and}\quad \bfheta \wedge\bfceta
\EqTag{Second3}
\end{equation}
	the products $\bfhsigma \wedge\bfcsigma$ are the only
	ones definitely not present in \EqRef{Second4}. We  shall now show that it is possible to make an adjustment to the 1-forms $\bftheta$
	of the type
\begin{equation}
	\bfthetapr = \bftheta + \bfP\, \bfheta + \bfQ \, \bfceta
\EqTag{Second20}
\end{equation}
	so that the structure equations for the modified forms $\bfthetapr$ are free of all the wedge products \EqRef{Second3}.   We begin with the following simple observation.
\begin{Lemma} If  $\{\,  \bfvecttheta,\, \bfvecthsigma,\, \bfvectheta,\, \bfvectcsigma,\, \bfvectceta \, \}$ is the dual frame to a
1-adapted coframe for the Darboux pair $\{\, \hV, \cV \, \}$, then
\begin{equation}
	[\, \bfvecthsigma, \,\bfvectcsigma \,] = 0.
\EqTag{Second9}
\end{equation}	
\end{Lemma}
\begin{proof}
	It suffices to note that none of the structure equations for a  1-adapted coframe contain any of the wedge products
	$\bfhsigma \wedge \bfcsigma$.
\end{proof}
	The construction of  the second adapted coframe is completely  algebraic in that only differentiations
	and linear algebraic operations are involved. Let
	$\hS_a =  \vecthsigma$ and $\hS_\alpha = \vectcsigma$
	and define two sequences of  vector fields inductively by 			
\begin{equation}
       \hS_{a_1a_2\cdots a_\ell} = [\,\hS_{a_1a_2\cdots a_{\ell-1}}, \, \hS_{a_\ell} ,\, ]
	\quad\text{and}\quad
	\cS_{\alpha_1\alpha_2\cdots \alpha_\ell} =
	[\,\cS_{\alpha_1\alpha_2\cdots \alpha_{\ell-1}},\, \cS_{\alpha_\ell}\,].
\EqTag{Second8}
\end{equation}
	Because the vector fields $\hS_{a_\ell} \in \hH$,
	the vector fields $S_{a_1a_2\cdots a_\ell}$ belong to $\hH^{(\infty)}$ and are therefore
	linear combinations of the vector fields $\{\bfvecttheta, \, \bfvecthsigma, \, \bfvectheta \}$. However, because
	$\extd \bf \bfhsigma = 0 $, one can deduce by a straightforward induction that in fact
\begin{equation}
	S_{a_1a_2\cdots a_\ell} \in \text{span}\{\,  \bfvecttheta,\, \bfvectheta \, \}
	\quad\text{and, likewise,} \quad
	\cS_{\alpha_1\alpha_2\cdots \alpha_\ell}\in \text{span}\{\,  \bfvecttheta,\, \bfvectceta \, \}.
\EqTag{Second15}
\end{equation}
	The Jacobi identity  also shows that
\begin{equation}
	\hH^{(\infty)} = \text{span} \{\, \hS_{a_1a_2\cdots a_\ell} \,\}_{\ell\geq 1}
	\quad\text{and} \quad
	\cH^{(\infty)} = \text{span} \{\, \cS_{\alpha_1\alpha_2\cdots \alpha_\ell} \,\}_{\ell\geq 1}.
\EqTag{Second10}
\end{equation}	
	From \EqRef{Second9} it also follows that
\begin{equation}
	[\, \hS_{a_1a_2 \cdots a_k}, \cS_{\alpha_1\alpha_2\cdots \alpha_\ell}\,] = 0.
\EqTag{Second7}
\end{equation}
\par
	for $k,l \geq 1$. By virtue  of \EqRef{Second10} we can  choose a specific collection of  iterated Lie brackets
	$\hS_{a_1a_2\cdots a_\ell}$ and   $\cS_{\alpha_1\alpha_2\cdots \alpha_\ell}$, $\ell \geq 2$,
	which complete $\{\,\bfvecttheta,\,\bfvecthsigma\, \}$ and $\{\,\bfvecttheta, \,\bfvectcsigma\, \}$
	to local bases for  $\hH^{(\infty)}$ and $\cH^{(\infty)}$ respectively. Denote the iterated brackets so chosen by
	$\{\,\bfvecthetapr\, \}$ and $\{\,\bfvectcetapr\, \}$.  Since
	$\text{span} \{\bfvecttheta, \, \bfvecthsigma, \, \bfvectheta \}$
	$= \text{span} \{\bfvecttheta, \, \bfvecthsigma, \, \bfvecthetapr \}$
	and
	$\text{span} \{\bfvecttheta, \, \bfvectcsigma, \,  \bfvectceta \}
	= \text{span} \{\bfvecttheta, $
	$\, \bfvectcsigma, \, \bfvectcetapr \}$,
	this gives us a preferred 0-adapted frame
	$\{\,  \bfvecttheta, \bfvecthsigma, \bfvecthetapr, \bfvectcsigma, \bfvectcetapr \, \}$ where, on account of \EqRef{Second7},
\begin{equation}
        [\,\bfvecthsigma,\, \bfvectcsigma\,] = 0, \quad  [\,\bfvecthsigma,\, \bfvectcetapr\,] = 0 \quad
	[\,\bfvecthetapr,\, \bfvectcsigma\,] = 0, \quad [\,\bfvecthetapr,\, \bfvectcetapr\,] = 0.
\EqTag{Second12}
\end{equation}
	It follows that the coframe $\{\,\bfthetapr,\, \bfhsigmapr,\, \bfhetapr,\, \bfcsigmapr,\,   \bfcetapr \, \}$ dual
	to $\{\,  \bfvecttheta,\, \bfvecthsigma,\, \bfvecthetapr,\, \bfvectcsigma,\, \bfvectcetapr \, \}$
	is a 0-adapted coframe for the Darboux pair $\{\hV,  \cV \}$.  Equations \EqRef{Second15} imply \EqRef{Second20}.

	On account of \EqRef{Second12}, the structure equations for the
	forms $\{\,\bfthetapr\, \}$ are free of all
	the wedge products $\bfhsigmapr \wedge\bfcsigmapr$,  $\bfhsigmapr\wedge\bfcetapr$,
	$\bfhetapr \wedge\bfcsigmapr$, and  $\bfhetapr \wedge\bfcetapr$.  We can express this result by writing
\begin{equation*}
	\extd \bfthetapr  \in \Omega^2(\hV^{(\infty)}) + \Omega^2(\cV^{(\infty)})
	\mod \{ \,\bfthetapr \,\}.
\end{equation*}
	The coframe
	$\{\, \bfthetapr,  \bfhsigma,\, \bfheta,\, \bfcsigma, \,\bfceta \, \}$ therefore satisfies
	the criteria of the following theorem.
\begin{Theorem}
\StTag{SecondThm1}
	Let $\{\, \hV, \cV\}$ be a Darboux pair on  a manifold $M$. Then about each point of $M$ there exists a
	1-adapted coframe   $\{\,\bftheta,\, \bfhsigma,\, \bfheta,\, \bfcsigma, \,\bfceta \, \}$
	with structure equations  \EqRef{First10} and 	
\begin{equation}
	\extd \bftheta
	=
	\bfA\, \bfhpi \wedge \bfhpi   +
	\bfB\, \bfcpi    \wedge \bfcpi
\quad
	\mod \{\,\bftheta\,\},
\EqTag{Second11}
\end{equation}
	where $\bfhpi$ and $\bfcpi$ denote the tuples of forms $(\,\bfhsigma,\, \bfheta\,)$ and $(\,\bfcsigma,\, \bfceta\,)$.
\end{Theorem}

	Coframes which satisfy the conditions of Theorem \StRef{FirstLem1} are call
	{\deffont 2-adapted coframes}.

\newpage
\subsection{The third adapted coframes for a Darboux pair}
\par
	Written out in full, the structure equations \EqRef{Second11} are
\begin{equation}
	\extd \bftheta	= \bfA\,  \bfhpi \wedge \bfhpi + \bfB\, \bfcpi\wedge \bfcpi  +
	\bfC \, \bftheta \wedge \bftheta + \bfM \, \bfhpi \wedge \bftheta + \bfN \, \bfcpi \wedge \bftheta.
\EqTag{Third1}
\end{equation}
	We obtain the third coframe reduction for a Darboux pair by showing that a change of coframe
	$\bfthetapr = \bfP \,\bftheta$ can be made so as to eliminate either all the
	wedge products  $\bfhpi \wedge \bftheta$  \underline{or} all the  wedge products  $\bfcpi \wedge \bftheta$ in \EqRef{Third1}.
        The construction of this coframe uses another set of iterated  Lie brackets.
\par
	We start with a 2-adapted coframe $\{\,\bftheta,\, \bfhsigma,\, \bfheta,\, \bfcsigma,\, \bfceta,\}$, where
	$\bfhpi = (\,\bfhsigma,\, \bfheta\,)$ and $\bfcpi= (\,\bfcsigma,\, \bfceta\,)$,
	and introduce the provisional frame $\{\,\bftheta,\, \bfhiota, \bfciota\}$,
	where  $\bfhiota = \extd \bfhI$ and $\bfciota = \extd \bfcI$.  This is not a 0-adapted coframe although it is the case that
\begin{equation}
	\text{span}\{\,  \bfhiota\,\} = \text{span}\{\,  \bfhpi\,\}
	\quad\text{and}\quad
	\text{span}\{\,  \bfciota\,\} = \text{span}\{\,  \bfcpi\,\}.
\EqTag{Third12}
\end{equation}
	From \EqRef{Third1} and \EqRef{Third12} we find that the structure equations for this coframe are
\begin{equation}
\begin{gathered}
	\extd\bfhiota  = 0, \quad \extd\bfciota = 0,
	\qquad\text{and}\qquad
\\
	\extd \bftheta = \bfA\, \bfhiota   \wedge  \bfhiota
	+ \bfB\,  \bfciota\wedge \bfciota
	+ \bfC \, \bftheta \wedge \bftheta + \bfM \, \bfhiota \wedge \bftheta + \bfN \, \bfciota \wedge \bftheta.
\end{gathered}
\EqTag{Third2}
\end{equation}
	Denote the dual to this provisional coframe by $\{\, \bfvecttheta, \bfhU, \bfcU\}$. Since
\begin{equation*}
	\text{span}\{\, \bfvecthsigma \, \} = \ann\, \cV = \ann\, \{ \bftheta, \bfhiota, \bfcsigma \}
	\subset \ann\, \{ \bftheta, \bfhiota \} =  \text{span} \{\,  \bfcU \, \},
\end{equation*}
	we have that
\begin{equation}
	\bfvecthsigma   \subset \text{span} \{\,  \bfcU \, \}
	\quad\text{and}\quad
	\bfvectcsigma  \subset \text{span} \{\, \bfhU\, \}.
\EqTag{Third3}
\end{equation}

	From  \EqRef{Third2} it follows that the   structure equations for the vectors fields $\bfhU$ and $\bfhU$ are
\begin{equation}
\begin{gathered}
\begin{aligned}{}
	[\, \bfhU,\, \bfhU  \,]& = -\bfA \bfvecttheta,
&	\quad [\, \bfhU, \, \bfvecttheta\,] & =  - \bfM \bfvecttheta,
\\
	[\, \bfcU, \, \bfcU\,]
&  	= - \bfB \bfvecttheta,
&	\quad [\, \bfcU, \, \bfvecttheta\,]
&	= - \bfN \bfvecttheta,
\end{aligned}
\\
\text{and}\quad [\, \bfcU, \, \bfhU\,] = 0.
\end{gathered}
\EqTag{Third4}
\end{equation}
	As in Section 4.2, define two sequences of vector  fields inductively by
\begin{equation}
       \hU_{a_1a_2\cdots a_\ell} = [\,\hU_{a_1a_2\cdots a_{\ell-1}}, \, \hU_{a_\ell} \, ]
	\quad\text{and}\quad
	\cU_{\alpha_1\alpha_2\cdots \alpha_\ell} =
	[\,\hU_{\alpha_1\alpha_2\cdots \alpha_{\ell-1}}\, \hU_{\alpha_\ell}\,].
\EqTag{Third5}
\end{equation}	
        A simple induction argument, based upon the last of \EqRef{Third4}  and the Jacobi identity, shows that
\begin{equation}
	[\, \hU_{a_1a_2 \cdots a_k}, \cU_{\alpha_1\alpha_2\cdots \alpha_\ell}\,] = 0.
\EqTag{Third6}
\end{equation}
	for $k,l \geq 1$. On account of \EqRef{First5}, \EqRef{First6} and \EqRef{Third3} the  iterated brackets $\hU_{a_1a_2\cdots a_\ell}$
	will span all of $\cH$ and therefore,
	by the first two equations in  \EqRef{First6}, we can  choose a basis  $\bfX$  for
	$K= \hH ^{(\infty)}  \cap \cH ^{(\infty)}$
	(see \EqRef{First5}) consisting of a linear 	
	independent set of the vector fields $\hU_{a_1a_2\cdots a_\ell}$.
	Alternatively,  we can choose  a basis  $\bfY$  for $K$ consisting
	of a linear independent set of the vector fields $\cU_{a_1a_2\cdots a_\ell}$.
	Due to \EqRef{Third6} 	these two bases for $K$ satisfy
\begin{equation}
	[\, \bfX,\, \bfcU\,] = 0, \quad [\,\bfX,\, \bfY\,] = 0, \quad [\,\bfY,\, \bfhU\,] = 0.
\EqTag{Third22}
\end{equation}
	Denote the dual of the coframe  $\{\bfX,  \bfhU,\, \bfcU\, \}$  by
	$\{\,\bfthetaX,\, \bfhiota\, , \bfciota\,\}$ and  the dual of the coframe $\{\bfY,  \bfhU,\, \bfcU\, \}$ by
	$\{\,\bfthetaY,\, \bfhiota\,  ,\bfciota\, \}$.

\begin{Theorem}
\StTag{Third7}
	Let $\{\, \hV, \cV\}$ be a Darboux pair on  a manifold $M$. Then about each point of $M$ there are two
	2-adapted coframes
\begin{equation*}
	\{\,\bfthetaX,\, \bfhsigma,\, \bfheta,\, \bfcsigma, \,\bfceta \, \}
	\quad\text{and}\quad
	\{\,\bfthetaY,\, \bfhsigma,\, \bfheta,\, \bfcsigma, \,\bfceta \, \}
\end{equation*}
	with structure equations   \EqRef{First10},	
\begin{align}
	\extd \bfthetaX
&	=
	\bfA\, \bfhpi \wedge \bfhpi  +
	\bfB\, \bfcpi \wedge \bfcpi
        +\bfC\,  \bfthetaX \wedge \bfthetaX
	+\bfM\, \bfhpi \wedge \bfthetaX,
\EqTag{Third8}
\\[-5\jot]
\intertext{and}\notag
\\[-12\jot]
	\extd \bfthetaY
&	=
	\bfE\, \bfhpi \wedge \bfhpi  +
	\bfF\, \bfcpi \wedge \bfcpi
        +\bfK\, \bfthetaY \wedge \bfthetaY
	+\bfN\, \bfcpi \wedge \bfthetaY .
\EqTag{Third9}	
\end{align}
	Moreover, $\text{\rm span} \{\,\bfthetaX \} = \text{\rm span} \{\,\bfthetaY \}$
	and the vector fields $\bfX$, $\bfY$ belonging to the dual frames
	$\{\, \bfX, \,\bfvecthsigma,\, \bfvectheta,\, \bfvectcsigma,\, \bfvectceta \, \}$ and
	$\{\, \bfY, \, \,\bfvecthsigma,\, \bfvectheta,\, \bfvectcsigma,\, \bfvectceta\, \}$  satisfy
\begin{equation}
[\,\bfX, \bfY] = 0.
\EqTag{Third10}
\end{equation}
\end{Theorem}

	Coframes which satisfy the conditions of Theorem \StRef{Third7} are called {\deffont 3-adapted coframes}.
\begin{Remark}
	In the special case where the pair $\{\, \hV, \cV \}$ admits an involution (see Remark \StRef{FirstRem2})
	the forms $\bfthetaY$ may be defined by $\bfthetaY = \Phi^*(\bfthetaX)$. Then the structure equations \EqRef{Third9}	
	can be immediately inferred from  \EqRef{Third8}.
\end{Remark}

\newpage
\subsection[The fourth adapted coframe  and the Vessiot algebra]{The fourth adapted coframe  and the Vessiot algebra for a Darboux pair.}
	We now show that any 3-adapted coframe may be adjusted so that the structure functions
	$C^i_{jk}$ and $K^i_{jk}$ (see \EqRef{Third8}--\EqRef{Third10}) are constants and satisfy $C^i_{jk} =- K^i_{jk}$.
\begin{Theorem}
\StTag{Fourth5}	
	Let $\{\hV, \cV\}$ be a Darboux pair on $M$.
\par
\medskip
\noindent
{\bf[i]}
\quad
	 Then, about  each point of $M$,  there exist  local coframes
	$\{\,\thetaX^i,\hpi^a, \cpi^\alpha \,\}$ and $\{\,\thetaY^i,\hpi^a, \cpi^\alpha\,\}$
	which are 3-adapted  and with structure equations
\begin{align}
	\extd \thetaX^i
&
	=  \frac12 A^i_{ab}\, \hpi^a \wedge \hpi^b  +
	\frac12 B^i_{\alpha\beta}\, \cpi^\alpha \wedge \cpi^\beta
        +\frac12 C^i_{jk}\thetaX^j\wedge \thetaX^k
	+M^i_{a j}\hpi^a\wedge \thetaX^j
\EqTag{Fourth6}
\\[-3\jot]
\intertext{and}
\notag
\\[-9\jot]
	\extd \thetaY^i
&
	=  \frac12 E^i_{ab}\, \hpi^a \wedge \hpi^b  +
	\frac12 F^i_{\alpha\beta}\, \cpi^\alpha \wedge \cpi^\beta
        -\frac12 C^i_{jk}\thetaY^j\wedge \thetaY^k
	+N^i_{\alpha j}\cpi^\alpha\wedge \thetaY^j.
\EqTag{Fourth7}	
\end{align}
	The  structure functions $C^i_{jk}=-C^i_{kj}$ are constant on $M$
	and are the structure constants
	of a real  Lie algebra.
\par
\smallskip
\noindent
{\bf[ii]}
\quad
	The  isomorphism class of the  Lie algebra defined by the structure constants $C^i_{jk}$ is an
	invariant of the Darboux pair  $\{\, \hV, \cV\}$.
\end{Theorem}

	Coframes which satisfy the conditions of Theorem \StRef{Fourth5}
	are called {\deffont 4-adapted coframes}. As with the 2 and 3-adapted  coframes defined previously, $\bfhpi = (\bfhsigma, \bfheta)$	
	and $\bfcpi = (\bfcsigma, \bfceta)$, with structure equations \EqRef{First10}.
	In passing from the 3-adapted coframes to the  4-adapted coframes only
	the forms $\thetaX^i$ and $\thetaY^i$ are modified. The (abstract) Lie algebra defined by the structure constants $C^i_{jk}$ is called
	the {\deffont Vessiot algebra} for the Darboux pair $\{\, \hV, \cV \, \}$ and is denoted by $\vess(\hV, \cV)$. The vectors fields
	$X_i$ in the dual basis $\{\, X_i, \hU_a,\cU_\alpha\,\}$ to the 4-adapted coframe $\{\,\thetaX^i,\hpi^a, \cpi^\alpha \,\}$
	define a realization for $\vess(\hV, \cV)$, with structure equations $[\,X_j, X_k\,] = - C^i_{jk} X_i$.

\begin{proof}[Proof of Theorem \StRef{Fourth5} ]
	 Theorem \StRef{Third7} provides us with two  locally defined 3-adapted frames
	$\{\, X_i, \hU_a,\cU_\alpha\,\}$ and  $\{\, Y_i, \hU_a,\cU_\alpha\,\}$
	and dual coframes $\{\,\thetaX^i,\hpi^a, \cpi^\alpha \,\}$ and $\{\,\thetaY^i,\hpi^a, \cpi^\alpha\,\}$.
	The structure equations are \EqRef{Third8}--\EqRef{Third10}.
	We begin the proof of part {\bf [i]} by showing that the $C^i_{jk}$ are functions
	only of the first integrals $\hI^a$ while the $K^i_{jk}$ are functions only of the first integrals
	$\cI^\alpha$.

	Since $\text{span}\{\, \bfthetaX \} =  \text{span}\{\, \bfthetaY \}$, there is an  invertible matrix $Q$
	such that
\begin{equation}
        \thetaX^i = Q^i_j \thetaY^j.
\EqTag{Fourth1}
\end{equation}
	On account of the  structure equations \EqRef{Third8} and the fact that the vector fields $X_i$ and $Y_j$ commute,
	the identity
\begin{equation*}
	\extd \thetaX^i(X_j, Y_k) = X_j(\thetaX^i(Y_k)) - Y_k(\thetaX^i(X_j)) - \thetaX^i([X_j,Y_k])
\end{equation*}
	leads to
\begin{equation}
	 X_j(Q^i_k) = C^i_{j\ell}Q^\ell_k .
\EqTag{Fourth2}
\end{equation}
	We shall use this  result  repeatedly in  what follows -- it is equivalent to the fact that the
	vector fields $X_i$ and $Y_j$ commute, a property of the two 3-adapted coframes  which is not encoded in the structure equations
	\EqRef{Third8} and \EqRef{Third9}.
\par
        We next substitute \EqRef{Fourth1} into \EqRef{Third8} and equate the
	coefficients of $\thetaY^j\wedge\thetaY^k$ to deduce that
\begin{equation*}
       X_\ell(Q^i_k) Q^\ell_j - X_l(Q^i_j) Q^l_k + Q^i_lK^l_{jk} = C^i_{\ell m} Q^\ell_j Q^m_k .
\end{equation*}
	By virtue of \EqRef{Fourth2}, this equation simplifies to
\begin{equation}
	Q^i_l K^l_{jk} = -C^i_{lm} Q^l_j Q^m_k.
\EqTag{Fourth3}
\end{equation}
	Also, by equating to zero the coefficients of $\cpi^\alpha\wedge\thetaX^i\wedge\thetaX^j$ and
	$\hpi^a\wedge\thetaY^i\wedge\thetaY^j$
        in the expansions of the identities $\extd^2 \thetaX^i=0$ and $\extd^2 \thetaY^i = 0$,
	we find that
\begin{equation*}
        \cU_a(C^i_{jk}) = 0 \quad\text{and}\quad \hU_\alpha(K^i_{jk}) = 0.
\end{equation*}
        By Lemma \StRef{FirstLem1}, these equations imply that
\begin{equation}
          C^i_{jk} \in  \Inv(\hV)\quad\text{and}\quad    K^i_{jk} \in  \Inv(\cV).
\EqTag{Fourth4}
\end{equation}
\par

	Our next goal  is to show that the  coframes
	$\thetaX^i$ and $\thetaY^i$  may be adjusted so that $K^i_{jk} = -C^i_{jk}$ while
	still preserving \EqRef{Third8}--\EqRef{Third10}. To this end
	we introduce  local coordinates  $(\hI^a, \cI^\alpha, z^m)$ on $M$
	satisfying \EqRef{First14} and  \EqRef{First15}. Then, on account of
	\EqRef{Fourth4}, equation \EqRef{Fourth3} becomes
\begin{equation}
        Q_k^\ell(\hI^a, \cI^\alpha,z^m)\,K^k_{ij}(\cI^\alpha)   =
	-C^\ell_{hk}(\hI^a) \, Q^h_i (\hI^a, \cI^\alpha,z^m)\, Q^k_j(\hI^a, \cI^\alpha,z^m).
\EqTag {Fourth9}
\end{equation}
	Evaluate this  equation, first at a fixed point $(\hI^a_0,\cI^\alpha_0, z^m_0)$ and then at the point
	$(\hI^a,\cI^\alpha_0, z^m_0)$.  With
\newcommand\Ko{{\kern-2.3pt}\stackrel{\scriptscriptstyle o}{K}{}{\kern-2.3pt}}
\newcommand\Co{{\kern-2.3pt}\stackrel{\scriptscriptstyle o}{C}{}{\kern-2.3pt}}
\newcommand\Qo{{\kern-2.3pt}\stackrel{\scriptscriptstyle o}{Q}{}{\kern-2.3pt}}
\newcommand\thetao{{\stackrel{\scriptscriptstyle o}{\theta}}{}}
\newcommand\Xo{{\stackrel{\scriptscriptstyle o}{X}}{\kern-1.3pt}}
\newcommand\Yo{{\stackrel{\scriptscriptstyle o}{Y}}{\kern-1.3pt}}
\begin{gather*}
  	\Ko^k_{ij}= K^k_{ij}(\cI^\alpha_0), \quad
	\Co^k_{ij}=   C^k_{ij}(\hI^a_0), \quad          		
             \Qo^j_i   =Q^j_i(\hI^a_0, \cI^\alpha_0,z^m_0),\quad\text{and}\quad
\\[2\jot]	
	Q^j_i(\hI^a) =Q^j_i(\hI^a, \cI^\alpha_0,z^m_0)
\end{gather*}
	the  results are
\begin{equation}
	\Qo_k^\ell \Ko^k_{ij} 	
	=
	-\Co^\ell_{hk}\, \Qo^h_i \Qo^k_j
	\quad\text{and}\quad
	Q_k^\ell(\hI^a)\, \Ko^k_{ij}	=
	-C^\ell_{hk}(\hI^a)\, Q^h_i (\hI^a)\, Q^k_j(\hI^a).
\end{equation}
	It then readily follows that the matrix
\begin{gather}
        P^i_j(\hI^a) =Q^i_\ell(\hI^a)  (\Qo^{-1})^\ell_j .
\EqTag{Fourth21}
\\[-5\jot]
\intertext{satisfies}\notag
\\[-8\jot]
        P^\ell_k(\hI^a) \Co^k_{ij}
	=  C^\ell_{hk}(\hI^a)\,  P^h_i(\hI^a)\, P^k_j(\hI^a) .
\EqTag{Fourth10}
\end{gather}
	The  1-forms  $\thetao{}_{\kern -1.5pt X}^i$ defined by
\begin{equation}
          \thetao{}_{\kern -1.5pt X}^j P^i_j = \thetaX^i
\EqTag{Fourth11}
\end{equation}
	then satisfy  structure equations of the  required form \EqRef{Fourth6},
	where the  structure functions $C^i_{jk}$ coincide with the
        constants $\Co^i_{jk}$.
	Finally, by evaluating \EqRef{Fourth9} at  $(\hI^a_0,\cI^\alpha,z^m_0)$, it follows that
\begin{gather}
    	R^i_j(\cI^\alpha) =(Q^{-1}(\hI^a_0,\cI^\alpha,z^m_0) )^i_j
\EqTag{Fourth22}
\\[-5\jot]
\intertext{satisfies}\notag
\\[-8\jot]
       -R^\ell_k(\cI^\alpha)\Co^k_{ij} =
	K^\ell_{hk}(\cI^\alpha)\, R^h_i(\cI^\alpha)\, R^k_j(\cI^\alpha)
\end{gather}
         and the 1-forms  $\thetao_{ Y}^i$  defined by
\begin{equation}
          \thetao{}_{ Y}^j R^i_j = \thetaY^i
\EqTag{Fourth13}
\end{equation}
	satisfy the required structure equations \EqRef{Fourth7},
	again with $C^i_{jk} =\Co^i_{jk}$.
\par
	The next step in the proof of {\bf [i]} requires us to check that the coframes we have just constructed are still 3-adapted.
	We must therefore show that the dual vector fields $\Xo_i$ and $\Yo_j$ commute and for this
	it suffices to show, because of our remarks at the beginning of this proof, that
\begin{equation}
	\Xo_j(\Qo^i_k) = \Co^i_{jl} \Qo^l_k
	\quad\text{where}\quad
	\thetao^i_{\kern -1.5 pt X} = \Qo^i_j \,\thetao^j_{Y},
\EqTag{Fourth23}
\end{equation}	
	The proof of this formula follows from three simple observations.
	We first  note that \EqRef{Fourth11} and \EqRef{Fourth13}  imply
	$ \Qo{\kern - 2pt}{}^i_k = (P^{-1})^i_\ell\, Q^\ell_m \, R^m_k.$
	Secondly, because $\{\, \Xo{}^i, \hpi^a, \cpi^\alpha\,\}$ is the dual frame to
	$\{\, \thetao^i_{\kern -1.5pt X}, \hpi^a,\cpi^\alpha\,\}$, we have $\Xo_i = P_i^j X{}_j$.
	Finally, because $P^i_j =P^i_j(\hI^a)$ and $R^i_j= R^i_j(\cI^\alpha)$, it follows that
	$\Xo{}_j(P^i_k) = \Xo{}_j(R^i_k) = 0$.
	A straightforward calculation based upon these three observations and equations
        \EqRef{Fourth2} and \EqRef{Fourth10} leads to \EqRef{Fourth23}, as required.
	
	Finally, by equating to zero the coefficients of
	$\thetaX^i\wedge\thetaX^j\wedge\thetaX^k$
        in the expansion of the equations $\extd^2 \thetaX^i=0$  one finds that  the $C^i_{jk}$
	satisfy the Jacobi identity and are therefore the structure constants of a real $r$-dimensional Lie algebra.
\par
        To prove  {\bf [ii]}, let  $\{\, \theta^{\prime i}_X , \hpi^{\prime a}, \cpi^{\prime \alpha} \, \}$ be
	another  coframe adapted to the Darboux pair $\{\,\hV, \cV\,\}$
        with structure equations
\begin{equation}
\extd \theta^{\prime i}_X =
	\frac12 A^{\prime i}_{ab}\, \hpi^{\prime a} \wedge \hpi^{\prime b}  +
	\frac12 B^{\prime i}_{\alpha\beta}\, \cpi^{\prime \alpha} \wedge \cpi^{\prime \beta} +
        \frac12 C^{\prime i}_{jk}\theta^{\prime j}\wedge \theta^{\prime k}
	+M^{\prime i}_{a j}\hpi^{\prime a}\wedge \theta^{\prime j}.
\EqTag{Fourth8}
\end{equation}
	From the definition of a 0-adapted coframe (see  \EqRef{First4b}) we have
\begin{equation*}
\theta^{\prime i}_X = R^i_j \theta^j  + {\hS}^i_p \heta^p + \cS^i_q\ceta^q,
\end{equation*}
	where the matrix $R^i_j$ is invertible.
	Substitute this equation into \EqRef{Fourth8} and then substitute from \EqRef{First10} and
	\EqRef{Fourth6}.
	From the coefficients of  $\bfhsigma \wedge \bftheta$ we deduce that $\bfvecthsigma R^i_j = 0$.
	By Lemma \StRef{FirstLem1} this implies that   $R^i_j \in \text{Inv}(\hV)$ in which case one finds
	from the coefficients of  $\theta^j \wedge\theta^k $ that	
\begin{equation*}
          C^{\prime i}_{lm} R^l_j R^m_k = R^i_l C^{l}_{jk}.
\end{equation*}
	Hence the  structure constants  $C^{i}_{jk}$ and $C^{\prime i}_{jk}$ define the
	same abstract Lie algebra and the proof of part {\bf [ii]} is complete. 		
\end{proof}

\begin{Corollary}
\StTag{Fourth14}
	Let $\{\,\hV,\cV\,\}$ and $\{\,\hW, \cW \,\}$ be Darboux pairs on  manifolds $M$ and $N$ respectively and suppose
	that $\phi\:M \to N$  is smooth map satisfying
\begin{equation}
	\phi^*(\hW) \subset \hV \quad\text{and}\quad \phi^*(\cW) \subset \cV.
\EqTag{Fourth15}
\end{equation}
        Then $\phi$ induces a Lie algebra homomorphism
\begin{equation}
	\tilde \phi\:\vess(\hV, \cV) \to \vess(\hW, \cW).
\EqTag{Fourth16}
\end{equation}
\end{Corollary}
\begin{proof}
	The  proof of this corollary is  similar to that of part {\bf [ii]} of Theorem \StRef{Fourth5}.
	Let  $\{\, \theta^{\prime s}_X , \hpi^{\prime c}, \cpi^{\prime \gamma} \, \}$ be a 4-adapted
	coframe on $N$ for the Darboux pair $\{\,\hW, \cW\,\}$,
        with structure equations
\begin{equation}
\extd \theta^{\prime s}_X =
	\frac12 A^{\prime s}_{cd}\, \hpi^{\prime c} \wedge \hpi^{\prime d}  +
	\frac12 B^{\prime s}_{\gamma\delta}\, \cpi^{\prime \gamma} \wedge \cpi^{\prime \delta} +
        \frac12 K^ s_{rt}\theta^{\prime r}\wedge \theta^{\prime t}
	+M^{\prime s}_{a t}\hpi^{\prime a}\wedge \theta^{\prime t}.
\EqTag{Fourth17}
\end{equation}
	The constants $K^s_{rt}$ are the structure constants for the Lie algebra
	$\vess(\hW, \cW)$. Denote the dual frame by 	
	$\{\, X^{\prime}_r, U^{\prime}_c , U^{\prime}_\gamma \, \}$.

	Let $\{\,\thetaX^i,\hpi^a, \cpi^\alpha \,\}$ be a 4-adapted coframe on $M$
        for the Darboux pair$\{\, \hV, \cV \, \}$ with  the structure equations \EqRef{Fourth6}.
        The inclusions \EqRef{Fourth15} imply that
\begin{equation*}
        \phi^*(\hW \cap \cW) \subset \hV \cap \cV,	
	\quad
	\phi^*(\hW^{(\infty)})   \subset \hV^{(\infty)}
	\quad\text{and}\quad
	\phi^*(\cW^{(\infty)}) \subset \cV^{(\infty)}
\end{equation*}
	and therefore there are functions
	$R^s_i$, $\hS^s_p$,  $\cS^s_q$ $\hT^c_a$ and $\cT^\gamma_\alpha$  on $M$ such that
\begin{equation*}
	\phi^*(\theta^{\prime s}_X)   = R^s_i\theta^i_X  + {\hS}^s_p \heta^p + \cS^s_q\ceta^q,
	\quad
	\phi^*(\hpi^{\prime c}) = \hT^c_a \hpi^a
	\quad\text{and}\quad
	\phi^*( \cpi^{\prime \gamma}) = \cT^\gamma_\alpha\cpi^\alpha.
\end{equation*}
        We pullback  \EqRef{Fourth17} using these equations and substitute
	from \EqRef{Fourth6}.  The same arguments as in the proof of  {\bf [ii]} of
	Theorem \StRef{Fourth5} now show that $R^s_i\in \text{Inv}(\hV)$ and
	$R^r_iC^i_{jk} = K^r_{st}R^s_jR^t_k$.
        This proves that the Jacobian mapping  $\phi_*(X_i) = R^t_i X'_t$ induces a  Lie algebra
        homomorphism of Vessiot algebras.
\end{proof}

\begin{Remark}
	Let
	$\{\,\thetaX^i,\, \hpi^a,\, \cpi^\alpha \,\}$ and $\{\,\thetaY^i,\, \hpi^a,\, \cpi^\alpha\,\}$
	be two coframes  which are 4-adapted to the Darboux pair
	$\{\, \hV, \cV \, \}$ and which satisfy  the structure equations \EqRef{Fourth6} and
	\EqRef{Fourth7}. Then it is not difficult to check that the commutativity of the
        dual  vector fields $X_i$ and $Y_j$ is equivalent to the supposition
        that the change of frame \EqRef{Fourth1} defines an automorphism of the Vessiot algebra,
	that is,
\begin{equation}
	Q^i_\ell C^\ell_{jk} = C^i_{lm} Q^\ell_j Q^m_k.
\EqTag{Fourth18}	
\end{equation}
        In the  next section we shall need all the derivatives of the matrix $Q^i_j$. The  $\hpi^a$ and $\cpi^\alpha$
	components of $\extd Q^i_j$ are  easily  determined by substituting \EqRef{Fourth1} into \EqRef{Fourth6} and
	comparing the 	result with \EqRef{Fourth7}.  If  we then  take into account \EqRef{Fourth2} and \EqRef{Fourth18}
	we  find that
\begin{equation}
	\extd Q^i_j
	= Q^\ell_j M^i_\ell - Q^i_\ell N^\ell_j + C^i_{j\ell}Q^\ell_k \thetaX^j ,
\EqTag{Fourth19}
\end{equation}
	where $M^i_j =M^i_{aj}\hpi^a$ and $N^i_j = N^i_{\alpha j}\cpi^\alpha$. We note, also for future use, that
\begin{equation}
	A^i_{ab} = Q^i_j E^j_{ab}	
	\quad\text{and}\quad
        B^i_{\alpha\beta} =Q^i_j F^j_{\alpha\beta}.
\EqTag{Fourth20}
\end{equation}
	where $E^j_{ab}$ and $F^j_{\alpha\beta}$ are defined by \EqRef{First10}.

\end{Remark}	

\newpage
\subsection[The proof of Theorem \StRef{CofrAdapt1} ]{The fifth adapted coframe for a Darboux pair and the proof of Theorem \StRef{CofrAdapt1} }

	Let $\{\,\hV, \cV\,\}$ be a Darboux pair on $M$  and let
	$\{\,\thetaX^i,\hpi^a, \cpi^\alpha \,\}$ and $\{\,\thetaY^i,\hpi^a, \cpi^\alpha\,\}$ be local coframes on $M$
	which are 4-adapted to the Darboux pair $\{\,\hV, \cV\,\}$ and which
	therefore satisfy the structure equations \EqRef{Fourth6} and \EqRef{Fourth7}.	
        In this section we shall prove that it is possible to define forms
\begin{gather}
        \htheta^i = \hR^i_j\thetaX^j + \phi^i_a\,\hpi^a
	\quad\text{and}\quad
	\ctheta^i = \cR^i_j\thetaY^j + \psi^i_a\,\cpi^a,
\EqTag{Fifth1}
\\[2\jot]
	\text{where}\quad
	\hR^i_j,\,\phi^i_a\in \Inv(\hV)\quad\text{and}\quad
	\cR^i_j,\,\psi^i_\alpha\in \Inv(\cV),
\notag
\end{gather}
	which satisfy structure equations
\begin{align}
   \extd \htheta^i
&	=
	\frac12   G^i_{\alpha\beta}\, \cpi^\alpha \wedge \cpi^\beta
        +\frac12 C^i_{jk}\htheta^j\wedge \htheta^k
\EqTag{Fifth2}
\\[-3\jot]
\intertext{and} \notag
\\[-10\jot]
	\extd \ctheta^i
&	=
	\frac12 H ^i_{ab}\, \hpi^a \wedge \hpi^b
        -\frac12 C^i_{jk}\ctheta^j\wedge \ctheta^k.
\EqTag{Fifth3}
\end{align}
	These are the structure equations for the Darboux pair $\{\hV, \cV\,\}$ announced
	in Theorem \StRef{CofrAdapt1}.

	We shall focus on \EqRef{Fifth2} and simply note that  the proof of \EqRef{Fifth3} is similar.
	Our starting point  are the equations \EqRef{Fourth6} and \EqRef{Fourth7}, in particular,
\begin{equation}
	\extd \thetaX^i
	=  \frac12 A^i_{ab}\, \hpi^a \wedge \hpi^b  +
	\frac12 B^i_{\alpha\beta}\, \cpi^\alpha \wedge \cpi^\beta
        +\frac12 C^i_{jk}\thetaX^j\wedge \thetaX^k
	+M^i_{a j}\hpi^a\wedge \thetaX^j.
\EqTag{Fifth5}
\end{equation}
	In what follows it will be useful to introduce the 1-forms and 2-forms
\begin{equation}	
	M^\ell_j =  M^\ell_{aj}\hpi^a,\, N^\ell_j = N^\ell_{\alpha j} \cpi^\alpha,\,
	\,A^\ell = A^\ell_{ab}\hpi^a\wedge\hpi^b  \, F^\ell = F^\ell_{\alpha\beta} \cpi^\alpha\wedge\cpi^\beta.
\EqTag{Fifth70}
\end{equation}
\par
	By setting to  zero the  coefficients of
	$\hpi^a\wedge\cpi^\alpha\wedge \thetaX^i$ and $\cpi^\alpha\wedge \hpi^a\wedge \hpi^b$
	in the equations
        $d^2\thetaX^i =0$,  we  find that
	$\cU_\alpha(M^i_{aj}) = 0$  and  $\cU_\alpha(A^i_{ab})=0$.
	These equations imply that $M^i_{aj}\in \Inv(\hV)$ and $A^i_{ab}\in \Inv(\hV)$
	and therefore (see Lemma \StRef{FirstLem1})
\begin{equation}
	\extd M^i_{aj} = \hU_b(M^i_{aj})\, \hpi^b
	\quad\text{and}\quad
        \extd A^i_{ab} = \hU_c(A^i_{ab})\, \hpi^c.
\EqTag{Fifth6}
\end{equation}
        Bearing these two results in mind,  we then  respectively deduce
	from the  coefficients  of  $\hpi^a\wedge \thetaX^j\wedge \thetaX^k$,
        $\hpi^a\wedge \hpi^b\wedge \thetaX^j$
	and $\hpi^c\wedge\hpi^a\wedge \hpi^b$   in  $d^2\thetaX^i =0$ that
\begin{gather}
	M^\ell_{aj} C^i_{\ell k} + M^\ell_{ak}C^i_{j\ell} -M^i_{a\ell} C^\ell_{jk} = 0,
\EqTag{Fifth7}
\\[2\jot]
	\extd M^i_j -  M^i_\ell \wedge M^\ell_j + C^i_{\ell j} A^\ell = 0, \quad\text{and}
\EqTag{Fifth8}
\\[2\jot]
       \extd A^i - A^\ell M^i_{\ell} = 0.
\EqTag{Fifth9}
\end{gather}
	The proof of Theorem \StRef{CofrAdapt1} hinges upon a detailed analysis of
	equations \EqRef{Fifth7}--\EqRef{Fifth9}.
	We deal first with \EqRef{Fifth7} since this is a purely algebraic constraint.
	It states that for each fixed value of
	$a$, the linear transformation $M_a\: \vess \to \vess$ defined by
\begin{equation*}
        M_a(X_j) = M_{aj}^i X_i
\end{equation*}
	is a derivation or infinitesimal automorphism of the  Vessiot Lie algebra
	$\vess = \vess(\hV, \cV)$. Consequently, to analyze this equation we shall invoke some basic structure theory for 	
	Lie algebras. Specifically, we shall  consider separately the cases where $\vess$ is  semi-simple,
	where $\vess$ is abelian,  and where $\vess$ is
	solvable.  Then we shall  make use of the Levi decomposition of $\vess$ to solve the  general case.
\par
\medskip
\noindent
{\bf Case I.\quad}
	We first consider the case where the Lie algebra $\vess$ is semi-simple.   Here the proof of \EqRef{Fifth2} is
	relatively straightforward and is based upon the fact that every derivation of a semi-simple Lie algebra is
        an inner derivation (see, for example,  Varadarajan\cite{varadarajan:1984a}, page 215).
	In fact, because  the proof of this result is  constructive,
	we can deduce from \EqRef{Fifth7} that there are uniquely defined smooth functions $\phi^\ell_a \in \Inv(\hV)$ such that 	
\begin{equation}
	M^i_{aj} =  \phi^\ell_a C^i_{\ell j }.
\EqTag{Fifth10}
\end{equation}
	The forms $\htheta^i$ defined by
\begin{equation}
	\htheta^i = \thetaX^i + \phi^i _a\hpi^a
\EqTag{Fifth68}
\end{equation}
	then satisfy structure equations of the form
\begin{equation}
        \extd \htheta^i
	=
	\frac12   A^i_{ab}\, \hpi^a \wedge \hpi^b   +
	\frac12	  B^i_{\alpha\beta}\,\cpi^\alpha\wedge\cpi^\beta+
	\frac12   C^i_{jk}\htheta^j\wedge \htheta^k.
\EqTag{Fifth11}
\end{equation}
        For these structure equations, equations \EqRef{Fifth8} reduce to
	$C^i_{\ell j} A^\ell_{ab} =0$.
	Since we are assuming  that $\vess$ is semi-simple, the  center of $\vess$ is trivial
	and therefore this  equation
	implies that $ A^\ell_{ab} =0$.
	The structure equations \EqRef{Fifth11} then reduce to the form \EqRef{Fifth5}, as desired.
\par
\medskip
\noindent
{\bf Case II. \quad}
	Now we consider the  other extreme case, namely,
	the case  where $\vess$ is abelian.
	Strictly speaking, we need not treat this as a separate case but the analysis here will simplify
        our  subsequent discussion of the case where $\vess$ is solvable.  When $\vess$ is abelian
	the structure equations \EqRef{Fifth5} are
\begin{equation}
\extd \thetaX^i
	=  \frac12 A^i_{ab}\, \hpi^a \wedge \hpi^b  +
	\frac12 B^i_{\alpha\beta}\, \cpi^\alpha \wedge \cpi^\beta
	+M^i_{a j}\hpi^a\wedge \thetaX^j,
\EqTag{Fifth12}
\end{equation}
        equation \EqRef{Fifth7} is vacuous  and,
	by virtue of  \EqRef{Fifth6}, 	
        \EqRef{Fifth8} simplifies to
\begin{equation}
	\extd M^i_j -  M^i_\ell \wedge M^\ell_j = 0,
	\quad\text{where}\quad
         M^i_j = M^i_{aj} \hpi^a.
\EqTag{Fifth55}
\end{equation}
         By the Frobenius theorem
	 we conclude that there are locally defined
	 functions $R^i_j\in \Inv(\hV)$ such that
\newcommand\Ao{{\accentset{0\rule[-1pt]{0pt}{0pt}}{A}}{}}
\newcommand\Bo{{\accentset{0\rule[-1pt]{0pt}{0pt}}{B}}{}}
\newcommand\Co{{\accentset{0\rule[-1pt]{0pt}{0pt}}{C}}{}}
\newcommand\Mo{{\accentset{0\rule[-1pt]{0pt}{0pt}}{M}}{}}
\begin{equation}
	\extd( R^i_j)  +R^i_\ell \,  M^\ell_j = 0\quad\text{and}\quad
	\det(R^i_j)\neq 0.
\EqTag{Fifth56}
\end{equation}
	(For this precise application of the Frobenius theorem see, for example, Flanders \cite{flanders:1963a}, page 102.)
	We remark that this is the first (and only) step in all our coframe adaptations
	that require the solution  to systems of linear ordinary differential equations
\par
         The  forms $\htheta^i_0 = R^i_j \thetaX^j$ are then easily seen to satisfy
\begin{equation}
        \extd \htheta^i_0 =
	\frac12 \Ao^i_{ab}\, \hpi^a\wedge \hpi^b +
	\frac12 \Bo^i_{\alpha\beta} \, \cpi^\alpha \wedge \cpi^\beta.
\EqTag{Fifth13}
\end{equation}
	Since the structure functions $\Ao^i_{ab} = R^i_j \, A^j_{ab} \in \Inv(\hV)$
	it follows, either directly from \EqRef{Fifth13} or from
	\EqRef{Fifth9} (with $M^i_{a\ell} =0$)
	that the 2-forms
\begin{equation}
	\chi^i= \frac{1}{2}\Ao^i_{ab}\, \hpi^a\wedge \hpi^b
\EqTag{Fifth57}
\end{equation}
	are all $d$-closed.  If we pick 1-forms $\phi^i = \phi^i_a\,\hpi^a$, with $\phi^i_a\in \Inv(\hV)$,
	such that $\extd \phi^i =-\chi^i$, then the forms
\begin{equation}
        \htheta^i  = \htheta^i_0   + \phi^i_a\,\hpi^a =  R^i_j\thetaX^j + \phi^i_a\,\hpi^a
\EqTag{Fifth66}
\end{equation}
        will satisfy the  required structure equations \EqRef{Fifth2}, with $C^i_{jk}=0$. By applying the usual homotopy
	formula for the de Rham complex, we see that this last coframe adaptation can be implemented by quadratures.
\par
\medskip
\noindent
{\bf Case III. \quad}
	Now we suppose that $\vess$ is solvable.
	Recall that a  Lie algebra $\lieg$ is said to  be
	$p$-step solvable if the derived algebras $\lieg^{(i)}=[\lieg^{(i-1)}, \lieg^{(i-1)}]$ satisfy
\begin{equation*}
        \lieg = \lieg^{(0)} \supset \lieg^{(1)}\supset \lieg^{(2)}\dots
	\supset \lieg^{(p-1)} \supset \lieg^{(p)} = \{0\}.
\EqTag{Fifth14}
\end{equation*}
	The annihilating subspaces $\Lambda^{(i)} =\ann(\lieg^{(i)})$ therefore satisfy
\begin{equation}
        \{0\} = \Lambda^{(0)} \subset  \Lambda^{(1)} \subset \Lambda^{(2)}  \dots \subset \Lambda^{(p-1)}
	\subset \Lambda^{(p)} = \lieg^*
\EqTag{Fifth15}
\end{equation}
	and,  because the  $\lieg^{(i)}$ are all ideals,
\begin{equation}
	\extd 	\Lambda^{(i)} \subset   \Lambda^{(i-1)} \otimes \Lambda^{(i)}.
\EqTag{Fifth16}
\end{equation}
	Let $s = \dim(\lieg)$ and  $s_i = \dim \Lambda^{(i)}$.
        If  $M\:\lieg \to \lieg$ is a derivation,
	then a simple induction shows that $M\:\lieg^{(i)} \to \lieg^{(i)}$ and therefore
\begin{equation}
	M^*\: \Lambda^{(i)} \to \Lambda^{(i)}.
\EqTag{Fifth19}
\end{equation}
	Thus the matrix representing $M^*$,  with respect to any basis
	for $\lieg^*$ adapted to the flag $\EqRef{Fifth15}$, is block triangular.
\par
        We apply these observations to the forms $\{\, \thetaX^i\,\}$
	and the structure equations \EqRef{Fifth5}.  The forms  $\{\, \thetaX^i\, \}$,   restricted to the vectors $X_i$, define a basis for $\vess^*$
        and consequently, by a {\it constant} coefficient  change of basis,
	we can suppose that the basis $\{\, \thetaX^i\, \}$ is adapted to the derived
        flag \EqRef{Fifth15}. Since the notation for this adaptation becomes rather cumbersome  in the  general case
	of a  $p$-step 	Lie algebra, we shall consider
        just  the cases  where $\vess$  is a $2$-step  or a $3$-step  solvable Lie algebra.  The construction of the
	coframe  with structure equations \EqRef{Fifth2} in
	these two special cases will make the nature of the  general construction apparent.
\par
	In the case  where $\vess$ is a 2-step solvable Lie algebra,
	we  begin with a 4-adapted coframe
\begin{gather*}
       \{\,\theta^1_1,\dots, \theta^{s_1}_1,\theta^{s_1+1}_2,\dots ,\theta^{s_2}_2, \hpi^a, \cpi^\alpha \,\},
\\[-4\jot]\intertext{where}
\\[-11\jot]\notag
        \text{span}\{\,\thetaX^1,\dots, \thetaX^{s} \,\} =
	\text{span} \{\,\theta^1_1,\dots, \theta^{s_1}_1,\theta^{s_1+1}_2,\dots , \theta^{s_2}_2\,\}\quad \text{over $\real$},
\end{gather*}
         and  where $\{\,\theta^1_1,\dots, \theta^{s_1}_1\}$ is a basis for $\Lambda^{(1)}(\vess)$.
	 In this basis the structure equations  \EqRef{Fifth5} become
\begin{subequations}
\label{eq:Fifth18}
\begin{align}
       \extd \theta_1^r
	&= \frac12 A^r_{ab}\, \hpi^a \wedge \hpi^b  +
	\frac12 B^r_{\alpha\beta}\, \cpi^\alpha \wedge \cpi^\beta
	+M^r_{a s}\, \hpi^a\wedge \theta_1^s, \quad \text{and}
\EqTag{Fifth18a}
\\
	\extd \theta_2^i
	&=  \frac12 A^i_{ab}\, \hpi^a \wedge \hpi^b  +
	\frac12 B^i_{\alpha\beta}\, \cpi^\alpha \wedge \cpi^\beta
        +\frac12 C^i_{rs}\, \theta_1^r\wedge \theta_1^s  +C^i_{rj}\,\theta_1^r\wedge \theta_2^j
\EqTag{Fifth18b}
\\
	& + M^i_{a r}\,\hpi^a\wedge \theta^r_1 + M^i_{a j}\,\hpi^a\wedge \theta^j_2.
\notag
\end{align}
\end{subequations}
	In these equations the indices $r,s$ range from 1 to $s_1$ and the indices $i,j$ range from $s_1+1$ to $s_2$.
	On account of  \EqRef{Fifth16}, there are no
        quadratic  $\theta$ terms  in \EqRef{Fifth18a} because
	$\theta_1^r\in\Lambda^{(1)}(\vess)$  and there are no $\theta^i_2 \wedge \theta^j_2$ terms in
	\EqRef{Fifth18b} because $\theta_1^2\in\Lambda^{(2)}(\vess)$.
	There are no  $\hpi^a\wedge \theta^i_2$ terms in \EqRef{Fifth18a}
	by virtue of \EqRef{Fifth19}.
\par
        Since the structure equations \EqRef{Fifth18a} are identical in form to the structure equations  \EqRef{Fifth12}
	for the abelian  case,  we  can  invoke the  arguments there to define  new forms
\begin{equation}
	\htheta^r_1 =  R^r_s\theta^s_1 +  \phi^r_a\hpi^a
\EqTag{Fifth20}
\end{equation}
	so that the structure equations \EqRef{Fifth18} reduce to
\begin{subequations}
\label{eq:Fifth21}
\begin{align}
       \extd \htheta_1^r
	&=
	\frac12 \Bo^r_{\alpha\beta}\, \cpi^\alpha \wedge \cpi^\beta, \quad\text{and}	
\EqTag{Fifth22}
\\
	\extd \theta_2^i
	&=  \frac12 \Ao^i_{ab}\, \hpi^a \wedge \hpi^b  +
	\frac12 \Bo^i_{\alpha\beta}\, \cpi^\alpha \wedge \cpi^\beta
        +\frac12 \Co^i_{rs}\,\htheta_1^r\wedge \htheta_1^s  +\Co^i_{rj}\,\htheta_1^r\wedge \theta_2^j
\EqTag{Fifth23}
\\
	& + \Mo^i_{a r}\, \hpi^a\wedge \htheta^r_1 + \Mo^i_{a j}\, \hpi^a\wedge \theta^j_2.
\notag
\end{align}
\end{subequations}
	It is important to  track the coordinate dependencies of the  coefficients
	in these structure equations as we perform  these frame changes.
	Since the coefficients  $A^i_{ab}$, $M^i_{ar}$ and $M^i_{aj}$ in \EqRef{Fifth18b}  and the coefficients
	$R^r_s$ and $\phi^r_a$ in
        \EqRef{Fifth20} are all $\hV$ first integrals,  it is easily checked that the same is true of the coefficients
	$\Ao^i_{ab}$, $\Co^i_{rs}$, $\Co^i_{rj}$, $\Mo^i_{a r}$ and $\Mo^i_{a j}$  in
	$\EqRef{Fifth23}$.
\par
	The  coefficients  of $\hpi^a\wedge \hpi^b\wedge \theta^i_2$ and
	$\hpi^a\wedge \hpi^b \wedge \hpi^c$  in $d^2 \theta^i_2 = 0$   give the  same equations as we  had in the
        abelian case and consequently we can define
\begin{equation}
	\htheta^i_2 = R^i_j\theta^j_2 + \phi^i_a \hpi^a
\end{equation}
	so as to eliminate the   $\hpi^a \wedge \hpi^b $ and $\hpi^a\wedge \theta^j_2$
	terms in \EqRef{Fifth23}. The structure equations are now
\newcommand\Aone{{\accentset{1\rule[-1pt]{0pt}{0pt}}{A}}}
\newcommand\Bone{{\accentset{1\rule[-1pt]{0pt}{0pt}}{B}}}
\newcommand\Cone{{\accentset{1\rule[-1pt]{0pt}{0pt}}{C}}}
\newcommand\Mone{{\accentset{1\rule[-1pt]{0pt}{0pt}}{M}}}
\begin{subequations}
\label{eq:Fifth24}
\begin{align}
       \extd \htheta_1^r
	&=
	\frac12 \Bo{}^r_{\alpha\beta}\, \cpi^\alpha \wedge \cpi^\beta, \quad\text{and}	
\EqTag{Fifth25}
\\
	\extd \htheta_2^i
	&=
	\frac12 \Bone{}^i_{\alpha\beta}\, \cpi^\alpha \wedge \cpi^\beta
        +\frac12 \Cone{}^i_{rs}\,\htheta_1^r\wedge \htheta_1^s  +\Cone{}^i_{rj}\, \htheta_1^r\wedge \htheta_2^j +
	\Mone{}^i_{a r}\,\hpi^a\wedge \htheta^r_1 .
\end{align}
\end{subequations}
	Finally,  from the coefficients of
	$\hpi^a\wedge \hpi^b\wedge \htheta^r_1$ in  $d^2 \htheta^i_2 = 0$ we find that
\begin{equation}
	\extd \Mone{}^i_{b} = 0,
	\quad\text{where}\quad
	\Mone{}^i_{b} =  \Mone{}^i_{br}\, \hpi^r,
\EqTag{Fifth71}
\end{equation}
	and  therefore, again by the de Rham homotopy formula,
	there are functions $R^i_r \in \Inv(\hV)$ such that $M^i_{ar} = \hU_a(R^i_r)$.
	The  change of frame
\begin{equation}
        \htheta^i_2=  \Hat{\Hat \theta}^i_2 +R^i_r \htheta^r_1
\EqTag{Fifth58}
\end{equation}
	then leads to the desired structure equations
\newcommand\Btwo{{\accentset{2\rule[-1pt]{0pt}{0pt}}{B}}}
\newcommand\Ctwo{{\accentset{2\rule[-1pt]{0pt}{0pt}}{C}}}
\begin{subequations}
\label{eq:Fifth26}
\begin{align}
       \extd \htheta_1^r
	&=
	\frac12 \Bo{}^r_{\alpha\beta}\, \cpi^\alpha \wedge \cpi^\beta, \quad\text{and} 	
\EqTag{Fifth27}
\\
	\extd \Hat{\Hat \theta}_2^i
	&=
	\frac12 \Btwo{}^i_{\alpha\beta}\, \cpi^\alpha \wedge \cpi^\beta
        +\frac12 \Ctwo{}^i_{rs}\htheta_1^r\wedge \htheta_1^s  +\Ctwo{}^i_{rj}\htheta_1^r\wedge \Hat{\Hat \theta}_2^j.
\EqTag{Fifth28}
\end{align}
\end{subequations}
	At this  point  it is a simple matter to check that the equations $d^2\Hat{\Hat \theta}_2^i =0$ force the
        coefficients  $\Ctwo{}^i_{rs}$  and  $\Ctwo{}^i_{rj}$ (which belong to $\Inv(\hV)$)  to be  constant. Equations
	\EqRef{Fifth26}  establish \EqRef{Fifth2} for the case of a 2-step solvable Vessiot algebra.
\par
        When   $\vess$ is a 3-step solvable Lie algebra  we assume that the 4-adapted coframe
	$\{\, \theta^r_1,\, \theta^u_2,\, \theta^i_3,\, \hpi^a,\, \cpi^\alpha\,\}$
	is adapted to the flag \EqRef{Fifth15} in the sense that
\begin{equation*}
	\Lambda^{(1)}= \text{span}\{\, \theta^r_1 \, \}, \quad
	\Lambda^{(2)}= \text{span}\{\, \theta^r_1,\, \theta^u_2\, \} \quad\text{and}\quad
	\Lambda^{(3)}= \text{span}\{\, \theta^r_1,\, \theta^u_2,\, \theta^i_3\,\}.
\end{equation*}
	The structure equations \EqRef{Fifth32} are then of the form
\begin{subequations}
\label{eq:Fifth30}
\begin{align}
       \extd \theta_1^r
	&= \frac12 A^r_{ab}\, \hpi^a \wedge \hpi^b  +
	\frac12 B^r_{\alpha\beta}\, \cpi^\alpha \wedge \cpi^\beta
	+M^r_{a s}\,\hpi^a\wedge \theta_1^s,
\EqTag{Fifth31}
\\[2\jot]
	\extd \theta_2^u
	&=  \frac12 A^u_{ab}\, \hpi^a \wedge \hpi^b  +
	\frac12 B^u_{\alpha\beta}\, \cpi^\alpha \wedge \cpi^\beta
        +\frac12 C^u_{rs}\,\theta_1^r\wedge \theta_1^s  +C^u_{rv}\,\theta_1^r\wedge \theta_2^v
\EqTag{Fifth32}
\\
	& + M^u_{a r}\,\hpi^a\wedge \theta^r_1 + M^u_{a v}\,\hpi^a\wedge \theta^v_2, \quad\text{and}
\notag
\\[2\jot]
	\extd \theta_3^i
	&=  \frac12 A^i_{ab}\, \hpi^a \wedge \hpi^b  +
	\frac12 B^i_{\alpha\beta}\, \cpi^\alpha \wedge \cpi^\beta
\EqTag{Fifth33}
\\\notag
        &+\frac12 C^i_{rs}\,\theta_1^r\wedge \theta_1^s  + C^i_{ru}\,\theta_1^r\wedge \theta_2^u
	+  C^i_{rj}\,\theta^r_1\wedge \theta^j_3 +
	\frac12 C^i_{uv}\,\theta_2^u\wedge \theta_2^v +  C^i_{uj}\,\theta_2^u\wedge  \theta_3^j
\\
	& + M^i_{a r}\,\hpi^a\wedge \theta^r_1 +  M^i_{au}\,\hpi^a\wedge \theta^u_2 + M^i_{a j}\,\hpi^a\wedge \theta^j_3.
\notag
\end{align}
\end{subequations}
	In these equations $r,s$ range from 1 to $s_1$,
	$u,v$ from $s_1+1$ to $s_2$, and $i,j$ from  $s_2+1$ to $s_3$.	
	Note that the form of the   structure equations \EqRef{Fifth31}--\EqRef{Fifth33}
	is preserved by changes of  coframe of the type
\begin{gather*}
       \bftheta_1 \to  \bfR_{11} \bftheta_1 + \bfphi_1\bfhpi,\quad
       \bftheta_2 \to  \bfR_{12} \bftheta_1 + \bfR_{22}\bftheta_2 + \bfphi_2\bfhpi,\quad\text{and}
\\
       \bftheta_3 \to  \bfR_{13} \bftheta_1 +\bfR_{23} \bftheta_2 + \bfR_{33} \bftheta_3 + \bfphi_3 \bfhpi,
\end{gather*}
	where the coefficients  $\bfR_{ij}\in \Inv(\hV)$. Exactly as in the previous case of a
	2-step  solvable algebra, we  use such a change of coframe  for $\bftheta_1, \bftheta_2$ to reduce
	the  structure equations \EqRef{Fifth30} to
\begin{subequations}
\label{eq:Fifth35}
\begin{align}
       \extd \htheta_1^r
	&= \frac12 B^r_{\alpha\beta}\, \cpi^\alpha \wedge \cpi^\beta ,
\EqTag{Fifth36}
\\[2\jot]
	\extd \htheta_2^u
	&=
	\frac12 B^u_{\alpha\beta}\, \cpi^\alpha \wedge \cpi^\beta
        +\frac12 C^u_{rs}\,\htheta_1^r\wedge \htheta_1^s  +C^u_{rv}\,\htheta_1^r\wedge \htheta_2^v, \quad \text{and}
\EqTag{Fifth37}
\\[2\jot]
	\extd \theta_3^i
	&=  \frac12 A^i_{ab}\, \hpi^a \wedge \hpi^b  +
	\frac12 B^i_{\alpha\beta}\, \cpi^\alpha \wedge \cpi^\beta
\EqTag{Fifth38}
\\
        &+\frac12 C^i_{rs}\, \htheta_1^r\wedge \htheta_1^s  + C^i_{ru}\, \htheta_1^r\wedge \htheta_2^u
	+  C^i_{rj}\, \htheta^r_1\wedge \theta^j_3 +
	\frac12 C^i_{uv}\, \htheta_2^u\wedge \htheta_2^v +  C^i_{uj}\, \htheta_2^u\wedge  \theta_3^j
\notag
\\
	& + M^i_{a r}\, \hpi^a\wedge \htheta^r_1 +  M^i_{au}\, \hpi^a\wedge \htheta^u_2 + M^i_{a j}\, \hpi^a\wedge \theta^j_3.
\notag
\end{align}
\end{subequations}
	Again, as in Case {\bf II}, a change of frame $\bfhtheta_3  = \bfR_{33} \bftheta_3 + \bfphi_3 \bfhpi$ leads to the
	simplification of \EqRef{Fifth38} to
\begin{align}
\extd \htheta_3^i
	&=
	\frac12 B^i_{\alpha\beta}\, \cpi^\alpha \wedge \cpi^\beta
	+ M^i_{a r}\, \hpi^a\wedge \htheta^r_1 +  M^i_{au}\, \hpi^a\wedge \htheta^u_2
\EqTag{Fifth40}
\\
        &+\frac12 C^i_{rs}\, \htheta_1^r\wedge \htheta_1^s  + C^i_{ru}\, \htheta_1^r\wedge \htheta_2^u
	+  C^i_{rj}\, \htheta^r_1\wedge \htheta^j_3 +
	\frac12 C^i_{uv}\, \htheta_2^u\wedge \htheta_2^v +  C^i_{uj}\, \htheta_2^u\wedge  \htheta_3^j.
\notag
\end{align}
      Finally, just as in the reduction from \EqRef{Fifth24} to \EqRef{Fifth26}, we  use a change of  frame
      $\bfhtheta_3 \to \hat \bftheta_3 +\bfR_{13} \bfhtheta_1 +\bfR_{23} \bfhtheta_2$  to transform \EqRef{Fifth40} to
\begin{align}
\extd \htheta_3^i
	&=
	\frac12 B^i_{\alpha\beta}\, \cpi^\alpha \wedge \cpi^\beta
\EqTag{Fifth41}
\\	&+\frac12 C^i_{rs}\,\htheta_1^r\wedge \htheta_1^s  + C^i_{ru}\,\htheta_1^r\wedge \htheta_2^u   +
	C^i_{rj}\, \htheta^r_1\wedge \htheta^j_3 +
	\frac12 C^i_{uv}\, \htheta_2^u\wedge \htheta_2^v +  C^i_{uj}\, \htheta_2^u\wedge  \htheta_3^j.
\notag
\end{align}
	Equations  \EqRef{Fifth36}, \EqRef{Fifth37} and \EqRef{Fifth41} give the desired result.
	We note  once more that the structure functions $C^*_{**}$ in  \EqRef{Fifth38} and\EqRef{Fifth40} are not necessarily constant  but
	they are constant in \EqRef{Fifth41}.
\par
	The reduction of the
        structure  equations  for a general $p$-step solvable algebra follows this construction and can be formally defined
        by induction on $p$.
\par
\medskip
\noindent
{\bf Case IV:}  \quad Finally, we  consider the  case where $\vess$ is a generic Lie algebra.
	We  use the fact that  every  real Lie algebra
	$\lieg$ admits a  Levi decomposition into a semi-direct sum $\lieg= \lies \oplus \lier$, where $\lier$ is the radical of
        $\lieg$ and $\lies$ is a semi-simple  subalgebra of $\lieg$.  The radical $\lier$ is the unique maximal solvable ideal in
	$\lieg$ -- the  semi-simple component $\lies$ in the  Levi decomposition is not unique. The dual space $\lieg^*$
	then decomposes according to
\begin{equation}
	\lieg^* = \ann(\lier) \oplus \ann(\lies).
\EqTag{Fifth42}
\end{equation}
	The fact that $\lier$ is an ideal implies that
\begin{equation}
       \extd \ann(\lier) \subset \Lambda^2(\ann(\lier))
	\quad\text{and}\quad
        \extd \ann(\lies) \subset \lieg^* \otimes\ann(\lies).
\EqTag{Fifth43}
\end{equation}
	If  $M\: \lieg \to \lieg$ is a derivation, then  $M$ preserves $\lier$ and hence $M^*\:\ann(\lier) \to \ann(\lier)$.
\par
	Now choose, by a constant coefficient change of basis, 1-forms $\{\,\theta^r_1, \theta^i_2  \, \}$  adapted to the decomposition
	\EqRef{Fifth42} with $\ann(\lier) = \text{span}\{\,\theta^r_1  \,\}$ and
	$\ann(\lies) = \text{span}\{\,\theta^i_2  \,\}$. In view of \EqRef{Fifth43}, the  structure
	equations \EqRef{Fifth5} become
\begin{subequations}
\label{eq:Fifth44}
\begin{align}
       \extd \theta_1^r
	&= \frac12 A^r_{ab}\, \hpi^a \wedge \hpi^b  +
	\frac12 B^r_{\alpha\beta}\, \cpi^\alpha \wedge \cpi^\beta
	+ \frac12 C^r_{st}\,\theta_1^s \wedge \theta_1^t + M^r_{a s}\,\hpi^a\wedge \theta_1^s,
\EqTag{Fifth44a}
\\[2\jot]
	\extd \theta_2^i
	&=  \frac12 A^i_{ab}\, \hpi^a \wedge \hpi^b  +
	\frac12 B^i_{\alpha\beta}\, \cpi^\alpha \wedge \cpi^\beta
        +\frac12 C^i_{jk}\,\theta_2^j\wedge \theta_2^k  +C^i_{rj}\,\theta_1^r\wedge \theta_2^j
\EqTag{Fifth44b}
\\
	& + M^i_{a r}\,\hpi^a\wedge \theta^r_1 + M^i_{a j}\,\hpi^a\wedge \theta^j_2.
\notag
\end{align}
\end{subequations}
	In these  equations the  indices $r$, $s$, $t$ range from 1 to $s_0=\dim(\ann(\lier))$ and $i$, $j$, $k$ range from
	$s_0+1$ to $s=\dim(\lieg)$.
\par
	Since the structure constants  $C^r_{st}$ are those for the semi-simple Lie algebra $\lies$,
	we can return to the analysis	presented in Case {\bf I} to prove that there is a
	change of  coframe $\htheta_1^r =  \theta_1^r + \phi^r_a \hpi^a$  (see \EqRef{Fifth68}) which reduces
	the structure equations \EqRef{Fifth44} to
\begin{subequations}
\label{eq:Fifth46}
\begin{align}
       \extd \htheta_1^r
	&=
	\frac12 B^r_{\alpha\beta}\, \cpi^\alpha \wedge \cpi^\beta
	+ \frac12 C^r_{st}\,\htheta_1^s \wedge \htheta_1^t,  \quad\text{and}
\EqTag{Fifth46a}
\\[2\jot]
	\extd \theta_2^i
	&=  \frac12 A^i_{ab}\, \hpi^a \wedge \hpi^b  +
	\frac12 B^i_{\alpha\beta}\, \cpi^\alpha \wedge \cpi^\beta
        +\frac12 C^i_{jk}\,\theta_2^j\wedge \theta_2^k  +C^i_{rj}\,\htheta_1^r\wedge \theta_2^j
\EqTag{Fifth46b}
\\
	& + M^i_{a r}\,\hpi^a\wedge \htheta^r_1 + M^i_{a j}\,\hpi^a\wedge \theta^j_2.
\notag
\end{align}
\end{subequations}
	The  structure functions $C^r_{st}$, $C^i_{jk}$ and $C^i_{rj}$  in  \EqRef{Fifth46} are  identical to the  		
	corresponding structure constants in  \EqRef{Fifth44}.
	One  now  checks that the  $\htheta^r_1$ terms in \EqRef{Fifth46b} do not  effect the  arguments made in
	Cases {\bf II}  and {\bf III}. Thus, by a change of frame $\bfhtheta_2= \bfR\bftheta_2  + \bfphi \bfhpi$, one can reduce
	\EqRef{Fifth46} to
\begin{subequations}
\label{eq:Fifth47}
\begin{align}
       \extd \htheta_1^r
	&=
	\frac12 B^r_{\alpha\beta}\, \cpi^\alpha \wedge \cpi^\beta
	+ \frac12 C^r_{st}\,\htheta_1^s \wedge \htheta_1^t,  \quad\text{and}
\EqTag{Fifth47a}
\\[2\jot]
	\extd \htheta_2^i
	&=    \frac12 B^i_{\alpha\beta}\, \cpi^\alpha \wedge \cpi^\beta
        +\frac12 C^i_{jk}\,\htheta_2^j\wedge \htheta_2^k  +C^i_{rj}\,\htheta_1^r\wedge \theta_2^j
	+ M^i_{a r}\,\hpi^a\wedge \htheta^r_1 .
\EqTag{Fifth47b}
\end{align}
\end{subequations}
	Finally,  by  a now familiar  computation,
	one sees  that  a  change of frame  $\bfhtheta_2\to \hat \bftheta_2  + \bfR \bfhtheta_1$
	allows one to eliminate the  $\hpi^a\wedge \htheta^r_1$  terms in  \EqRef{Fifth47b}.
\par
	This  completes our derivation of the structure equations \EqRef{Fifth2}
	and the proof of Theorem \StRef{CofrAdapt1} is, at last, finished.

\begin{Remark}
\StTag{Fifth24}
	We collect together a few additional
	formulas which will be needed  in the next section for the proof of
	Lemma \StRef{SuperForm5} and the construction of the superposition formula.
	If the forms $\htheta^i$ and $\ctheta^i$ in \EqRef{Fifth1}
	satisfy \EqRef{Fifth2} and \EqRef{Fifth3}, then it easy to  check that coefficients
	$\hR^i_j$,  $\cR^i_j$,  $\phi^i = \phi^i_a\hpi^a$ and $\psi^i = \psi^i_\alpha\cpi^\alpha$ satisfy
	(see \EqRef{Fifth70})
\begin{subequations}
\label{eq:Fifth59}
\begin{gather}
	\extd\hR^i_j = -C^i_{\ell k}\, \hR^\ell_j\,\phi^k - \hR^i_\ell \, M^\ell_j,
	\quad
	\extd\cR^i_j = C^i_{\ell k}\,  \cR^\ell_j \psi^k - \cR^i_\ell N^\ell_j,
\EqTag{Fifth59a}
\\[2\jot]
	\hR^i_\ell B^\ell_{\alpha\beta} = G^i_{\alpha\beta},
	\quad
	\cR^i_\ell\, E^\ell_{ab} = H^i_{ab},
\EqTag{Fifth59b}
\\[2\jot]
	\extd \phi^i  =\frac12C^i_{jk} \phi^j\wedge \phi^k -\frac12 \hR^i_\ell \,A^\ell,
	\quad
	\extd \psi^i  =-\frac12C^i_{jk} \psi^j\wedge \psi^k -\frac12 \cR^i_\ell \, F^\ell.
\EqTag{Fifth59d}
\end{gather}
\end{subequations}
	The computations leading to \EqRef{Fifth59} also show that
	the matrices $\hR^i_j$ and $\cR^i_j$ define automorphisms of the Vessiot algebra, that is
\begin{equation}
	\hR^i_\ell \,C^\ell_{jk} =C^i_{\ell m}\, \hR^\ell_j\,\hR^m_k,
	\quad\text{and}\quad
	\cR^i_\ell \,  C^\ell_{jk} =C^i_{\ell m}\, \cR^\ell_j \, \cR^m_k.
\EqTag{Fifth53}
\end{equation}
	Finally, a series of straightforward computations,
	based on \EqRef{Fourth18}, \EqRef{Fourth19}, \EqRef{Fifth59a} and \EqRef{Fifth53}, shows  that the matrix
\begin{gather}
	\lambda = \hR Q \cR^{-1} \quad\text{satisfies}
\EqTag{Fifth60}
\\[2\jot]
	\lambda^i_\ell C^\ell_{jk} = C^i_{\ell m} \lambda^\ell_j \lambda^m_k
	\quad\text{and}\quad
	\extd\lambda^i_j
        =C^i_{\ell m}\, \lambda^m_j\, \htheta^\ell  + \lambda^i_h\,C^h_{mj}\,\psi^m.
\EqTag{Fifth61}
\end{gather}
\qed
\par
\end{Remark}

	We conclude this section by  determining  the residual freedom in the determination of  the
	1-forms  $\htheta^i$. Specifically, we compute
        the  group of coframe transformations  which fix the forms $\hpi^a$ and $\cpi^\alpha$ and
	which transform the $\htheta^i$ by
\begin{equation}
        \tilde \theta^i  = \Lambda^i_j \htheta^i + \sigma^i,   \quad\text{where}\quad  \sigma^i= S^i_a\cpi^a
\EqTag{Fifth50}
\end{equation}
        in such a manner as to  preserve the form of the structure equations \EqRef{Fifth2}.
\par	
	If we take the exterior derivative of \EqRef{Fifth50} and substitute into the structure equations for
	$\tilde\theta^i$ we  find  first, from the  $\cpi^\alpha \wedge \theta^i$ and the $\hpi^a\wedge \cpi^\alpha$ terms,
        that $\Lambda^i_j, S^i_a \in \Inv(\hV)$ and then that
\begin{equation}
       	\extd \Lambda^i_j = C^i_{\ell k}\,  \Lambda^\ell_j \sigma^k ,
	\quad
        \Lambda^i_\ell \,  C^\ell_{jk} =C^i_{\ell m}\, \Lambda^\ell_j \, \Lambda^m_k,
        \quad
	\extd \sigma^i  =\frac12C^i_{jk} \sigma^j\wedge \sigma^k .
\EqTag{Fifth51}
\end{equation}
\par	
	To integrate these equations, let $G$ be a local Lie group  whose Lie algebra is the
	Vessiot algebra $\vess$
        with structure constants $C^i_{jk}$. On   $G$  construct  a  coframe  $\eta^i$ of invariant 1-forms
        with structure equations  $\extd \eta^i  =\frac12C^i_{jk} \eta^j\wedge \eta^k$.
	Then there exists (\cite{griffiths:1974a}, Proposition 1.3)
	a map $\sigma \colon \Inv(\hV) \to G$ such that $\sigma^i = \sigma^*(\eta^i)$.
        Define $S(\hI) = \Ad(\sigma(\hI)$.   Then
\begin{equation*}
        \extd S^i_j=C^i_{\ell k} \sigma^\ell S^k_j
	\quad
	\text{and}\quad
	\extd(\Lambda^i_\ell (S^{(-1)})^\ell_j)  = 0.
\end{equation*}
	and hence the  general solution to  \EqRef{Fifth51} is
\newcommand\Lambdao{{\accentset{0\rule[-1pt]{0pt}{0pt}}{\Lambda}}{}}
\begin{equation}
	\sigma^i=\sigma^*(\eta^i), \quad  \Lambda^i_j = \Lambdao^i_\ell\, S^\ell_j,  \quad S= \Ad(\sigma(\hI)),
\EqTag{Fifth52}
\end{equation}
	where  the matrix $\Lambdao^i_\ell$ is a constant  automorphism  of the Vessiot algebra.
        When $\sigma$ is a constant map, $S$ is constant inner automorphism  so that as far as the  general solution
        \EqRef{Fifth52} is concerned, one can restrict the  $\Lambdao^i_\ell$ to a set of representatives
        of the group of outer automorphisms of the Vessiot algebra.

\begin{Remark}
\StTag{Fifth8}
	The bases for the  infinitesimal Vessiot transformation groups
	$\hGamma = \{\,\hX_i\,\}$ and $\cGamma = \{\, \cX_i \, \}$
	are related by $\cX_j = \lambda^i_j \hX_i$, where $\lambda$ is the matrix \EqRef{Fifth60}.
	Let $\hX_r$, $r= 1\ldots m$, be	
	a basis for the center of $\hGamma$ and pick a complementary set of vectors $\hX_u$, $u=m+1\ldots s$,
	which complete the $\hX_r$ to a basis for $\hGamma$. Choose a similar basis for $\cGamma$.
	Then, because $\lambda$ defines a Lie algebra automorphism, we have that
\begin{equation*}
	\cX_r =  \lambda^t_r \hX_t \quad\text{and}\quad
	\cX_u =   \lambda^v_u \hX_v  + \lambda^t_u \hX_t,
\end{equation*}
	where $t= 1\ldots m$ and $v = m+1\ldots s$.
	The second equation  in \EqRef{Fifth61} now implies that $d\, \lambda^r_s = 0$.
	Consequently, we may always pick bases for the
	infinitesimal Vessiot transformation groups  $\hGamma$ and $\cGamma$ so that
	the infinitesimals generators for  the center coincide.
	In turn, this implies that any vector in the
	center of either infinitesimal Vessiot transformation group is a infinitesimal
	symmetry of {\it both\/} $\hV$ and $\cV$. \qed
\end{Remark}

\newpage

\section{The Superposition Formula for Darboux Pairs}

\subsection{A preliminary result} The proof of the superposition formula for a Darboux pair will depend upon the following technical
	result concerning group actions and Maurer-Cartan forms.

\begin{Theorem}
\StTag{SuperForm1}
	Let $M$ be a manifold and  $G$ a connected Lie group.  Let $\omega^i_L$ and $\omega^i_R$ be
	the left and  right invariant Maurer-Cartan forms on $G$, with structure equations
\begin{equation}
	\extd \omega_L^i = \frac12 C^i_{jk}\, \omega_L^j \wedge \omega_L^k
	\quad\text{and}\quad
	\extd \omega_R^i = -\frac12 C^i_{jk}\, \omega_R^j \wedge \omega_R^k.
\EqTag{SuperForm2}
\end{equation}
	Let $X^L_i$ and $X^R_i$ be the dual basis of  left and right invariant vector fields on $G$.

\smallskip
\noindent
{\bf [i]}  Suppose that there are   right and left  commuting,  regular, free  group actions of $G$ on $M$,
\begin{equation}
	\hmu \colon G \times M \to M
	\quad\text{and}\quad
	\cmu \colon G \times M  \to M,
\end{equation}
	with common orbits.
	Denote the infinitesimal generators of these actions by
	$(\hmu_x)_*(X^L_i) = \hX_i$  and
	$(\cmu_x)_*(X^R_i) = \cX_i.$

\smallskip
\noindent
{\bf [ii]} Suppose there are 1-forms  $\homega^i$ and $\comega^i$  on $M$ such that
	{\bf [a]} $\homega^i(\hX_j) = \delta^i_j$ and $\comega^i(\cX_j) = \delta^i_j$;
	{\bf [b]} $\text{span} \{\, \homega^i\,\} = \text{span} \{\,\comega^i\,\}$ with
	$\homega^i(x_0) =   \comega^i(x_0)$
	at some fixed point of $x_0 \in M$  and;  {\bf [c]}
\begin{equation}
	\extd \homega^i = \frac12 C^i_{jk}\ \homega^j\wedge \homega^k
	\quad\text{and}\quad
	\extd \comega^i= -\frac12 C^i_{jk}\ \comega^j \wedge \comega^k.
\EqTag{SuperForm3}
\end{equation}
	Then about each point $x_0$ of $M$ there is an open $G$ bi-invariant neighborhood $\mathcal U$ of $M$and a
	mapping $\rho \colon \mathcal U  \to G $ such that
\begin{equation}
	\rho^*(\omega_L^i) =  \homega^i,
	\quad
	\rho^*(\omega_R^i) =  \comega^i,
\quad
         \rho(\hmu(x, g)) = \rho(x)\cdot g,
	\quad \rho(\cmu(g,x)) = g\cdot \rho(x).
\EqTag{SuperForm4}
\end{equation}
	If $M/G$ is simply connected, then one can take $\mathcal U = M$.	
\end{Theorem}

\begin{Remark}
\StTag{SuperForm3}	
	This theorem can be viewed as a simple extension (or refinement) of  three different, well-known results.
	First, suppose that {\bf [i]} holds. Then, because the actions  $\hmu$ and $\cmu$ are free and regular, we may
	construct $\hmu$ and $\cmu$ invariant sets  $\hat {\mathcal U}$ and $\check {\mathcal U}$ and maps
\begin{equation}
	\hrho \colon \hat {\mathcal U}\to G
	\quad\text{and}\quad
	\crho  \colon \check {\mathcal U}  \to G
\end{equation}
	which are $G$ equivariant (see, for example,  \cite{fels-olver:1998a})
	with respect to the actions $\hmu$ and $\cmu$, respectively.  	 	
	Theorem \StRef{SuperForm1} states  that these maps can be chosen such that they are equal on a common bi-invariant domain,
	$\hrho^*(\omega^i_L) = \homega^i$ and $\crho^*(\omega^i_R) = \comega^i$.
	Secondly, assume {\bf [ii]}. Then  a fundamental 	theorem in
	Lie theory (see  Griffiths \cite{griffiths:1974a} or Sternburg \cite{sternburg:1984a}, page 220)  asserts that there always exists maps
	$\hvarphi : M \to G$ and $\cvarphi : M \to G$ such that
\begin{equation}
	\hvarphi^*(\omega_L^i) = \homega^i \quad\text{and}\quad
	\cvarphi^*(\omega_R^i) = \comega^i.
\end{equation}
	From this viewpoint, Theorem \StRef{SuperForm1} states that these  maps can be taken to be equal and equivariant with respect to the
	actions $\hmu$ and $\cmu$.  Thirdly, suppose that just one of the actions in {\bf[i]} is given, that $G$ acts regularly on $M$ and
	the corresponding forms (say $\homega^i$) in {\bf[ii]} are given. Then $\pi\:M \to M/G $ may be viewed as  a principal $G$ bundle
	 with  (by {\bf[ii]}) flat connection 1-forms  $\homega^i$.
	Then Kobayashi and Nomizu \cite{kobayashi-nomizu:1963a}(Theorem 9.1, Volume I, page 92) state that there is an open set $U$ in $M/G$ and a diffeomorphism
	$\psi:\pi^{-1}(U) \to U \times G$  such that $\psi^*(\omega^i_L) = \homega^i$ (where the Maurer-Cartan forms $\omega^i_L$ define
	the canonical flat connection on $U \times G$). Moreover when  $M/G$ is simply connected,
	then $M$ is globally the trivial $G$ bundle with canonical flat connection.
	Compare this result with Corollary \StRef{SuperForm4} to Theorem \StRef{SuperForm1}.
\end{Remark}
	We establish Theorem  \StRef{SuperForm1} with the following sequence of lemmas.

\begin{Lemma}
\StTag{SuperFormLem1}

	Under the hypothesis of Theorem \StRef{SuperForm1} there is a bi-invariant, connected neigborhood $\mathcal U$ around $x_0$ and a mapping
	$\rho \:\mathcal U \to G$ such that  $\rho^*(\omega_L^i) =  \homega^i$ and $\rho^*(\omega_R^i) =  \comega^i$.
	The map $\rho$ is uniquely defined for a given domain $\mathcal U$.
\end{Lemma}
\begin{proof}
	As is customary, we suppose  that $\omega_L^i(e) = \omega_R^i(e)$, where $e$ is the identity of $G$. In accordance with the aforementioned
	Theorem 9.1 in \cite{kobayashi-nomizu:1963a} there is a $\hmu$ invariant open set $\mathcal U$ in $M$ and a unique smooth map
	$\rho\: \mathcal U \to G$ such that
\begin{equation}
	\rho^*(\omega_L^i) =  \homega^i
	\quad\text{and}\quad
	\rho(x_0) =e.
\EqTag{SuperForm6}
\end{equation}
	Since the $\hmu$ and $\cmu$ orbits coincide, the set $\mathcal U$ is actually bi-invariant. Because $G$ is connected, we may assume that
	$\mathcal U$ is connected. The uniqueness of $\rho$ is established  in Theorem 2.3 (page 220) of  \cite{sternburg:1984a}.
	
	In order to prove that $\rho$ satisfies $\rho^*(\omega_R^i) =  \comega^i$, let  $\Lambda^i_j(g)$
	be the matrix for the linear transformation
\begin{equation*}
	\text{Ad}^*(g) = R^*_g \circ L^*_{g^{-1}} \: T^*_eG \to T^*_eG
\end{equation*}
	with respect to the basis $\homega^i(e)$. Then an easy computation shows that
\begin{equation}
	\omega_L^i(g) = \Lambda^i_j(g)\, \omega^j_R(g) \quad\text{and}\quad \Lambda^i_k(g_1g_2) = \Lambda^i_j(g_2) \Lambda^j_k(g_1).
\EqTag{SuperForm7}
\end{equation}
	The Jacobian of the multiplication map $m \: G\times G \to G$ may be computed in terms of $\Lambda^i_j$ as
\begin{equation}
	m^*(\omega^i_L)(g_1 g_2) =  \Lambda^i_j(g_2)\,\omega^j_L(g_1) +  \omega^i_L(g_2).
\EqTag{SuperForm8}
\end{equation}
	We take the Lie derivative of the first equation in \EqRef{SuperForm7} with respect to  $X^L_k$ to find that
	$L_{X^L_k} \Lambda^i_j= C^i_{k \ell} \Lambda^\ell_j$ and hence
\begin{equation}
	\extd \Lambda^i_j = C^i_{k\ell}\, \Lambda^\ell_j\,\omega^k_L.
\EqTag{SuperForm10}
\end{equation}
	Similarly,  by  {\bf [ii][b]}, we may write
	$\homega^i = \lambda^i_j \comega^j$ and calculate from \EqRef{SuperForm3}
\begin{equation}
	\extd \lambda^i_j = C^i_{k\ell}\, \lambda^\ell_j\,\homega^k.	
\EqTag{SuperForm11}
\end{equation}
	The combination of equations  \EqRef{SuperForm6}, \EqRef{SuperForm10} and \EqRef{SuperForm11} (and the fact that
	$\Lambda^i_j(e) = \lambda^i_j(x_0)$) now shows that
\begin{equation}
	\extd \big(\rho^*(\Lambda^i_j) (\lambda^{-1})^j_k  \big) = 0
	\quad\text{and hence}\quad
	\rho^*(\Lambda^i_j) = \lambda^i_j.
\EqTag{SuperForm12}
\end{equation}
	(Here we use the connectivity of $\mathcal U$).
	The equations $\rho^*(\omega_R^i) =  \comega^i$ now follows immediately from \EqRef{SuperForm6},
	\EqRef{SuperForm7}, and  \EqRef{SuperForm12}.
\end{proof}

	We note, for future use, that \EqRef{SuperForm11} also implies
\begin{equation}
	\CalL_{\hX_i} \comega^j= \CalL_{\hX_i} (\lambda^{-1})^j_k \  \homega^k +  (\lambda^{-1})^j_k\, C^k_{i\ell}\, \homega^\ell = 0.
\EqTag{SuperForm54}
\end{equation}
	This shows, by virtue of the connectivity of $G$, that
\begin{equation}
	(\hmu_g)^*(\comega^j) = \comega^j.
\EqTag{SuperForm44}
\end{equation}
\par

\begin{Lemma}
\StTag{SuperFormLem2}
	With $\rho: \mathcal U \to G$ as in Lemma \StRef{SuperFormLem1},
	the maps $\hvarphi = \rho  \circ \hmu_{x_0} \: G \to G$ and $\cvarphi = \rho  \circ \cmu_{x_0} \: G \to G$
	satisfy
\begin{equation}
	\hvarphi^*(\omega^i_L) = \omega^i_L
\quad\text{and}\quad
	\cvarphi^*(\omega^i_R) = \omega^i_R
\EqTag{SuperFormA1}
\end{equation}
	and hence  $\hvarphi = \cvarphi = \text{id}_G$.
\end{Lemma}
\begin{proof}
	The definitions $(\hmu_x)_*(X^L_i) = \hX_i$ and $(\cmu_x)_*(X^R_i) = \cX_i$, together with {\bf [ii](a)},
	imply that $\hmu^*_{x_0}(\homega^i) = \omega^i_L$ and  $\cmu^*_{x_0}(\comega^i) = \omega^i_R$. Equations
	\EqRef{SuperFormA1} then follow directly from Lemma \StRef{SuperFormLem1}. These equations imply that
	 $\hvarphi$  and  $\cvarphi$ are translations in $G$ and therefore, since $G$ is connected and $\hvarphi(e) = \cvarphi(e) = e$,
	we have that  $\hvarphi = \cvarphi = \text{id}_G$.
\end{proof}

\begin{Lemma}
	The maps $\rho\: \mathcal U \to G$ satisfies $\rho(\hmu(x, g)) = \rho(x)\cdot g$
	and  $\rho(\cmu(x, g)) = g \cdot \rho(x)$ for all $x \in \mathcal U$ and $g \in G$.
\end{Lemma}
\begin{proof}
	To prove these   equivariance properties  of $\rho$, define maps
\begin{equation*}
	\hrho = R_{g^{-1}} \circ \rho \circ\hmu_g \: M \to G
	\quad\text{and}\quad
	\crho = L_{g^{-1}} \circ \rho \circ\cmu_g \: M \to G,
\end{equation*}
	where $\hmu_g(x) = \hmu(g, x)$ and  $\cmu_g(x) = \cmu(g, x)$.  We wish to prove that $\hrho = \check \crho =\rho$.
	
	Lemma \StRef{SuperFormLem2}
	shows that $\hat \rho(x_0) =  \check \rho(x_0) = e$ so that  it suffices to show that
\begin{equation}
	\hrho^*(\omega^i_R) = \comega^i \quad\text{and}\quad \crho^*(\omega^i_L) = \comega^i
\EqTag{SuperForm14}
\end{equation}
	and then to appeal to  the uniqueness statement in  Lemma \StRef{SuperFormLem1}.
	The first equation \EqRef{SuperForm14} follows from the fact that the forms $\omega_R^i$ are
	right invariant, the second equation in \EqRef{SuperForm4} and \EqRef{SuperForm44}.
	The proof of the second equation
	in \EqRef{SuperForm14} is similar.
\end{proof}

\begin{Corollary}[Corollary to Theorem \StRef{SuperForm1}]
\StTag{SuperForm4}
	Let $\CalS$ be the submanifold of $M$ defined by $\CalS = \rho^{-1}(e)$, where $\rho$
	is the map \EqRef{SuperForm4}. Then the map
\begin{equation*}
	\Psi\colon \CalU \to \CalS \times G
	\quad\text{defined by}\quad
	\Psi(x) = (\hmu (\rho(x)^{-1}, x), \rho(x))
\end{equation*}
	is  a  $G$ bi-equivariant diffeomorphism  to an open subset of
	$\mathcal S \times G$ and satisfies
\begin{equation}
	\Psi_*(\hX_i) = X^R_i, \quad
	\Psi_*(\cX_i) = X^L_i, \quad
	\Psi^*(\omega^i_L) = \homega^i, \quad
	\Psi^*(\omega^i_R) = \comega^i.
\EqTag{SuperForm15}
\end{equation}
	If $M/G$ is simply-connected, then $\Psi\colon M \to \CalS \times G$ is a
	global diffeomorphism
\end{Corollary}
	
\begin{proof}
	It is easy to check that $\Psi$ sends the intersection of $\mathcal U$ with
	the domain of $\hmu$ in $M$  into $\mathcal S \times G$.
	To check the equivariance of $\Psi$, it is convenient to write $\hmu(g,x) = x \hstar g$ and $\cmu(g,x) = g \cstar x$.
	Because the orbits for the two actions coincide, we know that for each $g \in G$ there is  a $g'$
	such that  $g \cstar x  =  x \hstar g'$.  We then use \EqRef{SuperForm4} to calculate
\begin{align*}
	\Psi(x\hstar g)
& 	= \big((x\hstar g)\hstar \rho(x\hstar g)^{(-1)}, \, \rho(x\hstar g)\big)
	= \big(x \hstar \rho(x)^{(-1)},\, \rho(x) \hstar g\big)
\\[2\jot]
&	= \Psi(x)*g , \quad\text{and}
\\[2\jot]
	\Psi(g\cstar x)
& 	= \big((g\cstar x)\hstar \rho(g\cstar x)^{(-1)},\, \rho(g \cstar x)\big)  =
	 \big(x \hstar g')\hstar \rho(x \hstar g')^{(-1)},\, g\cstar\rho(x)\big)
\\[2\jot]
&	=  \big((x \hstar \rho(x)^{(-1)},\, g\cstar\rho(x)\big) =  g * \Psi(x),
\end{align*}
as required.
\end{proof}

\newpage
\subsection{The Superposition Formula for Darboux Pairs}
	
	We are now ready to establish, using the 5-adapted coframe of  Section 4.4 and the technical results of  Section 5.1,
	the superposition formula for any Darboux pair.

	By way of summary, let us  recall that our starting point is the local (0-adapted) coframe
	$\{\,\bftheta,\ \bfheta,\, \bfhsigma,\, \bfceta,\, \bfcsigma\,\}$ which is  adapted to a given Darboux pair
	$\{\, \hV, \cV\,\}$
	in the sense that
\begin{equation}
	I = \text{span} \{ \bftheta,\, \bfheta,\,  \bfceta\},
	\quad
	\hV = \text{span} \{\bftheta,\, \bfhsigma,\, \bfheta,\, \bfceta \},
	\quad
	\cV = \text{span} \{\bftheta,\, \bfheta,\, \bfcsigma,\, \bfceta \}.
\EqTag{SuperForm16}
\end{equation}
	In our first coframe adaption we adjusted the forms $\bfheta$ and $\bfceta$ to the form
\begin{equation}
	\bfheta = \extd \bfhI_2 + \bfhR\, \bfhsigma
	\quad\text{and}\quad
	\bfceta = \extd \bfcI_2 + \check \bfR\, \bfcsigma.
\end{equation}
	The fourth coframe adaptations
	$\{\, \bfthetaX,\, \bfhsigma,\, \bfheta,\, \bfcsigma,\, \bfceta \,\}$ and	
	$\{\,\bfthetaY,\, \bfhsigma,\, \bfheta,\, \bfcsigma,\, \bfceta \, \}$ allowed us to associate
	to the Darboux pair $\{\, \hV, \cV \,\}$
	an abstract Lie algebra $\vess(\hV, \cV)$. We have  $\thetaX^i = Q^i_j \thetaY^j$ and we can write
\begin{equation}
	\hV \cap \cV = \text{span} \{ \bfthetaX,\, \bfheta,\,  \bfceta\},
\EqTag{SuperForm42}
\end{equation}
	From the 4-adapted coframes we then constructed the final coframes
\begin{equation}
	\htheta^i = \hR^i_j \thetaX^j  + \phi^i_a\hpi^a,
	\quad
	\ctheta^i = \cR^i_j \thetaY^j  + \psi^i_\alpha \cpi^\alpha,
\EqTag{SuperForm17}
\end{equation}
	where $\bfhpi = [ \bfhsigma,\, \bfheta],$ and $\bfcpi = [\bfcsigma,\, \bfceta]$,
	with structure equations
\begin{equation}
	\extd \htheta^i =
	\frac12   G^i_{\alpha\beta}\, \cpi^\alpha \wedge \cpi^\beta
        +\frac12 C^i_{jk}\htheta^j\wedge \htheta^k ,
\quad
	\extd \ctheta^i =
	\frac12 H ^i_{ab}\, \hpi^a \wedge \hpi^b
        -\frac12 C^i_{jk}\ctheta^j\wedge \ctheta^k.
\EqTag{SuperForm18}
\end{equation}
	  The coefficients $G^i_{\alpha\beta}$ are functions
	of the first integrals $\cI^\alpha$ while the  $H ^i_{ab}$ are functions of the $\hI^a$.
	The two 5-adapted coframes
\begin{equation}
	\{\, \bfhtheta,\ \bfhsigma,\ \bfheta,\ \bfcsigma,\ \bfceta \, \}
	\quad\text{and}\quad
	\{\,\bfctheta,\  \bfhsigma,\, \bfheta,\ \bfcsigma,\ \bfceta \, \}
\EqTag{SuperForm40}
\end{equation}
	are {\it not} adapted to the Darboux pair $\{\,\hV, \cV\,\}$  (in the sense of \EqRef{First4})
	although we do have
\begin{equation}
       \hV  = \text{span}\{\, \bfhtheta,\  \bfhsigma,\ \bfheta,\ \bfceta \, \}
	\quad\text{and}\quad
	\cV  = \text{span}\{\, \bfctheta,\  \bfheta,\ \bfcsigma,\ \bfceta \,\}.
\EqTag{SuperForm41}
\end{equation}

\begin{Definition}	
\StTag{SuperForm33}
	Let $\{\, \hV, \cV,\}$ be a Darboux pair with 5-adapted coframes \EqRef{SuperForm40}. Let
	$\{\, \hX_i, \hV_a, \cV_\alpha\,\}$ and $\{\,\cX_i,\cW_a,\cW_\alpha\,\}$
	be the dual frames corresponding to the 5 adapted coframes so that, in particular,
\begin{equation}
	\htheta^i(\hX_j) = \delta^i_j,\  \hpi^a(\hX_j) = 0,\
	\ctheta^i(\cX_j) = \delta^i_j,\ \cpi^a(\cX_j) = 0.\ \
\EqTag{SuperForm19}
\end{equation}
	The {\deffont infinitesimal Vessiot group actions} for the
	Darboux pair  $\{\,\hV, \cV\,\}$ are defined by the Lie algebras of vector fields $\hX_i$ and $\cX_i$.
	These define (local) groups actions
\begin{equation}
	\hmu\: G \times M \to M \quad\text{and}\quad  \cmu\: G\times M \to M.
\end{equation}
	We  take $\hmu$ to be a {\it right} action on $M$ and $\cmu$ to be a {\it left} action.
\end{Definition}

The vectors fields $\hX_i$ and $\cX_i$ are related by
	(see \EqRef{Fifth60} and \EqRef{Fifth61})
\begin{equation}
	\hX_i = \lambda^j_i \cX_j\quad\text{and satisfy}\quad [\,\hX_i, \, \cX_j\,] =0.
\EqTag{SuperForm120}
\end{equation}	
	From the structure equations \EqRef{SuperForm18} we find that
\begin{gather}
	\quad\CalL_{\hX_j} \htheta^i = C^i_{jk} \htheta^k \quad \text{and} \quad
	\CalL_{\cX_j }\ctheta^i = C^i_{jk} \ctheta^k.
\EqTag{SuperForm101}
\\[2\jot]
\quad\CalL_{\hX_j}\hpi^a   = \CalL_{\hX_j}\cpi^\alpha   = \CalL_{\cX_j}\hpi^a   = \quad\CalL_{\cX_j} \cpi^\alpha = 0.
\EqTag{SuperForm102}
\end{gather}
	The actions $\hmu$  and $\cmu$   commute because $[\, \hX_i, \cX_j] = 0$.  In what follows
	we  shall assume that these actions are regular.
\par
	
	Now let $G$ be a Lie group with Lie algebra $\vess(\hV, \cV)$.  Let $\omega^i_L$ and $\omega^i_R$ be
	the left and right invariant Maurer-Cartan forms on $G$, with structure equations \EqRef{SuperForm2}.
	We  shall assume that these two coframes on $G$ coincide at the  identity $e$.
	Let $X^L_i$ and $X^R_i$ be the dual basis of  left and right invariant vector fields.

	To apply Theorem \StRef{SuperForm1} we need to identify the forms $\homega^i$ and $\comega^i$.
\begin{Lemma}
\StTag{SuperForm5}
	 The  forms  $\homega^i$ and $\comega^i$  defined in terms of the 5-adapted coframes (\EqRef{Fifth1} or \EqRef{SuperForm17}) by	
\begin{equation}
	\homega^i = \htheta^i + \lambda^i_j \psi^j_\alpha \cpi^\alpha
	\quad\text{and}\quad
	\comega^i = \ctheta^i + \mu^i_j \phi^j_a \hpi^a,
\EqTag{SuperForm20}
\end{equation}
	where $\lambda = \hR Q \cR ^{-1}$ and $\mu = \lambda^{-1}$ (see Remark \StRef{Fifth24}),
	have the same span and satisfy the structure equations
\begin{equation}
	\extd \homega^i = \frac12 C^i_{jk}\, \homega^j\wedge \homega^k
	\quad\text{and}\quad
	\extd \comega^i= -\frac12 C^i_{jk}\, \comega^j \wedge \comega^k.
\EqTag{SuperForm21}
\end{equation}
\end{Lemma}
\begin{proof}
	To  check that the  forms $\homega^i$ and $\comega^i$ have the same span we
	simply use the various  definitions  given here to calculate
\begin{align}
	\homega^i
	&= \htheta^i  + \lambda^i_j\,\psi^j_\alpha\,\cpi^\alpha = \hR^i_j\,\thetaX^j +\phi^i_a \hpi^a + \lambda^i_j\psi^j_\alpha \cpi^\alpha
\notag
\\[2\jot]
	&= \hR^i_j\, Q^j_k \, \thetaY^k  + \phi^i_a \hpi^a + \lambda^i_j \psi^j_\alpha \cpi^\alpha =
	(\hR \, Q \, \cR^{-1})^i_j \, \cR^j_k \, \thetaY^k +\phi^i_a\, \hpi^a + \lambda^i_j\, \psi^j_\alpha \,\cpi^\alpha
\notag
\\[2\jot]
	&= \lambda^i_j\,  ( \cR^j_k\, \thetaY^k + \psi^j_a \cpi^a + (\lambda^{-1})^j_k\, \phi^k_\alpha\, \hpi^\alpha) =
	\lambda^i_j\,( \ctheta^j + \mu^j_k \,\phi^k_\alpha\,\hpi^\alpha)
        = \lambda^i_j \, \comega^j.
\EqTag{SuperForm22}
\end{align}
	We now calculate $\extd \homega^i$ using \EqRef{Fifth2},  \EqRef{Fifth59},   \EqRef{Fifth53} and \EqRef{Fifth61}
	to find that
\begin{align*}
       \extd \homega^i
&      = \extd \htheta^i + \extd \lambda^i_j\wedge \psi^j + \lambda^i_j \wedge \extd \psi^j
\\[2\jot]
&      = \frac12   G^i_{\alpha\beta}\, \cpi^\alpha \wedge \cpi^\beta
        +\frac12 C^i_{jk}\,\htheta^j\wedge \htheta^k  + (C^i_{\ell m}\,\lambda^m_j\, \htheta^\ell  +\lambda^i_h\,C^h_{mj}\,
	\psi^m) \wedge \psi^j
\\[2\jot]
&	\qquad + \lambda^i_j \,(-\frac12C^j_{hk} \psi^h\wedge \psi^k -\frac12 \cR^j_\ell\,
	F^\ell_{\alpha\beta}\,\cpi^\alpha\wedge\cpi^\beta)
\\[2\jot]
&       =  \frac12 C^i_{jk}\,\htheta^j\wedge \htheta^k + C^i_{\ell m}\,\lambda^m_j\, \htheta^\ell \wedge\psi^j +
	 \frac12   \lambda^i_\ell\, C^\ell_{hk}\,\psi^h\wedge\psi^k = \frac12C^i_{jk}\homega^j \wedge \homega^k.
\end{align*}
	The derivation of the structure equations for $\comega^i$ is similar.
\end{proof}

	Let $x_0 \in M$. By a change of frame at $x_0$ we may suppose that $\homega^i(x_0) = \comega^i(x_0)$.
	At this point all the hypothesis for Theorem \StRef{SuperForm1} are satisfied and, in accordance with this  	
	theorem and Corollary \StRef{SuperForm4}, we can construct maps
\begin{equation}
	\rho \colon  \mathcal U \to G  \quad\text{and}\quad  \Psi\colon \mathcal U \to \CalS \times G
\end{equation}
	which satisfy \EqRef{SuperForm4}, \EqRef{SuperForm15} and $\rho(x_0) = e$.

	Define  $\CalS_1$ and $ \CalS_2$ to be the integral manifolds  through $x_0$ for the restriction
	of $\hV^{(\infty)}$ and $\cV^{(\infty)}$ to $\CalS$.  Because
\begin{equation}
	T^*\CalS_x = \hV^{(\infty)}_x \oplus \cV^{(\infty)}_x \quad\text{for all $x \in \CalS$},
\end{equation}
	we may choose $\CalS_1$ and $ \CalS_2$ small enough so as to be assured of the existence of a local diffeomorphism
\begin{equation}
	\chi \colon  \CalS_1 \times \CalS_2 \to \CalS_0,
\EqTag{SuperForm23}
\end{equation}
	where $S_0$ is an open set in $\CalS$. The map $\chi$ satisfies
	(with $s= \chi(s_1, s_2)$ and $s_i  \in \CalS_i$)
\begin{equation}
	\chi^*(\hV^{(\infty)}_s) = T^*_{s_2} \CalS_2
	\quad\text{and}\quad
	\chi^*(\cV^{(\infty)}_s) = T^*_{s_1} \CalS_1 .
\EqTag{SuperForm24}
\end{equation}
	Set $\chi(x^0_1, x^0_2) = x_0$.
	Hence, for any first integrals $\hI$ and $\cI$ of $\hV$ and $\cV$, we have
\begin{equation}
	\hI(\chi(s_1, s_2) = \hI(\chi(x^0_1, s_2))
	\quad\text{and}\quad
	\cI(\chi(s_1, s_2)) = \cI(\chi(s_1, x^0_2)).
\EqTag{SuperForm25}
\end{equation}
	Put $\CalU_0 = \Phi^{-1}(\CalS_0)$.	
\par	
	To define the superposition formula for the Darboux pair $\{\hV, \cV\}$, we take  $M_1$ to be the (maximal, connected) integral manifold of $\hV^{(\infty)}$ through $x_0$ in
	$\CalU_0$ and  $M_2$ to be the (maximal, connected) integral manifold of $\cV^{(\infty)}$ through $x_0$ in $\CalU_0$.
	The group actions $\hmu$  and $\cmu$ and the maps $\rho$ and $\Psi$ all restrict to maps
\begin{equation}
\begin{gathered}
	\hmu_i \colon G\times M_i \to M_i,
	\quad \cmu_i \colon G\times M_i \to M_i,
\\
	 \rho_i \colon  M_i \to G,\quad\text{and}\quad
	\Psi_i \colon  M_i \to \CalS_i \times G
\end{gathered}
\end{equation}
	for $i = 1, 2$.
\par
	With the maps $\Psi$ and $\chi$ defined by
	\EqRef{SuperForm4} and \EqRef{SuperForm23} and with these choices for  $M_1$ and $M_2$, we define
\begin{equation}
	\Sigma \colon  M_1 \times M_2 \to M
	\quad\text{by}\quad
	\Sigma(x_1, x_2) = \hmu\big(g_1 \cdot g_2, \chi(s_1, s_2)\big),
\EqTag{SuperForm26}
\end{equation}
where $x_i \in M_i$, and $\Psi(x_i) = ( s_i, g_i)$.
\par
	
        To prove that $\Sigma$ defines a superposition formula, we must calculate the  pullback by  $\Sigma$ of an adapted coframe on $M$.
	To this end, we introduce the following notation.
	For any differential form $\alpha$  on $M$, denote the restriction of $\alpha$ to $M_i$
	by $\alpha_i$.  If $f$ is a function on $M$ and
	$x_i\in M_i$, then $f(x_i)$ denotes the value of $f$ at $x_i$, viewed as a point of $M$.
	We also remark that the equivariance of $\rho$ with respect to the
	group action $\hmu$ and $\cmu$ implies that $\Sigma $ satisfies
\begin{equation}
\begin{aligned}
	(\rho \circ \Sigma)(x_1, x_2)
	& = \rho( \hmu(g_1\cdot g_2, \chi(s_1, s_2)) = g_1 \cdot g_2 \cdot \rho(\chi(s_1, s_2))
\\[2\jot]
	& = g_1 \cdot g_2 = \rho_1(x_1) \cdot \rho_2(x_2) = (m\circ (\rho_1 \times \rho_2))(x_1, x_2).
\end{aligned}
\EqTag{SuperForm27}
\end{equation}
\par

\begin{Lemma} The
	pullback by $\Sigma$ of the 4-adapted coframe (see Theorm \StRef{Fourth5}) on $M$ is given in terms of the restrictions
	of the 5-adapted coframes \EqRef{SuperForm40} on $M_1$ and $M_2$ by
\StTag{SuperForm25}
\begin{gather}
	\Sigma^*(\bfhsigma) = \bfhsigma_2,
	\quad
	\Sigma^*(\bfcsigma) = \bfcsigma_1,
	\quad
	\Sigma^*(\bfheta) =  \bfheta_2,
	\quad
	\Sigma^*(\bfceta) =   \bfceta_1,
\EqTag{SuperForm30}
\\[2\jot]
	\Sigma^*(\bfR\bfthetaX) = \boldsymbol{\lambda} \,(\bfhtheta_1 + \bfctheta_2).	
\EqTag{SuperForm31}
\end{gather}
\end{Lemma}
\begin{proof}
We begin with the observation that the first integrals
	$\hI^a$ and $\cI^\alpha$ are invariants of  $\hmu$ and therefore, by \EqRef{SuperForm25},
\begin{align}
	\Sigma^*(\hI^a) (x_1, x_2) &=  \hI^a(\hmu( g_1 g_2, \chi(s_1, s_2))) = \hI^a( \chi(s_1, s_2))
\notag
\\
	&=  \hI^a( \chi(x_{01}, s_2)) =  \hI^a( \hmu(g_2,\chi(x_{01}, s_2)))
\notag
\\	
	&= \hI^a(x_2),\qquad\text{and similarly}
\EqTag{SuperForm38}
\\
\Sigma^*(\cI^\alpha)(x_1, x_2) &= \cI^\alpha(x_1).
\EqTag{SuperForm29}
\end{align}
	For future reference,  we note (see \EqRef{SuperForm17}) that
\begin{equation}
	\Sigma^*(G^i_{\alpha\beta})(x_1, x_2) = G^i_{\alpha\beta}(x_1) \quad\text{and}\quad
	\Sigma^*(H^i_{ab})(x_1, x_2) = G^i_{ab}(x_2).
\end{equation}
	From equations  \EqRef{SuperForm38} and  \EqRef{SuperForm29} it follows directly that
\begin{equation}
	\Sigma^*(\bfhsigma)= \bfhsigma_2,
	\qquad
	\Sigma^*(\bfcsigma)= \bfcsigma_1,
\end{equation}
\begin{equation}
	\Sigma^*(\bfheta) = \Sigma^*( d \bfhI + \bfhR \, \bfhsigma) = \bfheta_2
	\quad\text{and}\quad
	\Sigma^*(\bfceta) = \Sigma^*( d \bfcI + \bfcR \, \bfcsigma) = \bfceta_1 .
\end{equation}
	Equations  \EqRef{SuperForm30} are therefore established.
\par		
	Equations \EqRef{SuperForm17}   and \EqRef{SuperForm20} imply that
\begin{equation}
	\hR^i_j \theta^j_X = \homega^i   - \phi^i_a\, \hpi^a - \lambda^i_k\psi^k_a\,\cpi^\alpha.
\EqTag{SuperForm32}
\end{equation}
	We use this equation to  calculate $\Sigma^*(\hR^i_j \thetaX^j)$.

	Since $\phi^i_a \in  \Inv(\hV)$ and $\psi^i_\alpha \in  \Inv(\cV)$,  equations
	\EqRef{SuperForm38} and  \EqRef{SuperForm29} imply that
\begin{equation}
	\Sigma^*(\phi^i_a \, \hpi^a)(x_1, x_2) =  \phi^i_a(x_2)\, \hpi^a_2
	\quad\text{and}\quad
	\Sigma^*(\psi^i_a\, \cpi^\alpha)(x_1, x_2) =  \psi^i_\alpha(x_1) \,\hpi^\alpha_1.
\EqTag{SuperForm33}
\end{equation}
	We deduce from equations  \EqRef{SuperForm7}, \EqRef{SuperForm12} and \EqRef{SuperForm27} that
\begin{equation}
\begin{aligned}
	\Sigma^*(\lambda^i_k)(x_1, x_2)
	& = \Sigma^* (\rho^* \Lambda^i_k)(x_1, x_2) = \Lambda^i_k(g_1 \cdot  g_2) =
\\[2\jot]
	& = \Lambda^i_j(g_2) \Lambda^j_k(g_1) = \lambda^i_j(x_2) \lambda^j_k(x_1).
\end{aligned}
\EqTag{SuperForm34}
\end{equation}
\par			
	To calculate $\Sigma^*(\homega^i)$ it is helpful to introduce the
	projection maps $\pi_i\colon G\times G \to G$ and to note that
        $\pi_i(\rho_1\times\rho_2)(x_1, x_2) = g_i$,  $i = 1, 2$. We can then
	re-write \EqRef{SuperForm8} as
\begin{equation*}
	m^*(\omega^i_L) = \pi_2^*(\lambda^i_j) \pi_1^*(\omega^j_L) + \pi^*_2(\omega^i_L).
\end{equation*}
 	This equation, together with \EqRef{SuperForm4} and \EqRef{SuperForm27} leads to
\begin{equation}
\begin{aligned}
	\Sigma^*(\homega^i)(x_1, x_2)
&	=  \Sigma^*(\rho^*(\omega^i_L))(x_1, x_2)  = (\rho_1^* \times \rho_2^*)(m^*(\omega^i_L))(x_1, x_2)
\\[2\jot]
&	= (\rho_1^* \times \rho_2^*)(\pi_2^*(\lambda^i_j) \pi_1^*(\omega^j_L) + \pi^*_2(\omega^i_L))(x_1, x_2) = 	
\\	
	& = \lambda^i_j(x_2)\, \homega^j_1   +\homega^i_2.	
\end{aligned}
\EqTag{SuperForm35}
\end{equation}
	Finally, the combination of equations \EqRef{SuperForm22} and \EqRef{SuperForm32} -- \EqRef{SuperForm35} allows us to calculate
\begin{align}
	\Sigma^*(\hR^i_j\,\theta^j_X)
&	= \Sigma^*\bigl(\homega^i  - \phi^i_a\,\hpi^a - \lambda^i_k\,\psi^k_\alpha\, \cpi^\alpha\bigr)
\notag
\\[2\jot]
&	=\lambda^i_j(x_2)\,\homega^j_1 +\homega^i_2 - \phi^i_a(x_2)\,\pi^a_2 - \lambda^i_j (x_2)\lambda^j_k(x_1)\psi^k_\alpha(x_1)\,\cpi^\alpha_1
\notag
\\[2\jot]
&	= \lambda^i_j(x_2)\,\bigl(\homega^j_1 - \lambda^j_k\,(x_1)\psi^k_\alpha(x_1) \,\cpi^\alpha_1 \,\bigr) +
	\lambda^i_j(x_2)\,\bigl (\comega^j_2 -\mu^j_k(x_2) \,\phi^k_a(x_2)\, \hpi^a_2\,\bigr)
\notag
\\[2\jot]
&         =\lambda^i_j(x_2)\,\htheta^j_1 + \lambda^i_j(x_2)\,\ctheta^j_2,
\EqTag{SuperForm36}
\end{align}
	which proves \EqRef{SuperForm31}.
\end{proof}

	Recall that for any Darboux pair $\{\, \hV, \cV \, \}$,  the EDS $\hCalV \mycap \cCalV$ is given by  Definition \StRef{Dec5}.
\begin{Theorem}[\sc The Superposition Formula]
\StTag{SuperForm6}
	Let  $\{\, \hV, \cV \, \}$ define a Darboux pair  on M. Define $M_1$
	and $M_2$ as above and let $W_1$ and  $W_2$ be the
	restrictions of $\cV$ and $\hV$ to $M_1$ and $M_2$, respectively.
	Then the map $\Sigma \colon M_1\times  M_2 \to M$, defined by \EqRef{SuperForm26}, satisfies
\begin{equation}
	\Sigma^*(\hV \cap \cV) \subset  W_1 +   W_2
	\quad
	\text{and}
	\quad
	\Sigma^*(\hCalV \mycap \cCalV) \subset  \CalW_1 +   \CalW_2.
\EqTag{SuperForm66}
\end{equation}
	Thus $\Sigma$ defines a superposition formula for
	$\hV \mycap \cV$ with respect to the Pfaffian systems $\CalW_1$  and $\CalW_2$.
\end{Theorem}
\begin{proof}
	To prove this theorem  we need only explicitly list the generators for
	$W_1 + W_2$, $\CalW_1 +\CalW_2$, $\hV \cap \cV$ and $\hCalV \mycap \cCalV$ and check that, using Lemma
	\EqRef{SuperForm15}, that the latter pullback into the former by $\Sigma$.
	For this we use the definitions of $\hV$ and $\cV$
	in terms of the 4-adapted and 5-adapted coframes by\EqRef{SuperForm42} and \EqRef{SuperForm41}.

	From the definition of $M_1$ and $M_2$ as integral manifolds  of $\hV^{(\infty)}$ and $\cV^{(\infty)}$, we have that
\begin{equation*}
	\bfheta_1 =0 ,\ \bfhsigma_1 = 0\text{ on $M_1$}
	\quad\text{and}\quad \bfceta_2 = 0,\ \bfcsigma_2 = 0
	\text{ on $M_2$}
\end{equation*}
	and hence the two 5-adapted coframes on $M$ restrict to coframes
\begin{equation}
	\text{ $\{\,\bfhtheta_1,\, \bfceta_1,\, \bfcsigma_1\, \}$ for  $M_1$ }
	\quad\text{and} \quad
	\text{ $\{\, \bfctheta_2,\,  \bfheta_2,\, \bfhsigma_2 \,\}$ for  $M_2$.}
\EqTag{SuperForm43}
\end{equation}
	These naturally combine to give a coframe on $M_1\times M_2$.
	The 1-form generators for  the Pfaffian systems $W_1$ and $W_2$ are then given by
\begin{equation}
	W_1 = \text{span}\{\,\bfhtheta_1,\ \bfceta_1 \,\}
	\quad\text{and}\quad
	W_2 = \text{span}\{\,\bfctheta_2,\ \bfheta_2 \,\} .
\EqTag{SuperForm51}
\end{equation}

	The differential system $\CalW_1 + \CalW_2$ is  algebraically generated
	by the 1-forms \EqRef{SuperForm51} and their
	exterior derivatives (see \EqRef{First10}, \EqRef{SuperForm18}), that is,
\begin{equation*}
	\CalW_1 + \CalW_2 = \{\, \bfhtheta_1,\  \bfceta_1,\ \bfctheta_2,\
	\bfheta_2,\ \extd \bfceta_1,\
	\extd \bfheta_2,\ \extd \bfhtheta_1,\  \extd \bfctheta_2 \, \}.
\end{equation*}
	The restrictions of \EqRef{SuperForm18} to $M_1$ and $M_2$ lead to the structure equations
\begin{align*}
	\extd \bfhtheta_1
	& = \frac12\bfG_1\bfcpi_1 \wedge\ \bfcpi_1 \mod \{\bfhtheta_1\}
	=  \frac12\bfG_1\bfcsigma_1 \wedge\ \bfcsigma_1 \mod \{\bfhtheta_1, \bfceta_1\}	\quad\text{and}
\\[2\jot]
	\extd \bfctheta_2
	&= \frac12\bfH_2\bfhpi_2 \wedge\ \bfhpi_2 \mod \{\bfctheta_2\}
	=  \frac12\bfH_2\bfhsigma_2 \wedge\ \bfhsigma_2 \mod \{\bfhtheta_2, \bfheta_2\}
\end{align*}
	from which it follows that $\CalW_1 + \CalW_2$  is algebraically generated by
\begin{equation}
	\CalW_1 + \CalW_2 = \{\, \bfhtheta_1,\  \bfceta_1,\ \bfctheta_2,\ \bfheta_2,\
	\bfcF_1 \bfcsigma_1 \wedge \bfcsigma_1,\  \bfhF_2 \bfhsigma_2 \wedge \bfhsigma_2,
	\bfH_2\bfhsigma_2 \wedge\ \bfhsigma_2  \}.
\EqTag{SuperForm52}
\end{equation}

	The Pfaffian system $\hV \cap \cV$ is generated by the forms $\{ \bftheta,\, \bfheta,\,  \bfceta\}$
	or equivalently  (see \EqRef{SuperForm42}) by $\{ \bfthetaX,\, \bfheta,\,  \bfceta\}$. The
	first inclusion in \EqRef{SuperForm66}  follows immediately from the Lemma \StRef{SuperForm25} and \EqRef{SuperForm51}.
\par
	By definition, the differential system $\hCalV \mycap \cCalV$ is algebraically generated,
	in terms of the 1-adapted coframe (see Theorem \StRef{First9}),
	 as
\begin{equation*}
	\hCalV \mycap \cCalV  =
	\{\, \bftheta,\ \bfheta,\ \bfceta,\  \extd \bfheta,\ \extd \bfceta,\
	\bfA\, \bfhsigma \wedge \bfhsigma,\ \bfB\, \bfcsigma    \wedge \bfcsigma \,\}
\end{equation*}
	or, equally as well, in terms of the 4-adapted coframe by
\begin{equation*}	
	\hCalV \mycap \cCalV  =
	\{\, \bftheta,\ \bfheta,\ \bfceta,\ \bfcF \bfcsigma \wedge \bfcsigma,\  \bfhF \bfhsigma \wedge \bfhsigma,
	\bfA\, \bfhsigma \wedge \bfhsigma,\ \bfB\, \bfcsigma    \wedge \bfcsigma \,\}.
\end{equation*}
	In this latter equation the coefficients $\bfA$ and $\bfB$ are those appearing in \EqRef{Fourth6}.
        By  \EqRef{Fourth20}  and  \EqRef{Fifth59}, we also have that
\begin{equation}
	\hCalV \mycap \cCalV = \{\, \bfR\bfthetaX,\ \bfheta,\ \bfceta,\
	\bfcF \bfcsigma \wedge \bfcsigma,\  \bfhF \bfhsigma \wedge \bfhsigma,
	\bfH\, \bfhsigma \wedge \bfhsigma,\ \bfG\, \bfcsigma    \wedge \bfcsigma \,\}.
\end{equation}
	With the algebraic generators for $\hCalV \mycap \cCalV$ in this form, it is easy to calculate the pullback of
	$\hCalV \mycap \cCalV$ to $M_1 \times M_2$ by $\Sigma$. By using
	\EqRef{SuperForm30}, \EqRef{SuperForm31}, \EqRef{SuperForm36}   we find that
\begin{equation}
	\Sigma^*(\hCalV \mycap \cCalV) = \{\,\bfhtheta_1 + \bfctheta_2,\ \bfheta_2,\ \bfceta_1,\
	\bfcF_1 \bfcsigma_1 \wedge \bfcsigma_1,\  \bfhF_2 \bfhsigma_2 \wedge \bfhsigma_2, \
	\bfH_2\, \bfhsigma_2 \wedge \bfhsigma_2,\ \bfG_1\, \bfcsigma_1    \wedge \bfcsigma_1 \,\}
\EqTag{SuperForm55}
\end{equation}
	and the second inclusion in \EqRef{SuperForm66} is firmly established.	
\end{proof}
	One may use the first integrals of $\hV$ and $\cV$ and
	the map $\rho\: \CalU \to G$, and local coordinates $z^i$ on $G$  to define  local coordinates
\begin{equation}
	(\hI^a = \hI^a(x),\ \cI^\alpha =\cI^\alpha(x),\ z^i= \rho^i(x))
\end{equation} on $M$ and induced coodinates $(\cI^\alpha_1, z^i_1)$ on
	$M_1$ and $(\hI^a_2, z^i_2)$ on $M_2$. In these coordinates the superposition formula becomes
\begin{equation}
\Sigma\big(( \cI^\alpha_1, z^j_1),  (\hI^a_2, z^k_2 ) \big) = \big ( \hI^a = \hI^a_2, \cI^\alpha = \cI^\alpha_1, z^i = (z_1^j)\cdot (z_2^k) \big) .
\EqTag{SuperForm99}
\end{equation}
	We shall use this formula extensively in the examples in Section 6.
	
	To prove that the superposition map $\Sigma$ is subjective, at the level of integral manifolds,
	we  use the concept of an integrable extension of a differential system $\CalI$ on $M$
	\cite{bryant-griffiths:1995b}.
	This is a differential system  $\CalJ$ on a manifold $N$ together with a submersion $\varphi\: N \to M$ such that
	$\varphi^*(\CalI) \subset \CalJ$ and such that the quotient differential system
	$\CalJ/\varphi^*\CalI$ is a completely integrable Pfaffian system.
	This means that there are 1-forms $\vartheta^\ell\in \CalJ$ such that $\CalJ$ is algebraically generated by the
	forms $\{\,\vartheta^\ell\,\} \cup \varphi^*\CalI$ and
\begin{equation}
	\extd \vartheta^\ell \equiv 0 \quad \mod \{\,\vartheta^\ell\,\} \cup \varphi^*\CalI
\end{equation}
	or, equivalently, modulo the 1-forms in $\CalJ$ and the 2-forms in $\varphi^*\CalI$.
	Under these conditions the map $\varphi$ is guaranteed to be a local surjection from the integral manifolds of $\CalJ$  to
	the integral elements of $\CalI$.
	Indeed, as argued in \cite{bryant-griffiths:1995b},
	let $P\subset M$ be an integral manifold of $\CalI$ defined in a neighborhood of $x\in M$. Choose a point $y \in N$ such that
	$\varphi(y)=x$.  Then  $\tilde P = \varphi^{-1}(P)$ is a submanifold on $N$ containing $y$  and the restriction of $\CalJ$ to $\tilde P$ is
	a completely integrable Pfaffian system $\tilde \CalJ$.  By the Frobenius theorem (applied to $\tilde \CalJ$ as a Pfaffian system on $\tilde P$)
	there is locally a unique integral manifold $Q$ for
	$\tilde \CalJ$ through $y$.  The manifold $Q$ is then an integral manifold of $\CalJ$ which projects by $\varphi$ to
	the original integral manifold $P$ on some open neighborhood of $x$.	
\begin{Corollary}
\StTag{SuperForm8}
	Let $\{\hV, \cV\}$ be a Darboux pair on $M$.
	Then the EDS\  $\CalW_1 + \CalW_2$ on  $M_1 \times M_2$ is an integrable extension of $\hCalV \mycap \cCalV$
	on  $M$ with respect to the superposition formula $\Sigma \colon M_1 \times M_2 \to M$.
\end{Corollary}
\begin{proof}
	In view of \EqRef{SuperForm52} and \EqRef{SuperForm55} we may take  the {\it differential ideal} generated by  $\CalF = \{\, \bfhtheta_1\}$
	as a complement to $\Sigma^*(\hCalV \mycap \cCalV)$ in $\CalW_1 + \CalW_2$.
	The 1-forms $\bfhtheta_1$ are closed modulo
	$\bfhtheta_1$, $\bfheta_1$  and  $\bfG_1 \bfcsigma_1 \wedge\bfcsigma_1$ and therefore
	the  generators for $\CalF$ are closed,  modulo $\CalF$ and  modulo
	the 1-forms and 2-forms in $\Sigma^*(\hCalV \mycap \cCalV)$.
	Hence $\CalW_1 + \CalW_2$ is an integral extension of $\hCalV \mycap \cCalV$.
\end{proof}

\begin{Corollary}
	Let $\CalI$ be a decomposable, Darboux integrable Pfaffian system. Then the map
	$\Sigma \: M_1 \times M_2 \to M$ is a superposition formula for $\CalI$ with respect to the EDS $\CalW_1 + \CalW_2$ which is
	locally surjective on integral manifolds.
\end{Corollary}
\begin{proof}
	If  $\{\, \hV, \cV \, \}$ is the Darboux pair for $\CalI$, then $\CalI =  \hCalV \mycap \cCalV$
	and the corollary follows from Theorem \StRef{SuperForm6} and Corollary \StRef{SuperForm8}.
\end{proof}

\subsection[Superposition Formulas and Quotients of EDS]{Superposition Formulas and Symmetry Reduction of Differential Systems}

	In this section we shall use the superposition formula established in Theorem \StRef{SuperForm6} to prove that if $\CalI$
	is any decomposable differential system for which the singular Pfaffian systems $\{\hV , \cV\}$ form a Darboux pair,
	then $\CalI$ can be realized as a quotient differential system, with respect to its Vessiot group,
	by the construction given in Corollary \StRef{SymRed2}.

	To precisely formulate this result, it is useful to first summarize the  essential results of  Sections 4 and 5.1--5.2.  We have constructed,
	through the coframe adaptation of Section 4, a local Lie group $G$ and local right and left group actions
	$\hmu , \cmu : G\times M \to M.$
	For the sake of simplicity, let us suppose that $G$ is a Lie group and that these actions are globally 	
	defined. The infinitesimal generators for these actions are the vector fields $\hX_i$ and $\cX_i$, 	
	defined by the duals of the 5-adapted coframes.

\begin{itemize}
\item[{\bf[i]}]
	The actions  $\hmu$ and $\cmu$ are free actions  with the same orbits (the vector fields $\hX_i$ and $\cX_i$
	are pointwise independent and related by $\hX_i = \lambda^j_i \cX_j$).
\item[{\bf[ii]}]	
	The actions $\hmu$ and $\cmu$ commute (see \EqRef{SuperForm120}).
\item[{\bf[iii]}]
	The actions $\hmu$ and $\cmu$ are symmetry groups of $\hV$ and $\cV$, respectively
	(see \EqRef{SuperForm101}).
\item[{\bf[iv]}]
	Each integral manifold of $\hV^{\infty}$ or $\cV^{\infty}$ is fixed by both actions $\hmu$ and $\cmu$
	(see \EqRef{SuperForm102}).
\item[{\bf[v]}]
	We have defined $\iota_1 \colon M_1 \to M$ and $\iota_2 \colon M_2 \to M$ to be fixed integral manifolds of $\hV^{\infty}$ and $\cV^{\infty}$
	and $W_1$ and $W_2$ to be the restrictions of $\hV$ and $\cV$ to these integral manifolds.
\item[{\bf[vi]}]
	Properties {\bf [i]} -- {\bf [iv]} imply that the actions $\hmu$ and $\cmu$ restrict to
	actions on  $\hmu_i$ and $\cmu_i$  on $M_i$. These restricted actions are  free,
	$\hmu_1$ and $\cmu_2$ are symmetries of $W_1$ and $W_2$ respectively. The actions $\hmu_1$ and $\cmu_2$  are transverse to
	 $W_1$ and $W_2$  (see \EqRef{SuperForm19}).
\end{itemize}
\par
\noindent
	The diagonal action $\delta \colon G \times (M_1\times M_2) \to M_1 \times M_2$
	is defined as the left action
\begin{equation}
	\delta(h, (x_1, x_2)) = \big( \hmu_1(h^{-1}, x_1), \,  \cmu_2(h , x_2) \big).
\EqTag{SuperForm46}
\end{equation}
	Note that the infinitesimal generators for $\delta$ are
\begin{equation}
	 Z_i = -\hX_{1i} + \cX_{2i}.
\EqTag{SuperForm70}
\end{equation}
	Granted that the action $\delta$ is regular, we then have that all the hypothesis of
	Corollary \StRef{SymRed2} are satisfied.  We can therefore construct the
	quotient manifold $\bfq \colon M_1 \times M_2 \to (M_1\times M_2)/G$, the
	quotient differential system  $\CalJ = (\CalW_1 + \CalW_2)/G$ and the Darboux pairs $\{\,\hU, \cU \}$ (see \EqRef{SymRed11}).
\par	
	We use the superposition formula  \EqRef{SuperForm26} to identify the manifold $M$ with $(M_1 \times M_2)/G$ and the original differential system
	$\CalI$ with the quotient system $\CalJ$.

\begin{Theorem}
\StTag{SuperForm22}
	Let $\CalI$ be a decomposable differential system on $M$ whose singular Pfaffian systems
	$\{\, \hV, \, \cV\,\}$ define a Darboux pair.
	Let $G$ be the Vessiot group for $\{\,\hV, \cV \,\}$ and define Pfaffian systems $W_1$ on $M_1$ and $W_2$ on $M_2$ as above.
	Then the manifold $M$ can be identified as the quotient of $M_1\times M_2$ by the diagonal action $\delta$ of the Vessiot group $G$,
	the superposition formula $\Sigma$ is the quotient map, and $\CalI = (\CalW_1 + \CalW_2)/G$.
\end{Theorem}
	
	All the manifolds, actions, and differential systems appearing in Theorem \StRef{SuperForm22}
	are presented in the following diagram:

\medskip

\begin{proof}[Proof of Theorem \StRef{SuperForm22}] The following elementary facts are  needed.
\par
\medskip
\noindent
{\bf [i]} The superposition map $\Sigma: M_1\times M_2 \to M$ is invariant with respect to the diagonal action $\delta$ of the
	Vessiot group $G$ on $M_1\times M_2$.
\par
\medskip
\noindent
{\bf [ii]} For each point $x = (x_1, x_2) \in M_1\times M_2$, $\ker \Sigma_*(x) = \ker \bfq_*(x)$.
	
\par
\medskip
\noindent
{\bf [iii]} $\Sigma^*(\hV) = [W_1 \oplus \Lambda^1(M_2)]_{\text{\bf sb}}$ and $\Sigma^*(\cV) = [\Lambda^1(M_1) \oplus W_2]_{\text{\bf sb}} $.
\par
\medskip
\noindent	
	Facts {\bf[i]}  and {\bf [ii]} show that $\Sigma\colon M_1 \times M_2 \to M $ can be identified with
	$\bfq\colon M_1 \times M_2 \to (M_1\times M_2)/G$. Fact  $\bf [iii]$ shows that $\hCalV = \hCalU$ and $\cCalV = \cCalU$ (in the notation of
	Corollary  \StRef{SymRed2}) so that $\CalI = \hCalV \mycap \cCalV =  \hCalU \mycap \cCalU = \CalJ$.
\par
	To prove {\bf [i]},  let $x_1 \in M_1$, $x_2\in M_2$ and $h \in G$. If $\Psi(x_1) = (s_1, g_1)$ and $\Psi(x_2) = (s_2, g_2)$ then
	it is a simple matter to check, using the $G$ bi-equivariance of $\rho$, that
	$ \hmu_1(h^{-1},x_1) = (s_1, g_1 \cdot h^{-1})$ and  $\cmu_2(h, x_2) = (s_2, h \cdot g_2)$.
	The invariance of $\Sigma$ then follows immediately from its definition \EqRef{SuperForm26}.
\par
	Lemma \StRef{SuperForm25} and \EqRef{SuperForm70} show that
\begin{equation}
 \ker \Sigma_* = \text{span}\{\partial_{\htheta_1} - \partial_{\ctheta_2}\} = \text{span}\{ -\hX_{1i} + \cX_{2i} \} = \text{span}\{\,Z_i\,\}
\EqTag{SuperForm71}
\end{equation}
	which proves {\bf[ii]}.
\par
Lemma \StRef{SuperForm25} and  equations
	\EqRef{SuperForm16} show that
\begin{equation}
	\Sigma^*(\hV) = \Sigma^*(\{\, \bfhtheta,\, \bfhsigma,\, \bfceta,\, \bfheta\,\}	
	= \{\, \bfhtheta_1 +\bfctheta_2,\, \bfhsigma_2,\, \bfheta_2,\, \bfceta_1 \, \}
\end{equation}
	while \EqRef{SuperForm43} and \EqRef{SuperForm51} give
\begin{equation}
	W_1 \oplus \Lambda^1(M_2) = \{\, \bfhtheta_1,\, \bfceta_1,\, \bfctheta_2,\, \bfheta_2,\, \bfhsigma_2 \,\}
\end{equation}
	in which case {\bf [iii]} follows immediately from \EqRef{SuperForm71}.
\par
	To make precise the identification of  $M$ with the quotient manifold $\barM =(M_1\times M_2)/G$,	
	define a map  $\Upsilon \colon \barM \to M$ as follows.
	For each point $\bar x \in \barM$,  pick a point $(x_1, x_2) \in M_1\times M_2$  such that
	$\bfq(x_1, x_2) = \bar x$ and let  $\Upsilon(\bar x) = \Sigma(x_1, x_2)$. The diagram
\begin{equation*}
\begindc{3}
	\obj(0,0){$M_1 \times M_2$}[M1M2]
	\obj(-12,-12){$\barM$}[barM]
	\obj(12,-12){$M$}[M]
	\mor{M1M2}{barM}{$\bfq$}[\atright,\solidarrow]	
	\mor{M1M2}{M}{$\Sigma$}
	\mor{barM}{M}{$\Upsilon$}[\atright, \solidarrow]	
\enddc
\end{equation*}
	commutes and,  on the domain of any (local) cross-section $\zeta$ of $\bfq$, one has
	$\Upsilon = \Sigma \circ \zeta$. This observation and facts {\bf [i]}
	and {\bf [ii]} then suffice to  show that $\Upsilon$ is a well-defined,  smooth diffeomorphism.
	Moreover, the same computations used to establish fact {\bf [i]} show that $\Sigma$, and hence $\Upsilon$, is $G$ equivariant
	with respect to the action $\cmu_1$ on $M_1\times M_2$ and $\cmu$
	on $M$ and also $G$ equivariant with respect to the action $\hmu_2$ on $M_1\times M_2$ and $\hmu$ on $M$.
\par
	Finally, we recall that
	$\hU$ and $\cU$ may be calculated  from the cross-section $\zeta$ and the $\delta$ semi-basis forms by
\begin{equation}
	\hU = \zeta^*([W_1 \oplus T^*(M_2)]_{\text{\bf sb}}) \quad \text{and}\quad \cU = \zeta^*(T^*(M_1) \oplus W_2]_{\text{\bf sb}})
\end{equation}
 	in which case  {\bf[iii]} implies that $\Upsilon^*(\hCalV) = \hCalU$ and $\Upsilon^*(\cCalV) = \cCalU$.
\end{proof}

\begin{Remark}
	Finally we remark that all the results of this section remain valid in so long as the singular systems for $\CalI$ have  algebraic generators
	
\begin{equation}
	\hCalV
	 = \{\, \bftheta,\ \bfheta,\ \bfceta,\  \bfhsigma,\ \bfP \, \bfcsigma \wedge  \bfcsigma \,\}
	\quad\text{and}\quad
	\cCalV
	  = \{\, \bftheta,\ \bfheta,\ \bfceta,\   \bfcsigma,\ \bfQ\, \bfhsigma \wedge \bfhsigma,\,\}.
\end{equation}
	where $\bfP \in \Inv(\cV)$ and $\bfQ \in \Inv(\hV)$.
	Such systems need not be Pfaffian.
\end{Remark}

\newpage
\section{Examples}
	In this section we illustrate our general theory  with a variety of examples.
        Examples \StRef{Ex1} and \StRef{Ex2} are taken from the classical literature and
	are simple enough that most of the computations can be explicitly given.
	We consider PDE where  the unknown functions take values in a group or in a
	non-commutative algebra in Examples \StRef{Ex3} and \StRef{Ex4}.
	In Example \StRef{Ex5} we present some novel examples of Darboux integrable systems constructed by
	the coupling of a nonlinear, Darboux integrable scalar equation to a linear or Moutard-type equation.
	A Toda lattice system and a wave map system are explicitly integrated in Examples \StRef{Ex6} and  \StRef{Ex7}.
	In Example \StRef{Ex8}, we  solve some non-linear, over-determined systems in 3
	independent variables.

\begin{Example}
\StTag{Ex1}
	 As our first example we shall find the closed-form, general solution to
\begin{equation}
\EqTag{ExIII1}
	u_{xy} = \frac{u_x  u_y}{u-x}.
\end{equation}
	This example is taken from  Goursat's well-known classification (Equation VI) of Darboux integrable equations \cite{goursat:1899a}
	and is simple enough that all the steps leading to the solution can be explicitly given.
	See also Vessiot \cite{vessiot:1942a} (pages 9--22) or Stomark \cite{stormark:2000a} (pages 350--356).

	The   canonical Pfaffian system for  \EqRef{ExIII1} is $I = \{\ \alpha^1, \alpha^2, \alpha^3 \ \}$, where
\begin{equation*}
\EqTag{ExIII2}
	\alpha^1 = du - p\,dx -q\,dy,
	\quad
	\alpha^2 = dp -r \,dx -vpq\, dy,
	\quad
	\alpha^3 = dq -vpq \, dx -t \, dy	
\end{equation*}
	and $v = 1/(u-x)$. The associated singular Pfaffian systems are
\begin{equation}
	\hV = \{\, \alpha^i,\  dx,\ dr - q(vr +v^2p)\, dy\, \}
	\quad\text{and}\quad
	\cV =  \{\, \alpha^i,\  dy,\ dt - tpv\, dy\,\}.
\end{equation}
	The  first integrals for $\hV$ and $\cV$ are
\begin{equation}
\EqTag{ExIII4}
	\hI^1= x,\quad \hI^2 = vp,\quad \hI^3 =  vr + v^2p, \quad\cI^1=y, \quad \cI^2 = \frac{t}{q} - y
\end{equation}
	and we easily calculate that  $\hV^{(\infty)} \cap \cV = \{\, \heta^1 \, \}$ and $\hV \cap \cV^{(\infty)} = \{\,0\,\}$,  where
\begin{equation}
	\heta^1 = d \hI^2 +((\hI^2)^2 - \hI^3)\,d\hI^1 = v \alpha^2   - v^2 p \alpha^1.
\end{equation}
\end{Example}
	A 1-adapted coframe \EqRef{First4} is therefore given by
\begin{equation}
	\theta^1 = \alpha^1,\
	\theta^2 = \alpha^3,\
	\hsigma^1 = d\hI^1, \
	\hsigma^2 = d\hI^3, \
	\heta^1, \
	\csigma^1= d\cI^1, \
	\csigma^2 = d \cI^2.
\end{equation}
	We relabel the coframes elements by $\hpi^1 = \sigma^1$,  $\hpi^2=\sigma^2$,
	$\hpi^3=\heta^1$,  $\cpi^1=\csigma^1$, $\cpi^2=\csigma^2$ and calculate
\begin{equation}
\EqTag{ExIII5}
\begin{aligned}
	d \hpi^3
&	= -2\,\hI^2\,\hpi^1 \wedge \hpi^3 + \hpi^1\wedge\hpi^2,
	\quad
\\
	d \theta^1
&	= (u-x) \, \hpi^1 \wedge \hpi^3 + \hI^2\,  \hpi^1 \wedge \theta^1 + \cpi^1 \wedge \theta^2,
\\
	d \theta^2
&	=
        \ q\, \hpi^1 \wedge \hpi^3 + q \,\cpi^1\wedge\cpi^2 + \hI^2\,\hpi^1\wedge \theta^2 +  \cI^2\, \cpi^1 \wedge \theta^2.
 \end{aligned}
\end{equation}
	This  coframe satisfies the structure equations \EqRef{Second11} and is therefore 2-adapted.

	The next step, described in Section 4.2,  is to eliminate the $\cpi ^\alpha \wedge \theta^i$ terms from \EqRef{ExIII5}.
	We calculate the distributions $\hU$ and $\cU$ (see \EqRef{Third12}) and their derived flags to be
\begin{alignat}{2}
	\kern -8 pt \hU
&	= \{\partial_{\displaystyle \hpi^1} + ((\hI^2)^2 - \cI^3)\partial_{\textstyle \hpi^3},\  \partial_{\displaystyle\hpi^3},\ \partial_{\displaystyle\hpi^2} \, \},&
	\cU &= \{\partial_{\displaystyle\cpi^1},\  \partial_{\displaystyle\cpi^2}\},
\\	
	\kern -8 pt\hU^{(\infty)}
&	=  \hU \cup \{ \, (u-x)\,\partial_{\displaystyle\theta^1} +q\partial_{\displaystyle\theta^2},\ \partial_{\displaystyle\theta^1} \},&
	\cU^{(\infty)}
	&= \cU \cup \{\, q\,\partial_{\displaystyle\theta^2},\  q\,\partial_{\displaystyle\theta^1}\}.
\notag
\end{alignat}
	and then, in accordance  with \EqRef{Third5}-\EqRef{Third22}, define
\begin{equation}
	X_1=  (u-x)\,\partial_{\displaystyle\theta^1} + q\partial_{\displaystyle\theta^2},
	\quad
	X_2 =  \partial_{\displaystyle\theta^1},
	\quad
	Y_1 = 	q\,\partial_{\displaystyle\theta^2},
	\quad
	Y_2 =  \partial_{\displaystyle\theta^1},	
\EqTag{ExIII25}
\end{equation}
	The coframes dual to the  vector fields $\{\, X_i, \hU, \cU\, \}$ and $\{\,Y_i,   \hU, \cU\, \}$ are the 3-adapted coframes
\begin{gather}
	\thetaX^1 = \frac{1}{q}\theta^2,\quad
	\thetaX^2 =  \theta^1 -\frac{u-x}{q} \theta^2,\quad
	\thetaY^1 =  \frac{1}{q}\theta^2, \quad
	\thetaY^2 =   \frac{1}{q}\theta^1 \quad\text{with}
\EqTag{ExIII8}
\\
\begin{aligned}
	d \thetaX^1
&	= \hpi^1\wedge \hpi^3 + \cpi^1\wedge \cpi^2 = d \thetaY^1,
\\
	d \thetaX^2
&	= -(u-x)\,\cpi^1 \wedge \cpi^2 + \thetaX^1 \wedge \thetaX^2  +\hpi^1 \wedge \thetaX^1 + \hI^2\,\hpi^1 \wedge \thetaX^2, 	
\\
	d \thetaY^2
&	= \frac{u-x}{q}\, \hpi^1 \wedge \hpi^3	-\thetaY^1 \wedge \thetaY^2 +\cpi^1\wedge\thetaY^1 - \cI^2\,\cpi^1 \wedge \thetaY^2.		
\EqTag{ExIII9}
\end{aligned}
\end{gather}
	The coframe \EqRef{ExIII8} is in fact 4-adapted and hence
	the Vessiot algebra for \EqRef{ExIII1} is a 2 dimensional non-abelian Lie algebra.
	
	 We may skip the adaptations given in Section 4.3 and move on to the final adaptations
	given in Section 4.4. The Vessiot algebra is 1-step solvable and the structure equations \EqRef{ExIII9} are precisely of the
	form \EqRef{Fifth18}. The change of coframe
	$\htheta^1= \thetaX^1+ \hI^2\pi^1$ transforms the structure equations \EqRef{ExIII9} to the form \EqRef{Fifth21}.
	The change of coframe $\htheta^2 = \thetaX^2 - x\htheta^1$
	leads to the $\hat 5$-adapted coframe $\{\htheta^1 ,\htheta^2\}$ with structure equations
\begin{equation}
	d \htheta^1 = \cpi^1 \wedge \cpi^2,
	\quad
	d \htheta^2 = -u\,\cpi^1 \wedge \cpi^2 +\htheta^1 \wedge \htheta^2.
\end{equation}
	Similarly, the $\check 5$-adapted coframe
	$\ctheta^1 = \thetaY^1 + \cI^2\cpi^1$,
	$\ctheta^2 = \thetaY^2 -y\ctheta^1$
	satisfies
\begin{equation}
	d\ctheta^1 = \hpi^1 \wedge \hpi^3,
	\quad
	d\ctheta^2 = \frac{u-x-yq}{q}\,\hpi^1 \wedge \hpi^3 - \ctheta^1 \wedge \ctheta^2.
\end{equation}

	Before continuing we remark that  the vector fields $X_1$, $X_2$, defined by \EqRef{ExIII25}, are given in terms
	of the dual vector fields $\hX_1$ and $\hX_2$, computed from the $\hat 5$-adapted  coframe, by
	$X_1 = \hX_1 - x\hX_2$ and $X_2 = \hX_2$.
	These vector field systems have the same orbits and structure equations but
	the actions are evidently different and it is the latter action
	that is needed to properly construct the superposition formula.

	The forms \EqRef{SuperForm20} are
\begin{equation}
\begin{alignedat}{2}
	\homega^1
&	= \htheta^1 +\cI^2\,\cpi^1 = \frac{dq}{q},
&\quad
	\homega^2
& = \htheta^2 - u\cI^2\,\cpi^1 =  du - qdy - u  \frac{dq}{q},
\\
	\comega^1
&	= \ctheta^1 + \hI^2\,\hpi^1= \frac{dq}{q},
&\quad
	\comega^2
&	= \ctheta^2 + (\frac{p}{q} +  yvp)\,\hpi^1 = \frac{du}{q} - dy - y \frac{dq}{q}.
\end{alignedat}
\end{equation}
	The Vessiot group for \EqRef{ExIII1}  is the matrix group
	$\begin{bmatrix} 1 & b \\ 0  & a  \end{bmatrix}$ with Maurer-Cartan forms
\begin{equation*}
	\omega^1_L = \frac{da}{a},\ \omega^1_L= db - \frac{b}{a} da,\
	\omega^1_R = \frac{da}{a}, \ \omega^2_R = \frac{db}{a}.
\end{equation*}
	The map $\rho\:M \to G$ defined by $a = q $ and $b = u - yq $ satisfies \EqRef{SuperForm4}.

	Finally, if  we introduce coordinates $y_1 = y$,  $u_1 = u$, $q_1 = q$, $t_1 = t$ on the $\hV^{(\infty)}$ integral manifold
	$M_1 = \{\, \hI^1=0,\ \hI^2 =0,\ \hI^3 =0\, \}$ and
	$x_2 = x$, $u_2= u$, $p_2=p$, $q_2 =u$, $r_2 = r$ on  the $\cV^{(\infty)}$ integral manifold
	$M_2 = \{\, \cI^1=0, \cI^2 =0 \, \}$,
	then the  superposition formula  \EqRef{SuperForm99} is
\begin{gather*}
	x = x_2,
	\quad
	\frac{p}{u-x} =\frac{p_2}{u_2 -x_2},
	\quad
	\frac{r}{u-x} + \frac{p}{(u-x)^2} =  \frac{r_2}{u_2-x_2} + \frac{p_2}{(u_2-x_2)^2}
\\
       y = y_1,
	\quad	
	\frac{t}{q} = \frac{t_1}{q_1},
	\quad
	q= q_1q_2,
	\quad
	u - yq = u_2 + (u_1 - y_1 q_1)q_2,
\end{gather*}			
	or, explicitly in terms of the original coordinates $\{x, y, u, p, q, r, t\}$ on M,
\begin{equation}
\begin{gathered}
	x= x_2, \quad y =y_1, \quad u =u_2 +q_2u_1, \quad p= (1 + \frac{u_1q_2}{u_2 - x_2})p_2, \quad q =q_1q_2,
\\
	r = (1 + \frac{u_1q_2}{u_2 - x_2})r_2+ \frac{u_1p_2q_2}{(u_2 -x_2)^2}, \quad t= t_1q_2.	
\end{gathered}
\EqTag{ExIII18}
\end{equation}

	It remains to find the integral manifolds for $\hW$ and $\cW$. Restricted to $M_1$,
	the Pfaffian system $\hV$ becomes $\hW = \{\, du_1 - q_1 \,dy_1, dq_1 -t_1 dy_1\,\}$  with integral manifolds
\begin{equation}
	y_1 = \beta,\ u_1 = f(\beta),\ q_1= f'(\beta),\  t_1 = f''(\beta).
\EqTag{ExIII15}
\end{equation}
	Restricted to $M_2$, the Pfaffian system $\cV$ becomes
	$\cW = \{\, du_2 - p_2 dx_2,\ dp_2-r_2dx_2,\ dq_2 - \dfrac{p_2q_2}{u_2 - x_2} dx_2  \}$.
	To find the integral manifolds of $\cW$ we calculate the second derived Pfaffian system to be
	$\cW^{(2)} = \{\, dq_2 -  \dfrac{q_2}{u_2 - x_2}\, du_2 \,\}$
	which leads to the equation  $q_2 \,du_2 - u_2 \,d q_2 + x_2\, d q_2 =0$ or
\begin{equation*}	
	d\,\left(\frac{u_2}{q_2}\right) -x_2\, d \, \left(\frac{1}{q_2}\right) =0
	\quad \text{or} \quad
	d \,\left( \frac{u_2 -x_2}{q_2} \right) +\frac{1}{q_2} \,dx_2 =0.
\end{equation*}
	The integral manifolds for $\cW$ are therefore given by
\begin{equation}
        x_2 = \alpha, \quad  u_2 = x_2 -g(\alpha)/ g'(\alpha), \quad  q_2 = -1/g'(\alpha),
\EqTag{ExIII16}
\end{equation}
	with $p_2$ and $r_2$ determined algebraically from the vanishing of the first and second forms in $\cW$.  The substitution of
	\EqRef{ExIII15} and \EqRef{ExIII16} into the superposition formula \EqRef{ExIII18}  leads to the closed form general solution
\begin{equation}
	u = \frac{-f(y) - g(x)}{g'(x)} + x
\end{equation}	
	for \EqRef{ExIII1}.
\newpage
\begin{Example}
\StTag{Ex2}
	In this example we shall construct the superposition formula for the  Pfaffian system
	$I = \{\alpha^1, \alpha^2, \alpha^3 \}$,  where
\begin{equation}
\begin{aligned}
	\alpha^1
&	=  du-p\,dx - q\, dy,
\quad
	\alpha^2
	= dp + \frac{1}{b^3}\,(\tan b\tau- b\tau)\,dx - s\,dy,
\\
\quad
	\alpha^3
&=	dq -s\,dx - b(b\tau +\cot b\tau)\,dy.
\EqTag{ExII1}
\end{aligned}
\end{equation}
	The coordinates for this example are $(x,y,u,p,q,s,\tau)$ and $b$ is a parameter.
	This example nicely illustrates the various coframe adaptations in Sections 4  and has a surprising	
	connection with some of Cartan's results in the celebrated 5 variables paper \cite{cartan:1910a}.
	The values $b=1$, $b= \sqrt{-1}$ (the Pffaffian system \EqRef{ExII1} remains real) and the
	limiting value $b=0$ give the three Pfaffian systems for the  equations
\begin{equation*}
	u_{xx} = f(u_{yy})
\end{equation*}
	which are Darboux integrable at the 2-jet level (\cite{anderson-juras:1997a}, pages 373-374)  and \cite{boer:1893a}, pages 400--411).
	The case $b=0$  is treated in Goursat(\cite{goursat:1897a} Vol 2, page 130-132.

	One easily calculates $\hV^{\infty} \cap \cV  =  \hV \cap \cV^{\infty} = \{\,0\,\}$
	and that the first integrals for $\hV$  and $\cV$ are
\begin{equation}
\begin{alignedat}{2}
	\hI{\,}^1 &= s + \tau,
	&\quad
	\hI{\,}^2 &= -(x +b^2y)\hI{\,}^1 + q+b^2p,\
\\
	\cI{\,}^1 &= s - \tau,
	&\quad
	\cI{\,}^2 &= -(x -b^2y)\cI{\,}^1 + q-b^2p.
\end{alignedat}
\EqTag{ExII2}
\end{equation}	
	We immediately arrive at the 2-adapted coframe
\begin{equation}
\begin{gathered}
	\hpi^1 = \extd \hI{\,}^1,\quad \hpi^2 = \extd \hI{\,}^2,\quad
	\cpi^1 = \extd \hI{\,}^1, \quad \cpi^2 = \extd \cI{\,}^2,\quad \theta^1= 2\alpha^1,
\\	
	\theta^2 =  -b \cot b\tau \, \alpha^2 + \frac{1}{b}\tan b\tau \, \alpha^3, \quad
	\theta^3 =  -b \cot b\tau\, \alpha^2  - \frac{1}{b}\tan b\tau \, \alpha^3,
\end{gathered}
\EqTag{EXII3}
\end{equation}
	with structure equations
\begin{equation}
\begin{aligned}
	\extd \theta^1
& 	= (x+b^2y)\, \hpi^1 \wedge \theta^2  + \hpi^2 \wedge \theta^2
	-(x- b^2y)\,\cpi^1\wedge \theta^3 -\cpi^2 \wedge \theta^3,
\\
	\extd \theta^2
&	 = \hpi^1\wedge \hpi^2 -b\cot 2b\tau\, \hpi^1\wedge\theta^2 + b\csc 2 b\tau \,\cpi^1 \wedge \theta^3,
\\
	\extd  \theta^3
& 	=\cpi^1\wedge \cpi^2 -b\csc 2 b\tau \,\hpi^1 \wedge \theta^2  + b\cot 2b\tau\, \cpi^1\wedge\theta^3.
\end{aligned}
\EqTag{ExII4}
\end{equation}
	
	To compute the 3-adapted coframe $\thetaX^i$, we
	simply calculate the derived flag for the 2 dimensional distribution
	$\hU = \{\, \partial_{\hpi^1}, \partial_{\hpi^2} \,\}$ (see  Section 4.2).
	From the  structure equations  \EqRef{ExII4} we find
\begin{equation}
\begin{aligned}
	\hU^{(1)}
&	= \hU \cup \{\, -\partial_{\theta^2} \,\} \quad\text{and}
\\
	\hU^{(2)}
&	= \hU \cup
 	\{\,  -\partial_{\theta^2}, (x+b^2y)\, \partial_{\theta^1} -b\cot 2b\tau \,\partial_{\theta^2}-b\csc 2b\tau\,
	\partial_{\theta^3}, \partial_{\theta^1}\,\}.
\EqTag{ExII7}
\end{aligned}
\end{equation}
	We take $\{\, X_1, X_2, X_3\, \}$ to be the last 3 vectors in $\hU^{(2)}$
	and calculate
\begin{equation}
	[\, X_1,X_2\,] = b^2X_3.
\end{equation}
	Therefore the Vessiot algebra for the Pfaffian system \EqRef{ExII1} is abelian if $b=0$ and
	nilpotent otherwise. The 3-adapted coframe $\bfthetaX$ (see Theorem \StRef{Third7}) is
\begin{equation*}
\begin{gathered}
	\thetaX^1
	= -\theta^2 + \cos 2b\tau \, \theta^3,
	\
	\thetaX^2
	= -\frac{1}{b}\sin 2b\tau\, \theta^3,
	\
	\thetaX^3
	=\theta^1 + \frac{1}{b}(x+b^2y)\sin 2b\tau \, \theta^3,
\end{gathered}		
\end{equation*}	
	and the  structure equations \EqRef{Third8} are
\begin{equation}
\begin{aligned}
	\extd \thetaX^1
&	= - \hpi^1 \wedge \hpi^2 + \cos 2b\tau\, \cpi^1\wedge \cpi^2 + b^2\hpi^1 \wedge \thetaX^2,
\\
	\extd \thetaX^2
&	= -\frac{1}{b} \sin 2b\tau\,\cpi^1 \wedge \cpi^2  -\hpi^1 \wedge \thetaX^1,	
\\
	\extd \thetaX^3
&	= \frac{1}{b}(x +b^2y)\,\sin 2b\tau \, \cpi^1 \wedge \cpi^2 -\hpi^2\wedge \thetaX^1 - b^2 \thetaX^1 \wedge \thetaX^2.
\end{aligned}
\EqTag{ExII6}		
\end{equation}
	This coframe is actually 4-adapted.

	We labeled the vectors in the derived flag \EqRef{ExII7} so that the
	last vector $X_3$ spans the derived algebra of the Vessiot algebra. By doing so we are assured that the
	structure equations \EqRef{ExII6} are of the precise form \EqRef{Fifth18a}-\EqRef{Fifth18b}. The analysis at
	this point refers back to  Case II and the  structure equations \EqRef{Fifth12},
	as applied to just the first two equations in  \EqRef{ExII6}.
	The matrices $M$ and $R$ (see equations \EqRef{Fifth55} and \EqRef{Fifth56}) are found to be
\begin{equation}
	M =
\begin{bmatrix}
	0 &  b^2\hpi^1
\\
	-\hpi^1 & 0
\end{bmatrix}
\quad\text{and}\quad
	R =
\begin{bmatrix}
	\cos b\hI{}\,^1 & -b\sin b\hI{\,}^1
\\
	\dfrac{1}{b} \sin b \hI{\,}^1  &  \cos b \hI{\,}^1
\end{bmatrix}.
\end{equation}
	We then compute the 2-forms $\chi^i$  and the 1-forms $\phi^i$ (see \EqRef{Fifth57}) to be
\begin{equation}
\begin{gathered}
	\chi^1	= -\cos b\hI{\,}^1 \,\hpi^1 \wedge \hpi^2,
	\quad
	\chi^2  = -\frac{1}{b} \sin b\hI{\,}^1 \, \hpi^1\wedge \hpi^2,
\\
	\phi^1  = -\hI{\,}^2 \cos b\hI{\,}^1\, \hpi^1,\quad \phi^2  = -\frac{1}{b} \hI{\,}^2\sin b\hI{\,}^1\, \hpi^1,
\end{gathered}
\end{equation}
	so that the forms \EqRef{Fifth20} are
\begin{equation}
\begin{aligned}
	\htheta^1_1
&	= \ \ \cos b\hI{\,}^1\,\thetaX^1 - b\sin b\hI{\,}^1\,\thetaX^2  - \  \hI^2 \cos b\hI{\,}^1 \, \hpi^1,
\\
	\htheta^2_1
&	= \frac{1}{b}{\sin b\hI{\,}^1}\, \thetaX^1 + \ \cos b\hI{\,}^1 \thetaX^2 -\frac{1}{b}\hI{\,}^2\sin b\hI{\,}^1\, \hpi^1.
\end{aligned}
\end{equation}
	The  structure equations \EqRef{ExII6} are now reduced  to the form \EqRef{Fifth21}.
	The final required frame change \EqRef{Fifth58} leads to the $\hat 5$-adapted coframe
\begin{equation}
	\htheta^1 = \htheta^1_1,
	\quad
	\htheta^2
	=\htheta^2_1,
	\quad
	\htheta^3
	= \theta^3_X + \hI{\,}^2\cos b\hI{\,}^1 \, \htheta^1_1  + b \hI{\,}^2\sin b \hI^1\, \theta^2_1  +\frac{1}{2} (\hI{\,}^2)^2\,\hpi^2
\end{equation}
	with structure equations
\begin{equation}
\begin{aligned}
	\extd \htheta^1
	&= \cos b\cI{\,}^1\, \cpi^1 \wedge \cpi^2,
	\quad
	\extd \htheta^2
	= \frac{1}{b} \sin b\cI{\,}^1 \,  \cpi^1 \wedge \cpi^2,
\\
	\extd \htheta^3
	&=  \frac{1}{b}\big( (x+ b^2y)\sin 2b\tau  + b\hI{\,}^2\cos 2b\tau \big) \,\cpi^1 \wedge \cpi^2      -b^2 \htheta^1 \wedge \htheta^2.
\end{aligned}
\end{equation}
	
	The  diffeomorphism $\Phi(x,y,u,p,q,s,\tau) = (x,-y,u,p,-q,-s,\tau)$
	is an involution for \EqRef{ExII1} in the sense of Remark \StRef{FirstRem2} and can therefore be used to find the $\check 5$ adapted coframe.
	
	The forms \EqRef{SuperForm20} are
\begin{equation}
	\homega^1 = \htheta^1 + \cI{\,}^2 \cos b\cI{\,}^1 \cpi^1,
	\quad
	\homega^2 = \htheta^2  + \cI{\,}^2 \frac{1}{b}\sin b\cI{\,}^1 \cpi^1,
	\quad
	\homega^3 = \htheta^3 + A\cpi^1,
\end{equation}
	where $A= \cI^2(\hI{\,}^2 \cos 2b\tau +\dfrac{1}{b}(x+b^2y) \sin 2b\tau  -\dfrac{1}{2}\cI{\,}^2)$.

	The Vessiot group for this example is the matrix group
\begin{equation}
\begin{bmatrix}
	1 & z^1  & z^3 + b^2z^1z^2/2 \\
	0 &  1   & b^2 z^2 \\
	0 &  0   & 1
\end{bmatrix}
\end{equation}
with multiplication
\begin{equation}
	z^1= z^1_1 +z^1_2,
	\quad
	z^2= z^2_1 +z^2_2,
	\quad
	z^3= z^3_1 +z^3_2 -\frac{\ b^2}{2}(z^1_1z^2_2 - z^2_1z^1_2)
\EqTag{ExII19}
\end{equation}
	and  left invariant forms
\begin{equation}
	\homega^1 = dz^1, \quad
	\homega^2 = dz^2, \quad
	\homega^3 = dz^3  + \frac{\ b^2}{2}(z^1 dz^2 - z^2 dz^1).
\EqTag{ExII15}
\end{equation}
	The map  $\rho$ (Theorem \StRef{SuperForm1}) is then found to be
\begin{align}
	z^1
&	=
	\frac{1}{b^2} \,(-2x\sin bs\sin b\tau    +2b^2y \cos b s  \cos b \tau - b\hI{\,}^2 \sin b\hI{\,}^1  +b\cI{\,}^2 \sin b \cI{\,}^1),
\notag
\\
	z^2
&	= \frac{1}{b^3} \,(\ 2bx\cos bs \sin b \tau + 2y b^2\sin bs \cos b\tau \, + \hI{\,}^2b\cos b\hI{\,}^1 -  \cI{\,}^2 b\cos \cI{\,}^1),
\EqTag{ExII18}
\\
	z^3
&	=2u  -\frac{1}{2 b^3}(\sin 2b\tau  - 2b\tau\cos 2 b \tau)(x^2 - b^4y^2) -2p x\cos^2 2b\tau
\notag
\\
	&\qquad -2 qy\sin^2 2b\tau  -\frac{1}{2b}\hI^2\cI^2\sin 2b \tau.
\notag
\end{align}

	Finally, the combination of equations \EqRef{ExII2}, \EqRef{ExII19} and  \EqRef{ExII18}  leads to the superposition formula
\begin{equation}
\begin{aligned}
	x
&	= \frac{1}{2}{(x_1 +x_2 +  b^2(y_2 -y_1))}
	- \frac {\ \sin b(\tau_1-\tau_2)}{\sin b(\tau_1+\tau_2)}\xi,
\\
	y
&	= \frac{1}{2b^2}(x_2 -x_1 +b^2(y_1+y_2)) - \frac{\cos b(\tau_1-\tau_2)}{b^2\cos b(\tau_1+\tau_2)}\xi,
\\
	u
& 	= u_1 + u_2
	+ 2\,\frac{p_1 \sin 2b\tau_2- p_2\sin 2b\tau_1)}{\sin 2b(\tau_1+\tau_2}\, \xi
\\
&	+\frac{1}{b^2} \left( \frac{2 \tau_1 \sin^2(2b\tau_2)}{\sin^2(2b(\tau_1+\tau_2))}+
	\frac{2 \tau_2 \sin^2(2b\tau_1)}{\sin^2(2b(\tau_1+\tau_2))}
	- \frac{\sin2b\tau_1\sin2b\tau_2}{b\sin 2b(\tau_1+\tau_2)} \right)\, \xi^2,
\\
	p
&	= p_1+p_2+ 2\,\frac{\tau_1\sin 2b\tau_2-\tau_2\sin2b\tau_1}{b^2\sin 2b(\tau_1+\tau_2)}\,\xi ,
\\
	q
&	= -b^2p_1+b^2p_2 - 2\frac{ \tau_1\sin 2b\tau_2 +\tau_2\sin 2b\tau_1 }{\sin 2b(\tau_1+\tau_2)}\,\xi,
\\
	s
&	= -\tau_1 +\tau_2, \quad \tau = \tau_1 + \tau_2, \quad\text{where}\quad \xi = \frac{1}{2}(x_2-x_1-b^2(y_1+y_2)).
\end{aligned}
\end{equation}

	The restriction of $\hV$ to  the manifold $M_1 = \{\, \hI^1 = \hI^2 = 0\,\}$ gives	
\begin{align}
	\hW  = \{\, &du_1 - p_1\,dx_1 +b^2\,dy_1,\  dp_1 -1/b^3(\tan b\tau_1 - b\tau_1)dx_1 + \tau_1 dy_1,
\\
	\,  &\tau_1\,dx_1-b(b\tau_1+\cot(b \tau_1))\,dy_1-b^2dp_1\, \}.
\EqTag{ExII25}
\end{align}
	This is a rank 3 Pfaffian system on a 5 manifold  whose derived flag has dimensions $[3,2,0]$. The equivalence problem for such systems
	was analyzed in detail by Cartan \cite{cartan:1910a}, where it is established that the  fundamental invariant for such systems
	is a certain rank 4 symmetric tensor $T$ in two variables.  For  $b= 0$, this  tensor vanishes while for $b\neq 0$
	we find that $T$ is the 4-th symmetric power of a 1-form. In accordance with  Cartan's result the symmetry algebra of $\hW$ when $b=0$
	is the  14 dimensional exceptional Lie algebra $g_2$ and, indeed, it is not difficult to  transform $\hW$ to the
	canonical Pfaffian system for the Hilbert-Cartan equation  $z' =(y'')^2$. For  $b\neq 0$ the symmetry algebra  of $\hW$
	is the 7 dimensional solvable Lie algebra with infinitesimal generators
\begin{equation}
\{\, \partial_{x_1}, \partial_{y_1},  \partial_{u_1},  x_1\partial_{x_1}  +y_1\partial_{y_1} + 2u_1\partial_{u_1} + p_1 \partial_{p_1},
	(x_1- b^2y_1) \partial_{u_1} - \partial_{p_1}, Y_1, Y_2 \,\},
\end{equation}
	where  $Y_2 = [\partial_{y_1}, Y_1]$ and
\begin{align}
	Y_1 = &
	(x_1 +b^2y_1)\big(b\cot b\tau_1 \partial_{x_1} - \frac{1}{b}\tan b\tau_1 \partial_{y_1}  +
	2p_1b \csc b\tau_1 \partial_{u_1}  + 2\tau_1\csc 2 b\tau_1 \partial_{p_1} \big)
\notag
\\
	&  + 2x_1y_1\partial_{u_1} + (y_1 -\frac{x_1}{b^2})\partial_{p_1} - \partial_{\tau_1}.
\end{align}
\end{Example}
	The Pfaffian \EqRef{ExII25} with $b\neq 0$ may be transformed into Cartan \cite{cartan:1910a}, page 170, equation (5').
\begin{Example}
\StTag{Ex3}
	For our next example, let $G$ be an $n$-parameter matrix group and,
	for the  mapping $(x,y)\to U(x,y) \in G$, consider the system of differential equations
\begin{equation}
	U_{xy} = U_x U^{-1} U_y.
\EqTag{ExI5}
\end{equation}
	The general solution to these equations  is well-known to be $U(x,y) = A(x)B(y)$,
	with $A(x), B(y) \in G$. In the case when $U$ is a $1 \times 1$ matrix,
	this system reduces to the wave equation $v_{xy}=0$ under the change of variable $u = \exp(v)$.
	We show how our integration method leads directly to
	the general solution and, in the process, we calculate the Vessiot algebra of \EqRef{ExI5}
	to be the Lie algebra of $G$.
	
	The Pfaffian system for $\EqRef{ExI5}$ is $I = \{\Theta, \Theta^1, \Theta^2\}$, where
\begin{equation}
\begin{aligned}	
	\Theta &= dU - U_x\, dx - U_y\, dy,
\\
	\Theta^1 &= dU_x - U_{xx}\, dx - U_x U^{-1} U_y\, dy,
\\
	\Theta^2 &= dU_y - U_x U^{-1} U_y \,dx - U_{yy}\, dy.
\end{aligned}
\end{equation}
	The  first integrals for the singular systems are
\begin{equation}
\begin{alignedat}{3}
	\hI^1=y,  \quad& \hI^2 &= U^{-1}U_y, \quad \hI^3 &= D_y(\hI^2) = U^{-1}U_{yy} - U^{-1}U_y U^{-1}U_y,
\\
	\cI^1=x,  \quad& \cI^2 &= U_xU^{-1}, \quad \cI^3 &= D_x(\cI^2) = U_{xx}U^{-1} - U_x U^{-1}U_xU^{-1},
\end{alignedat}
\end{equation}
	and  our 0-adapted coframe for $I$ is $\{\Theta, d\hI^1, d\hI^3, \heta,   d\cI^1, d\cI^3, \ceta \}$, where
\begin{align*}
	\heta
&	= d \hI^2 - \hI^3 d  \hI^1 = U^{-1}\Theta^2- U^{-1}\, \Theta \, \hI^2,
	\quad\text{and}\quad
\\
	\ceta
&	= d \cI^2 - \cI^3 d \cI^1 =  \Theta^1 U^{-1}- \cI^2\, \Theta \, U^{-1}.
\end{align*}
	The structure equations are
\begin{equation}
	\extd \Theta
	=  d \hI^1 \wedge (U\ \heta + \Theta\,\hI^2)  +  d\cI^1 \wedge (\ceta \ U +\cI^2\Theta).
\EqTag{ExI7}
\end{equation}
	This coframe satisfies \EqRef{Second11} and is therefore 2-adapted.
	
	The next step  is to eliminate
	either the   $ d\cI^1 \wedge (\hI^2\Theta)$ or the  $ d\hI^1 \wedge (\Theta \,  \cI^2)$ terms in \EqRef{ExI7}.
	By inspection, we see that the forms $\Theta_X = U^{-1} \Theta$ and $\Theta_Y = \Theta \,U^{-1}$ provide us with the
	required 4-adapted coframes, with structure equations
\begin{equation}
\begin{aligned}
	\extd \Theta_X
&	=   d\hI^1 \wedge \ceta  +  d \cI^1 \wedge ( U^{-1}\heta  \ U)
	\, - \Theta_X \wedge \Theta_X      + d \hI^1 \wedge (\Theta_X \hI^2-\hI^2\Theta_X),
\\
	\extd \Theta_Y
&	=  d \hI^1 \wedge (U \ \ceta\ U^{-1}) +   d \cI^1 \wedge \ceta +\Theta_Y \wedge \Theta_Y
	- d \cI^1 \wedge (\Theta_Y\cI^2 - \cI^2\Theta_Y) .
\end{aligned}
\end{equation}
	Since the  forms $\Theta_X$  and $\Theta_Y$  are Lie algebra valued, these structure equations show
	that the Vessiot algebra for \EqRef{ExI5} is the Lie algebra of $G$ (Theorem \StRef{Fourth5}).

	The final coframe adaptation  in Section 4.4 is given by
\begin{equation}
\begin{aligned}
	\accentset{\Largehat}{\Theta} &=  \Theta_X + \hI^2 d \hI^1  = U^{-1} dU  - U^{-1} U_x dx \quad\text{and}
\\
	\accentset{\Largecheck}{\Theta} &= \Theta_Y + \cI^2  d \cI^1 = dU\, U^{-1} -  U_y\, U^{-1} dy,
\end{aligned}
\end{equation}
	with structure equations
\begin{equation}
	\extd \accentset{\Largehat}{\Theta}
	=   d \cI^1 \wedge ( U^{-1}\ceta  \ U)  - \accentset{\Largehat}{\Theta} \wedge \accentset{\Largehat}{\Theta}
	\quad\text{and}\quad
	\extd \accentset{\Largecheck}{\Theta}
	=  d \hI^1 \wedge (U \ \heta\ U^{-1}) +  \accentset{\Largecheck}{\Theta}\wedge \accentset{\Largecheck}{\Theta}.
\end{equation}	
	The form \EqRef{SuperForm20} are then found to be precisely the left and right invariant forms on $G$, that is,
\begin{equation}
	\homega = U^{-1}\, dU 	
	\quad\text{and}\quad
	\comega = dU\, U^{-1},
\end{equation}
	so that the map $\rho$ constructed in Theorem \StRef{SuperForm1} is simply $\rho(x,y,U,\ldots) = U$. 	

	With respect to  coordinates $x, U_1, U_{1x}, U_{1xx}$  on the $M_1 = \{\,\hat I^a =0\,\}$ and
	coordinates $y, U_{2y}, U_{2yy}$ on the level set $M_2 = \{\, \cI^a= 0\,\}$, the Pfaffian systems $\hW$ and $\cW$ are
\begin{equation*}
	\hW = \{\, dU_1 - U_{1x} dx,\, dU_{1x} - U_{1xx} dx \}
	\quad\text{and}\quad
	\cW = \{\, dU_2 - U_{2y} dy,\, dU_{2y} - U_{2yy} dy \}
\end{equation*}
	and the superposition formula is
\begin{gather*}
	U = U_1 U_2, \ U_x =U_{1x} U_2,\  U_y = U_1U_{2y},\   U_{xx} = U_{1xx} U_2, \ U_{yy} = U_1U_{2yy}.	
\end{gather*}
\end{Example}
\begin{Example}
\StTag{Ex4}
	It is an open problem to determine which scalar Darboux integrable equations admit generalizations wherein
	the dependent variable $U$ takes values in an arbitrary  non-commutative, finite dimensional algebra $\CalA$.
	Here are two such examples which provide us with many Darboux integrable systems amenable to the methods presented
	in this paper.
\begin{multline*}
\shoveright{
\begin{alignedat}{2}
\qquad{\bf I.}
	\quad  U_{xy} &= (U_x  +I)\,U^{-1}\, U_y&
\\
{\bf II.}\quad  U_{xy} &= U_x \,(U-y)^{-1}\,U_y + U_y\,(U-x)^{-1}\,U_x.&
\end{alignedat}}
\end{multline*}
	The  first integrals for the singular systems  (excluding $x$ and $y$) and general solutions are
	\footnote{We  remark that the   solution to  {\bf II}  given  by Vessiot
	\cite{vessiot:1942a} (equations $C^2_{II}$ and (70), pages 6 and 45) in the scalar case  is  incorrect.}
	
\begin{multline*}
\shoveright{
\begin{aligned}
\qquad{\bf I.}
	\quad  \hI^1 &= (U_x +I)\,U^{-1}, \quad\hI^2 = D_x(\hI^1),\quad \cI^2 = U_y^{-1}\,U_{yy},
\\
	\quad U & = (F')^{-1}\,(-F+G),	
\\[2\jot]
{\bf II.}
	\quad  \hI^1& = (U-x)^{-1}\,[\,U_{xx}\, U_x^{-1}\,(U-x) - 2U_x + I\,],
\\
	\quad \cI^1 &= (U-y)^{-1}\,[\,U_{yy}\, U_y^{-1}\,(U-y) - 2U_y  +I\,] ,
\\
	\quad U & = (xF' +yG'-F- G)\,(F' +G')^{-1},
\end{aligned}}
\end{multline*}
	where $F = F(x)$ and $G = G(y)$ take values in $\CalA$.
	For both systems {\bf I} and {\bf II} the Vessiot algebra is the tensor product of $\CalA$
	with the Vessiot algebra for the corresponding scalar equation.
	We conjecture that all equations of Moutard type (\cite{goursat:1897a} Volume II, page 250, equation 19)
	admit non-commutative  generalizations.
\end{Example}

\begin{Example}
\StTag{Ex5}
	Some of the simplest examples of Darboux integrable systems can be constructed by the coupling of a
	Darboux integrable scalar equation to a linear or Moutard-type equation.  As examples, we give
\begin{multline*}
\shoveright{
\begin{alignedat}{2}
\qquad{\bf I.}
	\quad u_{xy} &= e^{2u}, &\quad v_{xy} &= n(n+1)e^{2 u} v,
\\
{\bf II.}
	\quad u_{xy} &= e^{u} u_y, &\quad v_{xy} & +  ((n- \alpha)e^{u}  + \alpha u_x)v_x = 0,
\\
{\bf III.}
	\quad u_{xy} &= e^{u} u_y, &\quad v_{xy}& -e^uv_y +(n+1)\,u_y v_x  +(n+1)!\,e^u u_y =0,
\\
{\bf IV.}
	\quad u_{xy} &= e^{u} u_x & \quad v_{xy} &  + (e^v)_x - (nB e^{-v})_y +(n-2)B = 0,
\end{alignedat}}
\end{multline*}
	where $B = e^uu_x$ and $n$ is a positive integer.
	The system {\bf I} appears in \cite{leznov-saveliev:1980a}(page 116); systems {\bf II}--{\bf IV} do not seem
	to have appeared in the literature.
	
	For each of these systems the restricted Pfaffian systems $\hW$ and $\cW$
	are jet spaces for two functions of a single variable ($x$ or $y$).
	The Vessiot algebra for {\bf I} is the semi-direct product of $\mathfrak{sl}(2)$ and an Abelian Lie algebra
	of dimension $2n+1$, as determined by the  (unique)
	$(2n+1)$-dimensional irreducible representation of $\mathfrak{sl}(2)$.
	The infinitesimal action of the Vessiot group  for {\bf I}, restricted to $\hW$ or $\cW$, is the action
	listed in \cite{gonzalez-kamran-olver:1992a}
	as number 27 (where now the variables $x$, $y$ in \cite{gonzalez-kamran-olver:1992a}
	serve as the dependent variables for the jet spaces   $\hW$ and $\cW$).

	For {\bf II} the Vessiot algebra is a
	semi-direct product of the 2-dimensional solvable algebra with an $(n+1)$-dimensional Abelian algebra.
	The infinitesimal action of the Vessiot group, restricted to $\hW$ or $\cW$,
        is number 24 in \cite{gonzalez-kamran-olver:1992a}.  The  infinitesimal Vessiot groups for {\bf III} and
	{\bf IV}  have dimensions $n+3$  and $n+4$ and coincide, respectively,  with
	numbers 25 and 26 in  \cite{gonzalez-kamran-olver:1992a}.

	For $n = 1$ the general solutions to these systems are
\begin{multline*}
\shoveright{
\begin{aligned}
\quad{\bf I.}
	\quad u &=  \frac{1}{2}\ln \frac{F_1'G_1'}{(F_1 +G_1)^2},
	\quad v  =  2\frac{F_2 - G_2}{F_1 +G_1} - \frac{F_2'}{F_1'} + \frac{G_2'}{G_1'}
\\
{\bf II.}
	\quad u &= \ln\dfrac{F'_1}{G_1-F_1},
	\quad v  = \dfrac{1}{F_1^\alpha}\, \big( F_2 -G_2 -  (F_1 -G_1)\dfrac{G'_2}{G'_2}\,\big),
\\
{\bf III.}\quad  u &= \ln\dfrac{F'_1}{G_1-F_1},\quad v = \frac{F_2 - G_2}{(F_1 -G_1)^2} -\frac{G_2'}{(F_1 -G_1)G_2'} - \ln(G_1 - F_1),
\\
{\bf IV.}
	\quad  u &= \ln\dfrac{G'_1}{F_1-G_1},
\\
	\quad  v  &= \ln\big(\frac{\ (G_1'\,(F_2 - G_2) - G_2'\,(F_1 - G_1)\,) F_1'\,\hfill }
	{(F_1'\,(F_2 -G_2) -F_2'\,(F_1 - G_1))\,(F_1 -G_1) }\big) .
\end{aligned}
}
\end{multline*}
	The general solutions for arbitrary $n$ can be obtained in closed compact form by the method of Laplace.

	In addition,  any non-linear Darboux integrable system can be coupled to its
	formal linearization to obtain another Darboux integrable system. For example,
	if we prolong the partial differential equations
\begin{multline*}
	\shoveright{\quad{\bf V.}\quad  3u_{xx}u_{yy}^3 +1 =0, \quad v_{xx} - \frac{1}{u_{yy}^4} v_{yy}= 0}
\end{multline*}
	to order 3 in the derivatives of $u$, we obtain a rank 8 Pfaffian system on a 14-dimensional manifold which is Darboux integrable,
	with 4 first integrals for each  associated singular Pfaffian system.
	
	The Vessiot algebra is Abelian and of dimension 6.  The two Lie algebras of vector fields dual to the forms $\htheta$ and $\ctheta$ for the
	5-adapted coframe coincide and are given by
\begin{equation*}
	\{\partial_y,\ \partial_u,\ x\partial_u +\partial_{u_x},\
	\partial_v,\ x\partial_{v} + \partial_{v_x},
	u_y\partial_v+ u_{xy}\partial_{v_x}+ u_{yy}\partial_{v_y}+u_{xyy}\partial_{v_{xy}}+u_{yyy}\partial_{v_{yy}}\}.
\end{equation*}
	In accordance with Remark \StRef{Fifth8}, these  vector fields are also infinitesimal symmetries for {\bf V}.
	In terms of the arbitrary functions $\phi(\alpha)$ and $\psi(\beta)$ appearing in the  general solution to $3u_{xx}u_{yy}^3 +1 =0$
	(Goursat(\cite{goursat:1897a} Vol. 2, page 130), the general solution for $v$ is given implicitly as
\begin{gather}
	x = \frac12\frac{\phi'' - \psi''}{\alpha- \beta}, \quad
	y = \frac12(\beta-\alpha)(\phi'' + \psi'') +\phi' - \psi', \quad
	t = \frac1{\alpha -\beta},
\\
	v = F  + G - \frac{\phi'' - \psi'' +(\beta -\alpha)\phi'''}{(\beta -\alpha)\phi''''} F'
         - \frac{\phi''- \psi'' +(\beta-\alpha)\psi'''}{(\beta -\alpha)\psi''''}G',
\end{gather}
	where $F= F(\alpha)$ and $G = G(\beta)$.

	Finally, we remark that the Laplace transformation (not the integral one, Darboux \cite{darboux:1896a} , Forsyth \cite{forsyth:1959a}, pages 45--59.)
	can be applied to the linear components of any of the above systems to obtain new Darboux integrability systems.
	As well, the linear component in any of these systems can be replaced by their formal adjoint to obtain yet other
	Darboux integrable systems.
\end{Example}

\begin{Example}
\StTag{Ex6}
	Although we are unaware of an explicit general proof
	it is  generally acknowledged that the Toda lattice systems (see, for example, \cite{leznov-saveliev:1980a}, \cite{sokoliv-ziber:1995a})
	are Darboux integrable.
	In this example we shall check that the $B_2$ Toda lattice equations
\begin{equation}
	u_{xy} = 2e^u - 2e^v, \quad  v_{xy} =-e^u + 2e^v
\EqTag{ExVI1}
\end{equation}
	are Darboux integrable and find the closed-form, general solution.
	We use this example to illustrate a slightly different computational approach, one based upon the
	symmetry reduction interpretation of the superposition formula given in Sections 3 and 5.3.

	The canonical Pfaffian system  for \EqRef{ExVI1} satisfies our definition of  Darboux integrable upon prolongation to 4-th order,
	that is, as a rank 14 Pfaffian system $I$ on a 20 dimensional manifold. The diffeomorphism $x \leftrightarrow y$ interchanges the
	singular Pfaffian systems $\hV$ and $\cV$.
	The first integrals for the singular Pfaffian system  $\hV$ (containing $dx$) are 	
\begin{align*}
	\hI^1 &= x,\
	\hI^2 =  v_{xx} + \frac{2}{3}u_{xx} - \frac{1}{3}u_xv_x - \frac{1}{6}u_x^2 - \frac{1}{3}v_x^2,\
	\hI^3 = D_x\hI^2,\ \hI^4 = D_x \hI^3,\ \text{and} 	
\\
	\hI^5 &=
	u_{xxxx}+2v_{xxxx} -2v_xv_{xxx}-u_{xx}u_x^2-2u_xu_{xx}v_x+ \frac{1}{8}u_x^4+\frac{1}{2}u_{xx}^2+
\notag
\\
	& \qquad \frac{1}{2}v_x^2u_x^2 - v_x^2u_{xx}-\frac{1}{2}v_{xx}u_x^2+v_{xx} u_{xx}+ \frac{1}{2}u_x^3v_x-u_xv_{xxx}.
\end{align*}
	Let $\hW$ be the restriction of $\hV$ (or, equivalently,  $I$) to
	$M_1 = \{\,\hI^a = 0\}$. We find that
	$\hW$ is a rank 12 Pfaffian system on a 15 manifold. The derived flag of $\hW$ has dimensions $[12, 10, 8, 6, 4, 2, 1, 0]$
        while the dimensions of the space of Cauchy characteristics for these derived Pfaffian sytems are [0, 2, 4, 6, 8, 10, 12, 15].
	By  using the invariants of these Cauchy characteristics as new coordinates, we are able to write $\hW$ in the canonical form
\begin{equation}
	\hW = \{\, du_2 - \dot u_2 du_1, \ d\dot u_2 - \ddot u_2 du_1,\ \dots, \ du_1 - u_1' dx,\  du_1' - du_1'' dx,\  \dots \,\} .
\end{equation}
       Here there are 7 contact forms for $u_2$ and 5 for $u_1$. In these coordinates the integral manifolds of $\hW$ are given
	by $u_1 = F_1(x)$, $u_1' = F_1'(x)$, \dots  and $u_2 = F_2(F_1(x))$,  $\dot u_2 = (\dot F_2)(F_1(x))$,\ \dots.

	The  10 dimensional infinitesimal  Vessiot  group, restricted to $M_1$, now takes a remarkably simple and
	well-known form  -- it is the infinitesimal conformal group $o(3,2)$ acting on the 3-dimensional space
	$(u_1, u_2, \dot u_2)$ by contact transforms.  (See, for example, Olver \cite{olver:1995a} page 473.)
	Explicitly, the generating functions for the infinitesimal action of the Vessiot group on $M_1$  are
\begin{align}
	Q &= [u_2, -u_2+u_1\dot u_2, -u_1\dot u_2^2+2u_2\dot u_2, \frac{1}{2}, \dot u_2, \frac{1}{2}\dot u_2^2, u_1,
\notag
\\
&\quad
	-\frac{1}{2}u_1^2\dot u_2^2-2u_2^2+2u_1u_2\dot u_2, -2u_1u_2+u_1^2\dot u_2, \frac{1}{2}u_1^2].
\end{align}
	To obtain the infinitesimal generator $\hX_q$ corresponding to a function $q\in Q$, first construct the vector field
	$X^0_q = -q_{\dot u_2} \partial_{u_1} + (q - \dot u_2 q_{\dot u_2}) \partial_{u_2} +(q_u + \dot u_2 q_{u_2}) \partial_{u_2}$
	and then prolong $X^0_q$ to the vector field $\hX_q$ on $M_1$ by requiring it to be a symmetry of $\hW$.
	We remark that the basis for $o(3,2)$ so obtained is the canonical Chevalley basis in the sense that the first 2 vectors
	define the Cartan subalgebra, the next 4 correspond to the positive roots, and the last 4 to the negative roots.

	By Theorem \StRef{SuperForm22}, the superposition formula for the $B_2$ Toda lattice can therefore
	be constructed from the joint invariants for the diagonal action of the conformal algebra $\mathfrak{o}(3,2)$ on $M_1 \times M_2$. We use coordinates
        $[y, v_1 , v_1', v_1'' \dots, v_2, \dot v_2, \ddot v_2, \dots ]$ on $M_2$. To compactly describe these joint invariants
	we first calculate the joint differential invariants in the variables
\begin{equation*}
	\{ u_1,\ u_1',\ v_1,\ v_1',\ \dot u_2,\ \dot v_2,\ \ddot u_2,\ \ddot v_2,\ \dddot u_2,\ \dddot v_2\, \}
\end{equation*}
	for the 7 dimensional subalgebra of $o(3,2)$ generated by $Q_1$, $Q_2$, $Q_4$, $Q_5$, $Q_6$, $Q_7$, $Q_{10}$. These are
\begin{align}
	J_1 &= v_1' (\dddot u_2)^{1/3}(\dddot v_2)^{2/3}/(\ddot u_2- \ddot v_2),
\quad 	J_2 =   u_1' (\dddot u_2)^{2/3}(\dddot v_2)^{1/3} /(\ddot u_2 -\ddot v_2),
\notag
\\
	J_3 &= (\ddot u_2 v_1 - \ddot u_2 u_1 + \dot u_2 - \dot v_2) (\dddot u_2)^{1/3}(\dddot v_2)^{2/3}/(\ddot v_2- \ddot u_2)^2,
\notag
\\
\quad   J_4 &= ( \ddot v_2 u_1 - \ddot v_2 v_1 +\dot v_2 -\dot u_2)(\dddot u_2)^{2/3}(\dddot v_2)^{1/3}/(\ddot v_2- \ddot u_2)^2, \quad\text{and}
\\
	J_5 &= -(\ddot u_2 \ddot v_2  v_1^2 - 2\ddot u_2 \ddot v_2v_1u_1+
	\ddot u_2 \ddot v_2 u_1^2- 2\ddot u_2 \dot u_2u_1 + 2\ddot u_2 \dot u_2v_1+2\ddot u_2 u_2
\notag
\\	
	 &\qquad-2\ddot u_2 v_2  -2\ddot v_2 \dot v_2 v_1+  2 \ddot v_2\dot v_2 u_1+2 \ddot v_2 v_2-2\ddot v_2 u_2
\\
	&\qquad +\dot u_2^2-2\dot u_2 \dot v_2+\dot v_2^2)\dddot u_2\dddot v_2/(2(\ddot v_2 -\ddot u_2)^4).
\notag
\end{align}
	Then, in terms of these partial invariants  the (lowest) order  joint differential invariants for $o(3,2)$ are
\begin{equation}
	K_1=   -\frac{J_1J_2(J_3J_4 -2J_5)^2}{(J_3J_4 -J_5)^2}
	\quad\text{and}\quad
	K_2 = -\frac{J_1J_2(J_3J_4-J_5)}{(J_3J_4- 2J_5)^2}
\end{equation}
and the solutions to the $B_2$ Toda lattice are
\begin{gather}
	u = \ln(K_1/4) \quad\text{and} \quad v = \ln(2K_2), \quad\text{where}
\\
	u_1= F_1(x),\ u_2= F_2(F_1(x)),\ v_1 = G_1(y),\ v_2 = G_2(G_1(y)).
\end{gather}
	It is hoped that a more transparent representation of these solutions, similar to that available for the $A_n$ Toda lattice will be
	be obtained. 	
\end{Example}

\begin{Example}
\StTag{Ex7}
	Let $P$ denote the 2-dimensional Minkowski plane with metric $dx \odot dy$ and
	let $N$ be a pseudo-Riemannian manifold with metric $g$.  A mapping  $\varphi: P \to N$ which is a
	solution to the Euler-Lagrange equations for the Lagrangian
\begin{equation}
	L = g(\frac{\partial \varphi}{\partial x}, \frac{\partial \varphi}{\partial y})\, dx \wedge dy
\end{equation}
	is  called a wave map.  There are precisely two inequivalent, non-flat  metrics
	(up to constant scaling)
	in 2 dimensions, namely
\begin{equation}
	g_1 = \frac{1}{1 +e^{-u} }(du^2 + dv^2) \quad\text{and}\quad
	g_2 =  \frac{1}{1 -e^{-u} }(du^2 + dv^2)
\end{equation}
	which define Darboux integrable wave maps at  the 2-jet level
	  (that is, without prolongation).
	Surprizingly, these metrics are not constant curvature. It is not difficult to
	check that  under the change of coordinates
\begin{equation}
	x = x-y,\quad
	t = x+y, \quad
	\theta =  \arctan(\sqrt{e^u -1}), \quad
	\chi = v/2
\end{equation}
	the  differential equations  (4.13) in \cite{barbashov-nesterenko-chervyakov:1982a}
	become the wave map equations for the metric $g_2$.  The Vessiot algebras for the wave map equations for
	$g_1$ and $g_2$ are $\mathfrak{sl}(2)\times R$  and $\mathfrak{so}(3)\times R$ respectively.

	The wave map equations for $g_1$ are
\begin{equation}
	u_{xy} = \frac{v_x v_y - u_x u_y}{2e^{u} +2},\quad v_{xy} = -\frac{u_xv_y + u_y v_x}{2e^{u} +2}.
\end{equation}
	The standard encoding of these equations as a Pfaffian system results in a rank 6 Pfaffian system
	on a 12 manifold. There are 4 first integrals for each singular  Pfaffian system -- for the singular Pfaffian system  $\hV$ containing $dx$
	the first integrals are $\hI^1 =x$,\ $\hI^2 = \dfrac{e^u(u_x^2 +v_x^2)}{1+e^u}$,  $\hI^3 = D_x\hI^2$, and
\begin{equation}
	\hI^4 = \frac{v_x}{u_x^2+v_x^2} u_{xx} -\frac{u_x}{u_x^2+v_x^2} v_{xx}
	-\frac{(1+2e^u)v_x}{2+2e^u}.
\end{equation}
	After considerable computation,
	the superposition is obtained and we find the general solution,
	in terms of the four arbitrary functions $F_1(x)$, $F_2(x)$, $G_1(y)$, $G_2(y)$ to be
\begin{equation}
	2e^u
	=-1 + \sqrt{1+A^2} \sqrt{1+B^2} +AB\sin(\Delta)
\end{equation}
and
\begin{align}
	&v = F_1(x)-F_2(x) +G_1(y)-G_2(y)
\\
&	 +\arctan\left(\frac{AF_2'\sqrt{1+A^2}}{A'}\right)
	+\arctan \left(\frac{BG_2'\sqrt{1+B^2}}{B'}\right)
\notag
\\
& 	+\arctan \left(\frac{AB'\cos(\Delta)+G_2' B^2\sqrt{1+A^2}\sqrt{1+B^2}+G_2'AB(1+B^2)\sin(\Delta)}
	{G_2'AB\sqrt{1+B^2} \cos(\Delta)-BB'\sqrt{1+A^2} -AB'\sqrt{1+B^2}\sin(\Delta)} \right)
\notag
\\
& 	+\arctan \left( \frac{A'B\cos(\Delta)-F_2' A^2\sqrt{1+A^2}\sqrt{1+B^2}-F_2'AB(1+A^2)\sin(\Delta)
	}{F_2'AB\sqrt{1+A^2} \cos(\Delta)+AA'\sqrt{1+B^2} +A'B\sqrt{1+A^2}\sin(\Delta)}
	\right),
\notag
\end{align}
where
\begin{equation}
	A(x) = \sqrt{ \left( \frac{F_1'}{F_2'} \right)^2-2 \frac{F_1'}{F_2'}},
\quad 	B(y) = \sqrt{ \left( \frac{G_1'}{K'} \right)^2-2 \frac{G_1'}{G_2'}},
\quad   \Delta = F_2 -G_2.
\end{equation}
	Note that
\begin{equation}
	F_1'(x) = F_2'(x)(1+ \sqrt{1+A(x)^2}), \quad G_1'(y) = G_2'(y) (1+\sqrt{ 1+B(y)^2}).
\end{equation}


\end{Example}

\newpage

\begin{Example}
\StTag{Ex8}
	We turn now to some simple examples of overdetermined systems for a single unknown function of 3 independent variables,
	beginning with the system
\begin{equation}
	u_{xz} = u u_x, \quad    u_{yz} =u u_y.
\EqTag{ExV1}
\end{equation}
	The structure equations for the canonical encoding of this system  as a rank 4 Pfaffian system
	$I= \{\alpha^1, \alpha^2, \alpha^3, \alpha^4\}$ on an 11-manifold
	are (modulo $I$), $d\alpha^1 \equiv 0$,
\begin{equation}
	d\alpha^2 \equiv \hpi^1 \wedge \hpi^3 + \pi^2\wedge \pi^4, \quad
	d\alpha^3 \equiv \hpi^1 \wedge \hpi^4 + \pi^2\wedge \pi^5, \quad
	d\alpha^4 \equiv  \cpi^1 \wedge \cpi^2,
\end{equation}
	where $\hpi^1= dx$, $\hpi^2 = dy$,  $\cpi^1 =dz$,
\begin{equation}
\begin{alignedat}{2}
	\hpi^3 &=  du_{xx} - (u_{xx}u+u_x^2)\, dz,
&\quad
	\hpi^4 &= du_{xy} - (u_{xy}u + u_y u_x)\, dz,\quad
\\
	\hpi^5 &= du_{yy} - (u_{yy}u+ u_y^2)\,dz,
&\quad
	\cpi^2 &= du_{zz} - (u_z + u^2)(u_x\,dx +u_y\,dy).
\end{alignedat}
\end{equation}
	The  first integrals for the singular Pfaffian systems $\hV = I \cup \{ \hpi^1, \ldots, \hpi^5\}$
	and $\cV = I \cup \{\cpi^1, \cpi^2\}$ are	$\hI^1 =x$,  $\hI^2= y$, $\cI^1= z$,
\begin{equation}
	\hI^3= \frac{u_y}{u_x},\quad \hI^4 = D_x \hI^3,\quad \hI^5= D_y \hI^3,\quad \cI^2 = u_z - u^2/2,\quad \cI^3 = D_z \cI^2.
\EqTag{ExV3}
\end{equation}	
	The form $\hpi^3$ is not in $\hV^{\infty} + \cV$ and therefore \EqRef{ExV1} is not Darboux integrable
	on the 2-jet.  The prolongation of \EqRef{ExV1} defines a decomposable rank 8 Pfaffian system $I^{[1]}= \{ \alpha^1, \ldots, \alpha^8\}$
	on a 16 dimensional manifold. In addition to the first integrals \EqRef{ExV3}, we now also have
\begin{equation}
	\hI^6 = D_x^2 \hI^3,\
	\hI^7 = D_{xy} \hI^3,\
	\hI^8 = D_y^2\hI^3,\
	\hI^9 = \frac{u_{xxx}}{u_x}  - \frac{3u_{xx}^2}{2 u_x^2},\
	\cI^4 = D_z^2\cI^2	
\EqTag{ExV4}
\end{equation}
	and  the conditions \EqRef{Intro7}--\EqRef{Intro9} for Darboux integrability of $I^{[1]}$ are  now satisfied.

	A 1-adapted coframe is
	$\hsigma^1 = d \hI^1$,\
	$\hsigma^2 = d \hI^2$,\
        $\hsigma^6 = d \hI^6$,\
        $\hsigma^7 = d \hI^7$,\	
        $\hsigma^8 = d \hI^8$,\
	$\hsigma^9 = d \hI^9$,\
        $\csigma^1 = d \cI^1$,\
	$\csigma^2 = d \cI^4$,\
\begin{align*}
	\heta^1 &= \frac{1}{u_x}\,\alpha^3 -\frac{u_y}{u_x^2}\,\alpha^2 = d\hI^3 - \hI^4d\hI^1- \hI^5d\hI^2,
\\
	\heta^2 &= \frac{1}{u_x}\, \alpha^6 -\frac{u_y}{u_x^2}\,\alpha^5 - \frac{u_{xx}}{u_x^2}\,\alpha^3
	- \frac{u_{xy}u_x - 2u_yu_{xx}}{u_x^3}\,\alpha^2 =  d\hI^4 - \hI^6d\hI^1- \hI^7d\hI^2,
\\
	\heta^3 &= \frac{1}{u_x}\,\alpha^7 -\frac{u_y}{u_x^2}\,\alpha^6 - \frac{u_{xy}}{u_x^2}\,\alpha^3 -
	\frac{-2u_yu_{xy} + u_{yy}u_x}{u_x^3}\alpha^2 = d \cI^5 -  \hI^7d\hI^1- \hI^8 d\hI^2,
\\
	\ceta^1	&= \alpha^4 - u\,\alpha^1 = d\cI^2 -  \cI^3 d\cI^1,
\quad
	\ceta^2 = \alpha^8 - u\,\alpha^4 -u_z\alpha^1 = d\cI^3- \cI^4 d \cI^1,
\\
	\theta^1 & = \alpha^1 = du - u_x\, dx - u_y\, dy -u_z\, dz,
\\
	\theta^2 &= \alpha^2 = du_x - u_{xx}\, dx - u_{xy}\, dy - u u_x\, dz,
 \\
	\theta^3 & = \alpha^3 = du_{xx} - u_{xxx}\, dx - u_{xxy}\, dy - (uu_{xx} +u_x^2)\,dz.
\end{align*}
	This frame is in fact 2-adapted. We next calculate  $\hU^{(\infty)} = \hU \cup \{X_1, X_2, X_3\}$  and
	$\cU^{(\infty)} = \cU \cup \{Y_1, Y_2, Y_3\}$, where
\begin{align}
	X_1 &=  -u_x\hI^4\partial_{\theta^1} -(u_{xx} \hI^4 +u_x\hI^6)\partial_{\theta^2} -\frac{3u_{xx}^2\hI^4 + 4u_{xx}u_x \hI^6 -
	2u_x^2\hI^4\hI^9}{2u_x}\partial_{\theta^3},
\notag
\\
	X_2 &=  -u_x\partial_{\theta^3},
\
	X_3 = 	-u_x\partial_{\theta^1} -u_{xx}\partial_{\theta^2} -\frac{3u_{xx}^2 + 2u_x^2\hI^9}{2u_x}\partial_{\theta^3},\
	Y_1 = - \partial_{\theta^1},
\\
	Y_2 &=  u\partial_{\theta^1}  +u_x\partial_{\theta^2} + u_{xx}\partial_{\theta^3},
\
	Y_3 =  (-\frac{u^2}{2} +\cI^2)\partial_{\theta^1} - u u_x\partial_{\theta^2} -(u_x^2 + uu_{xx})\partial_{\theta^3}.
\notag
\end{align}
	For the computations of Section 4.3  we use the  base point defined by setting
	$u=0$, $u_x=1$, $u_{xx} = 0$, $\hI^6=1$, and all other first integrals \EqRef{ExV3}--\EqRef{ExV4} to 0,
	The matrices \EqRef{Fourth21} and \EqRef{Fourth22} are  then given by
\begin{equation}
	P =
\begin{bmatrix}
	\dfrac{1}{\hI^6} &  \dfrac{2\hI^4\hI^9}{\hI^6} & -\dfrac{\hI^4}{\hI^6}
\\
        0 & 1 &0
\\ 	0  &-\hI^9  & 1
\end{bmatrix}
\quad\text{and}\quad	
	R = \begin{bmatrix} 0 & -1  & 0 \\ \cI^2 &0 &1 \\ 1 & 0 & 0  \end{bmatrix}.
\end{equation}	
	From these matrices we  calculate the 4 adapted coframes
\begin{alignat}{2}
	\thetaX^1 & = \frac{u_{xx}}{u_x^2}\theta^1 - \frac{1}{u_x}\theta^2,
&	\quad
	\thetaY^1 & = -\frac{u_x^2 + uu_{xx}}{u_x^3}\theta^2 + \frac{u}{u_x^2}\theta^3, \quad
\notag
\\
	\thetaX^2  &= -\frac{u_{xx}^2}{2u_x^3}\theta^1 +\frac{2 u_{xx}}{u_x^2}\theta^2 - \frac{1}{u_x}\theta^3,
&\quad
	\thetaY^2  &= \frac{u_{xx}}{u_x^3}\theta^2 - \frac{1}{u_x^2}\theta^3 , \quad
\EqTag{ExV12}
\\
	\thetaX^3  &= -\frac{1}{u_x} \theta^1,
&\quad
	\thetaY^3  &= -\theta^1 + \frac{u(2 u_x^2 + u u_{xx})}{2u_x^3}\theta^2 -\frac{u^2}{2u_x^2}\theta^3.
\notag
\end{alignat}

	The Vessiot algebra is $\mathfrak{sl}(2)$. Because this algebra is semi-simple,
	we can use Case {\bf I} of Section 4.4, Theorem \StRef{SuperForm1}, and  \EqRef{ExV12}
	to directly determine the map $\rho:M \to \text{Aut}(\mathfrak{sl}(2)$ as

\begin{equation}
	\rho (u, u_x, u_{xx}) = \lambda =
\begin{bmatrix}
	\dfrac{u_x^2-uu_{xx}}{u_x^2} &  -\dfrac{u}{u_x}  & -\dfrac{u_{xx}(uu_{xx}-2u_x^2)}{2u_x^3}
\\[2\jot]
	\dfrac{u_{xx}}{u_x^2}  & \dfrac{1}{u_x} & \dfrac{u_{xx}^2}{2u_x^3}
\\[2\jot]
	\dfrac{u(uu_{xx}-2u_x^2)}{2u_x^2}  & \dfrac{u^2}{2u_x}  & \dfrac{(u u_{xx} - 2u_x)^2}{4u_{x}^3}
\end{bmatrix}.
\end{equation}

	To obtain the superposition formula we introduce local coordinates  $(z$, $w$, $w_x$, $w_{xx}$, $w_z$, $w_{zz}$, $w_{zzz})$ for $M_1$ and
	$(x$, $y$, $v$, $v_x$, $v_y$, $v_z$, $v_{xx}$, $v_{xy}$, $v_{yy}$, $v_{yz}$, $v_{xxx}$, $v_{xxy}$, $v_{xyy}$, $v_{yyy}$, $v_{yyz})$ for $M_2$.
	The inclusions
	$\iota_1\: M_1 \to M$ and $\iota_2\: M_2  \to M$ are fixed by
\begin{equation*}
	\iota_1(w_I) = u_I, \quad
	\iota^*_1(\hI^a) = 0, \quad
	\iota_2(v_I) = u_I,  \quad
	\iota^*_2(\cI^a) = 0.
\end{equation*}
	The superposition formula is then found by solving the equations
\begin{equation}
	\iota^*_1(\cI^a) = \cI^a, \quad
	\iota^*_2(\hI^a) = \hI^a, \quad
	\rho(u, u_x, u_{xx}) = \rho(w, w_x, w_{xx})  \cdot \rho(v, v_x, v_{xx})
\end{equation}
	for the coordinates of $M$.  We find that
\begin{equation}
	u = w-  \frac{2 v w_x^2}{-2w_x + v w_{xx}} .
\EqTag{ExV17}
\end{equation}
	
        Finally,  we calculate the integral manifolds for $\hW$ and $\cW$. It is immediate that $\cW$ is the
	canonical Pfaffian system on  $J^3(\Real, \Real^2)$ and hence the integral manifolds are defined by  $v = V(x,y)$.
	The last non-zero form in the derived flag for
\begin{align*}
	\hW &= \{\,dw -w_z dz,\ dw_z -w_{zz}\,dz,\ dw_{zz} -w_{zzz}\,dz,\  dw_x - w_x w dz,
\\
	&\quad dw_{xx} + (-w_x^2-w w_{xx}) dz \, \}
\end{align*}
yields the Pfaffian equation
\begin{equation*}
	d w_{xx} - \frac{w_{xx}}{w_x}dw_x - w_x^2  dz = 0
\end{equation*}
	from which it follows that  $w_x = G'(z)$  and  $w_{xx} = G(z)G'(z)$. One of the remaining equations in $\hW$ then gives $w =  G''(z)/ G'(z)$.
	On  replacing $V(x,y)$ by $-2/F(x,y)$,  the superposition formula \EqRef{ExV17} becomes
\begin{equation}
	u(x, y, z) = \frac{G''(z)}{G'(z)} - \frac{2 G'(z)}{F(x,y) +G(z)},
\end{equation}
	which gives the general solution to \EqRef{ExV1}.

        We continue this example by considering two variations on  \EqRef{ExV1}.
	First we observe that to \EqRef{ExV1} we may add any  equation of the form
\begin{equation}
	F(x,y, \frac{u_y}{u_x}, \frac{u_yu_{xx} - u_x u_{xy}}{u_x}, \frac{u_yu_{xy} - u_xu_{tyy}}{u_x^2} ) = 0
\EqTag{ExV15}
\end{equation}	
	to obtain a rank 4 Pfaffian system  $I= \{\alpha^1, \alpha^2, \alpha^3, \alpha^4\}$ on a 10 dimensional manifold. The
	structure equations become ($\mod I$) $d\alpha^1 \equiv 0$,
\begin{equation}
	d\alpha^2 \equiv \hpi^1 \wedge \hpi^2,  \quad
	d\alpha^3 \equiv \hpi^3 \wedge \hpi^4,  \quad
	d\alpha^4 \equiv  \cpi^1 \wedge \cpi^2
\end{equation}
	and hence all such systems are involutive with Cartan character $s_1 =3$.

	For example, consider the system of  3 equations
\begin{equation}
        u_yu_{xy} - u_x u_{yy}= 0, \quad u_{xz} = u u_x, \quad    u_{yz} =u u_y
\EqTag{ExV16}
\end{equation}
	The foregoing calculations can be repeated, almost without modification, to arrive at the
	same superposition formula \EqRef{ExV17} -- the only difference
	is that now $\cW$ is the (prolonged) canonical  Pfaffian system for the equation   $v_yv_{xy} - v_x v_{yy}= 0$, an equation which
	is itself  Darboux integrable. Thus, in more complicated situations, the method of Darboux, can be used to integrate the Pfaffian
	systems $\hW$ and $\cW$ and the superposition formula for the original system is  given by a composition of superposition formulas.
	In the case of the present example, the calculation of the first integrals for $\cW$  reveals that this system is contact equivalent
	to the wave equation (via  $X = x$, $Y= v$, $V= y$, $V_X = - v_x/v_y$, $V_Y = 1/v_y$) and leads to  the parametric solution
\begin{equation}
	x= \sigma,\quad y = f(\sigma) + g(\tau),\quad u =  \frac{G''(z)}{G'(z)} - \frac{2 G'(z)}{\tau +G(z)}.
\end{equation}

	Our second variation of \EqRef{ExV1} is obtained by the differential substitution $u_x = \exp(v)$. This leads to the equations
\begin{equation}
	v_{xz} = \exp(v)  \quad\text{and}\quad
	v_{yzz} = v_{yz} v_z.
\EqTag{ExV18}
\end{equation}	
	It is surprising that the canonical Pfaffian system  for these equations (obtained by the restriction of the contact ideal on $J^3(\Real, \Real^2)$)
	does {\it not} define a decomposable Pfaffian system. The following theorem resolves this difficulty.
\begin{Theorem}
\StTag{Examples10}
	The system of differential equations
\begin{equation}
\begin{aligned}
	u_{xz}  & = F(x,y,z,u, u_z, u_{yz}, u_{zz}),
\\
        u_{yzz} & = G(x,y,z,u, u_x, u_y, u_z, u_{yz}, u_{zz}, u_{yyz}, u_{zzz})
\end{aligned}
\EqTag{ExV25}
\end{equation}
	determines the  rank 4 Pfaffian system
\begin{equation}
\begin{aligned}
	\alpha^1 & = du - u_xdx -u_y dy -u_zdz,
	\quad
	\alpha^2  = du_z - Fdx - u_{yz}dy -u_{zz}dz,
\\
	\alpha^3  & = du_{yz} - D_y(F)dx - u_{yyz} dy - G dz,
\\
	\alpha^4  & = du_{zz} -  D_z(F)dx - G dy -u_{zzz}dz,
\end{aligned}
\end{equation}
	on an 11 dimensional manifold.  If the compatiblity  conditions for \EqRef{ExV25} hold, then
	this Pfaffian system is decomposable and involutive with Cartan characters $s_1=1$ and $s_2 =1$.
\end{Theorem}

	We use  Theorem \StRef{Examples10}  to write \EqRef{ExV18} as a rank 4 Pfaffian system on an 11-manifold. The
	prolongation of this system is  Darboux integrable and calculations, virtually identical to
	those provided 	for \EqRef{ExV1}, lead directly to the general solution
\begin{equation}
	v(x, y, z) = \ln\bigl( \frac{2 F_x(x,y) G'(z)}{(F(x,y) + G(z))^2}\bigr).
\end{equation}
\end{Example}


\begin{bibdiv}
\begin{biblist}


\bib{anderson-fels:2005a}{article}{
  author={Anderson, I. M.},
  author={Fels, M. E.},
  title={Exterior Differential Systems with Symmetry},
  journal={Acta. Appl. Math.},
  year={2005},
  volume={87},
  pages={3--31},
}

\bib{anderson-juras:1997a}{article}{
  author={Anderson, I. M.},
  author={Jur\'a\v s, M.},
  title={Generalized Laplace Invariants and the Method of Darboux},
  year={1997},
  volume={89},
  journal={Duke J. Math},
  pages={351--375},
}

\bib{barbashov-nesterenko-chervyakov:1982a}{article}{
  author={B. M. Barbashov},
  author={V. V. Nesterenko},
  author={A. M. Chervyakov},
  title={General solutions of nonlinear equations in the geometric theory of the relativistic string},
  journal={Commun. Math Physics},
  volume={84},
  year={1982},
  pages={471--481},
}

\bib{boer:1893a}{article}{
  author={F. De Boer},
  title={Application de la m\'ethode de Darboux \`a l'int\'egration de l'\'equation diff\'erentielle s = f(r,t). },
  journal={Archives Neerlandaises},
  volume={27},
  year={1893},
  pages={355--412},
}

\bib{bryant-griffiths:1995b}{article}{
  author={Bryant, R. L},
  author={Griffiths, P. A.},
  title={Characteristic cohomology of differential systems (I)},
  journal={Selecta Math. (N.S.)},
  volume={1},
  year={1995},
  pages={21--112},
}

\bib{bryant-griffiths-hsu:1995a}{article}{
  author={Bryant, R. L.},
  author={Griffiths, P. A.},
  author={Hsu, L.},
  title={Hyperbolic exterior differential systems and their conservation laws, Parts {I} and {I}{I}},
  journal={Selecta Math., New series},
  year={1995},
  volume={1},
  pages={21--122 and 265--323},
}

\bib{cartan:1910a}{article}{
  author={Cartan, \'E},
  title={Les syst\`emes de {P}faff \`a cinq variables et les \'equations aux d\'eriv\'ees partielles du second ordre},
  year={1910},
  volume={3},
  number={27},
  journal={Ann. Sci. \'Ecole Norm.},
  pages={109--192},
}

\bib{cartan:1911a}{article}{
  author={Cartan, \'E},
  title={Sur les syst\`emes en involution d'\'equations aux d\'eriv'ees partielles du second ordre \`a une fonction inconnue de trois variable ind\`ependantes},
  year={1911},
  volume={39},
  journal={Bull Soc. Math France},
  pages={352-443},
}

\bib{darboux:1896a}{book}{
  author={Darboux, G.},
  title={Le\c cons sur la th\'eorie g\'en\'erale des surfaces et les applications g\'eom\'etriques du calcul infinit\'esimal},
  publisher={Gauthier-Villars},
  address={Paris},
  year={1896},
}

\bib{eendebak:2007a}{thesis}{
  author={P. T. Eendebak},
  title={Contact Structures of Partial Differential Equations},
  publisher={Utrecht University},
  year={2007},
}

\bib{flanders:1963a}{book}{
  author={H. Flanders},
  title={Differential Forms with Applications to the Physical Sciences},
  publisher={Dover},
  address={New York},
  year={1963},
}

\bib{fels-olver:1998a}{article}{
  author={Fels, M. E.},
  author={Olver, P. J.},
  title={Moving coframes I. A practical algorithm},
  journal={Acta. Appl. Math.},
  year={1998},
  volume={51},
  pages={161-213},
}

\bib{forsyth:1959a}{book}{
  author={A. Forsyth},
  title={Theory of Differential Equations, Vol 6},
  publisher={Dover Press},
  address={New York},
  year={1959},
}

\bib{gonzalez-kamran-olver:1992a}{article}{
  author={Gonz\'alez-L\'opez, A.},
  author={Kamran, N.},
  author={Olver, P. J.},
  title={Lie algebras of vector fields in the real plane},
  journal={Proc. London Math. Soc.},
  volume={64},
  year={1992},
  pages={339-368},
}

\bib{goursat:1897a}{book}{
  author={Goursat, E.},
  title={Lecon sur l'int\'egration des \'equations aux d\'eri\'ees partielles du second ordre \'a deux variables ind\'ependantes, Tome 1, Tome 2},
  publisher={Hermann},
  address={Paris},
  year={1897},
}

\bib{goursat:1899a}{article}{
  author={E. Goursat, E.},
  title={Recherches sur quelques \'equations aux d\'eri\'ees partielles du second ordre},
  journal={Ann. Fac. Sci. Toulouse},
  volume={1},
  year={1899},
  pages={31--78 and 439--464.},
}

\bib{griffiths:1974a}{article}{
  author={Griffiths, P.},
  title={On Cartan's method of Lie groups and moving frames as applied to uniqueness and existence questions in differential geometry},
  journal={Duke Math. J.},
  year={1974},
  volume={41},
  pages={775--814},
}

\bib{kobayashi-nomizu:1963a}{book}{
  author={S.S. Kobyashi and K. Nomizu},
  title={Foundations of Differential Geometry},
  publisher={John Wiley},
  year={1963},
}

\bib{leznov-saveliev:1980a}{article}{
  author={Leznov, A. N.},
  author={Saleliev, M. V.},
  title={Representation theory and integration of nonlinear spherically symmetric equations to gauge theories},
  year={1980},
  volume={74},
  journal={Commun. Math. Phys.},
  pages={111-118},
}

\bib{leznov-saveliev:1983a}{article}{
  author={Leznov, A. N.},
  author={Saleliev, M. V.},
  title={Two-dimensional exactly and completely integrable dynamical systems},
  year={1983},
  volume={89},
  journal={Commun. Math. Phys.},
  pages={59--75},
}

\bib{olver:1995a}{article}{
  author={Olver, P. J.},
  title={Equivalence, Invariants, and Symmetry},
  year={1995},
  publisher={Cambridge},
  address={University Press},
}

\bib{sokoliv-ziber:1995a}{article}{
  author={Sokolov, V. V. },
  author={Ziber, A. V.},
  title={On the {D}arboux integrable hyperbolic equations},
  journal={Phys Lett. A},
  volume={208},
  pages={303--308},
}

\bib{sternburg:1984a}{book}{
  author={Sternberg, S},
  title={Lectures on Differential Geometry},
  edition={Second},
  year={1984},
  publisher={Chelsa},
  address={New York},
}

\bib{stormark:2000a}{article}{
  author={Stormark, O.},
  title={Lie's structural approach to PDE systems},
  series={Encyclopedia of Mathematics and its Applications},
  volume={80},
  publisher={Cambridge Univ. Press},
  address={Cambridge, UK},
  year={2000},
}

\bib{varadarajan:1984a}{book}{
  author={Varadarajan, V. S.},
  title={Lie Groups, Lie Algebras and Their Representations},
  publisher={Springer-Verlag},
  address={New York},
  year={1984},
}

\bib{vassiliou:2001a}{article}{
  author={Vassiliou, P. J.},
  title={{V}essiot structure for manifolds of $(p,q)$-hyperbolic type: {D}arboux integrability and symmetry},
  journal={Trans. Amer. Math. Soc.},
  volume={353},
  year={2001},
  pages={1705--1739},
}

\bib{vessiot:1939a}{article}{
  author={Vessiot, E.},
  title={Sur les \'equations aux d\'eriv\'ees partielles du second ordre, F(x,y,z,p,q,r,s,t)=0, int\'egrables par la m\'ethode de {D}arboux},
  year={1939},
  volume={18},
  journal={J. Math. Pure Appl.},
  pages={1--61},
}

\bib{vessiot:1942a}{article}{
  author={Vessiot, E.},
  title={Sur les \'equations aux d\'eriv\'ees partielles du second ordre, F(x,y,z,p,q,r,s,t)=0, int\'egrables par la m\'ethode de {D}arboux},
  year={1942},
  volume={21},
  journal={J. Math. Pure Appl.},
  pages={1--66},
}










\end{biblist}
\end{bibdiv}
\end{document}